\documentclass[10pt,oneside,reqno]{amsart}

\usepackage[cal=cm,scr=boondoxo]{mathalfa}

\usepackage[utf8]{inputenc}
\usepackage[T1]{fontenc}
\usepackage[russian,ngerman,english]{babel}
\usepackage{lmodern}

\usepackage{amsmath}
\usepackage{amsthm}
\usepackage{amsfonts}
\usepackage{amssymb}
\usepackage{amscd}
\usepackage{amsbsy}
\usepackage{appendix}
\usepackage{bibentry}
\usepackage{bm}
\usepackage{dsfont}
\usepackage{enumerate}
\usepackage{epsfig}
\usepackage{graphicx}
\usepackage[colorlinks=true,allcolors=purple]{hyperref}
\usepackage{cleveref}		
	\crefformat{equation}{#2(#1)#3}
\usepackage{mathabx}
\usepackage{mathtools}
\usepackage[sort,numbers]{natbib}
\usepackage{fancyhdr}

\usepackage{pdfpages}
\usepackage{pifont}
\usepackage{setspace}
\usepackage{tabularx}
\usepackage{tikz}
	\usetikzlibrary{arrows}
	\usetikzlibrary{patterns}
	\usetikzlibrary{cd}
\usepackage{todonotes}
\usepackage{xfrac}

\theoremstyle{plain}
\newtheorem{theorem}{Theorem}[section]
\newtheorem{lemma}[theorem]{Lemma}
\newtheorem{proposition}[theorem]{Proposition}
\newtheorem{corollary}[theorem]{Corollary}

\theoremstyle{definition}
\newtheorem{definition}[theorem]{Definition}

\theoremstyle{remark}
\newtheorem{remark}[theorem]{Remark}
\newtheorem{conjecture}[theorem]{Conjecture}
\newtheorem{example}[theorem]{Example}

\numberwithin{equation}{section}

\newcounter{StepsCount}
\newenvironment{Elist}{%
	\begin{list}{\ding{\value{StepsCount}}}{\usecounter{StepsCount} \leftmargin=0pt \labelwidth=12pt \itemindent=\labelwidth%
	\itemsep=5pt\listparindent=\parindent} \setcounter{StepsCount}{171}}{\end{list}}

\newcommand{\cA}{{\mathcal A}}
\newcommand{\cB}{{\mathcal B}}
\newcommand{\cC}{{\mathcal C}}

\newcommand{\cH}{{\mathcal H}}
\newcommand{\cI}{{\mathcal I}}

\newcommand{\cL}{{\mathcal L}}

\newcommand{\cN}{{\mathcal N}}

\newcommand{\cP}{{\mathcal P}}

\newcommand{\cT}{{\mathcal T}}

\newcommand{\cW}{{\mathcal W}}

\newcommand{\scrC}{{\mathscr C}}

\newcommand{\scrg}{{\mathscr g}}
\newcommand{\scrh}{{\mathscr h}}

\newcommand{\bbD}{{\mathbb D}}
\newcommand{\bbE}{{\mathbb E}}

\newcommand{\bbH}{{\mathbb H}}

\newcommand{\bbN}{{\mathbb N}}

\newcommand{\bbR}{{\mathbb R}}
\newcommand{\bbC}{{\mathbb C}}
\newcommand{\bbM}{{\mathbb M}}

\newcommand{\bbZ}{{\mathbb Z}}

\renewcommand{\a}{\alpha}
\renewcommand{\b}{\beta}
\newcommand{\g}{\gamma}
\renewcommand{\d}{\delta}
\newcommand{\e}{{\varepsilon}}
\renewcommand{\k}{\kappa}

\DeclareMathOperator{\clos}{clos}
\DeclareMathOperator{\tr}{tr}

\DeclareMathOperator{\supp}{supp}
\DeclareMathOperator{\spann}{span}

\DeclareMathOperator{\Tan}{Tan}
\DeclareMathOperator{\ran}{ran}
\newcommand{\SL}{\mathrm{SL}}
\newcommand{\ol}{\overline}
\newcommand{\loc}{\mathrm{loc}}
\DeclareMathOperator{\Hol}{Hol}
\DeclareMathOperator{\Int}{Int}
\DeclareMathOperator{\wlim}{wlim}

\renewcommand{\Re}{\operatorname{Re}}
\renewcommand{\Im}{\operatorname{Im}}

\newcommand{\Dummy}{\text{\textvisiblespace\kern1pt}}

\newcommand{\bbHe}{\bbH^{\raisebox{1.5pt}{$\scriptscriptstyle 1$}}}
\newcommand{\bbHle}{\bbH^{\raisebox{1.5pt}{$\scriptscriptstyle\leq 1$}}}
\newcommand{\bbHen}{\bbH^{\sfrac 10}}
\newcommand{\bbTM}{{\mathbb T\mathbb M}}
\newcommand{\RK}{{\mathbb{R}\mathbb K}}
\newcommand{\DB}{{\mathbb{D}\mathbb B}}
\newcommand{\HB}{{\mathbb{H}\mathbb B}}
\newcommand{\clH}{{\mathcal H^{\flat}}}
\newcommand{\Ch}{\mathrm{Ch}}
\newcommand{\bCh}{\mathrm{b}\text{-}\mathrm{Ch}}
\newcommand{\ubCh}{\mathrm{ub}\text{-}\mathrm{Ch}}

\allowdisplaybreaks

\author[B.\ Eichinger]{Benjamin Eichinger}
\address{B.\ Eichinger: TU Wien, Wien, A-1040, Austria and School of Mathematical Sciences, Lancaster University, Lancaster LA1 4YX, United Kingdom }
\email{benjamin.eichinger@tuwien.ac.at}
\thanks{B.E.\ was supported by the Austrian Science Fund FWF, project no: P33885}

\author[M.\ Luki\'c]{Milivoje Luki\'c}
\address{M.\ Luki\'c: Emory University, Atlanta, GA~30305, USA}
\email{milivoje.lukic@emory.edu}
\thanks{M.L.\ was supported in part by NSF grant DMS--2154563/2626207 and DMS--2453758/2626193.}

\author[H.\ Woracek]{Harald Woracek}
\address{H.\ Woracek: TU Wien, Wien, A-1040, Austria}
\email{harald.woracek@tuwien.ac.at}
\thanks{H.W.\ was supported by the project I 4600 of the Austrian Science Found (FWF)}

\title{Necessary and sufficient conditions for universality limits}
\date{\today}

\begin{document}

\maketitle

\vspace{-10pt}
\begin{center}
	{\it  Dedicated to the memory of Heinz Langer, one of the founders \\of the theory of indefinite inner product spaces.}
\end{center}

\begin{abstract}
	We derive necessary and sufficient conditions for universality limits for orthogonal polynomials associated to a fixed measure on the real line and
	related systems. One of our results is that the Christoffel--Darboux kernel has sine kernel asymptotics at a point
	$\xi$, with regularly varying scaling, if and only if the orthogonality measure (spectral measure) has a unique tangent
	measure at $\xi$ and that tangent measure is the Lebesgue measure. All previously known cases of sine kernel asymptotics, with absolutely continuous or singular measures, have regularly varying scaling.

	Our work is not limited to bulk universality; we show that the Christoffel--Darboux kernel has a regularly varying scaling limit with a nontrivial limit kernel if and only if the orthogonality measure
	has a unique tangent measure at $\xi$ and that tangent measure is not a point mass.  The possible limit kernels correspond to homogeneous de Branges spaces; in particular, this equivalence completely characterizes several prominent universality classes such as hard edge universality, Fisher--Hartwig singularities, and jump discontinuities in the weights.
	
A key ingredient in the proof is a new correspondence between measures and chains of de Branges spaces, not between
Weyl functions and Hamiltonians. In order to handle limits with power law weights, this correspondence goes beyond
the more common setting of Poisson-finite measures, and allows arbitrary power bounded measures.
\end{abstract}

%
%

\section{Introduction}

In this paper, we derive necessary and sufficient conditions for universality limits. 
Our main result is \Cref{theorem:main}, which is formulated and proved in a very general setting. 
We will initially present the results in the setting of orthogonal polynomials on the 
real line (OPRL), although this is only a particular case. In this setting the topic has a long history.
In course of the presentation we gradually increase generality, and \Cref{theorem3} is the most general result in the OPRL setting.
\subsubsection*{Stage 1: Bulk universality}
\phantom{}\\[4pt]
We start our journey with discussing universality for OPRL in the sine kernel regime. 
Consider orthogonal polynomials $(p_n(z))_{n=0}^\infty$ with respect to a measure $\mu$ on $\bbR$, obtained by the Gram--Schmidt process in $L^2(\bbR,d\mu)$ from the sequence $(z^n)_{n=0}^\infty$.
 By general principles, the polynomial $p_n$ has $n$ real simple zeros. Their local distribution/spacing on the real line is a question of classical interest; in the setting of Jacobi polynomials, it was long known \cite{Szego1939} that zeros of Jacobi polynomials are locally asymptotically uniformly spaced, and this was generalized by Erd\H os--Tur\'an \cite{ErdosTuran1940} to a class of smooth weights on an interval. In modern literature, this phenomenon is known as clock behavior, and it is stated as follows. Zeros of $p_n$ can be indexed by counting to the left and right from a fixed point $\xi \in \bbR$, denoting them by $\xi_j^{(n)}$ so that
\[
\dots < \xi_{-1}^{(n)} < \xi_0^{(n)} \le \xi < \xi_1^{(n)} < \xi_2^{(n)} < \dots
\]
(of course, only $n$ of these are well-defined for fixed $n$).  The measure can be said to have clock behavior at $\xi$ if for every $j\in\bbZ$, $\xi_j^{(n)}$ is well-defined for all large enough $n$ and if 
 for some scaling sequence $\tau_n \to \infty$ as $n\to\infty$,
\begin{equation}\label{eqn:clockbehavior}
\lim_{n\to\infty} \tau_n (\xi_{j+1}^{(n)} - \xi_j^{(n)} ) = 1
\end{equation}
for every $j \in \bbZ$ (it is common to impose additional assumptions on $\tau_n$).

A theorem of Freud \cite{freud:1969}, rediscovered by Levin in \cite{LevinLubinsky08}, states that clock behavior follows from a local scaling limit of the Christoffel--Darboux (CD) kernel: namely, the CD kernel for the measure $\mu$ is defined as
\begin{equation}\label{eqnCDkernel1}
K(n,z,w) = \sum_{j=0}^{n-1} p_j(z) \ol{ p_j(w) }
\end{equation}
and clock behavior \eqref{eqn:clockbehavior} follows from the local scaling limit
\begin{equation}\label{eqn:scalinglimit}
\lim_{n\to\infty} \frac 1{K(n,\xi,\xi)} K\left( n,\xi + \frac{z}{\tau_n}, \xi + \frac{w}{\tau_n} \right) = \frac{ \sin (\pi(z - \ol{w})) }{ \pi (z - \ol{w})}.
\end{equation}
The phenomenon \eqref{eqn:scalinglimit} is called bulk universality, and sufficient conditions for bulk universality have been greatly studied in the literature.

Many different methods were developed to prove it under different sufficient conditions. Riemann--Hilbert techniques were used by Kuijlaars--Vanlessen \cite{KuiljaarsVanlessenIMRN02} for Jacobi-like analytic weights on $[-1,1]$. Another method was found by Lubinsky \cite{LubinskyAnnals}, with further developments by \cite{Findley08,SimonTwoExt08,TotikUniv09,Totik16}, which instead requires Stahl--Totik regularity \cite{StahlTotik92} of the measure and Lebesgue point and local Szeg\H o conditions at the point $\xi$.  A second approach of Lubinsky \cite{LubinskyJourdAnalyse08} is conditional on the behavior of the CD kernel on the diagonal; this was used by Avila--Last--Simon \cite{AvilaLastSimon} to prove bulk universality for ergodic Jacobi matrices on an essential support of the a.c.\ spectrum. \'Swiderski--Trojan proved bulk unversality for certain classes of orthogonal polynomials for measures with unbounded support \cite{swiderski.trojan:2021,swiderski.trojan2:2021}. Breuer \cite{breuer:2011} found the first examples of bulk universality with singular measures, within the class of sparse decaying discrete Schr\"odinger operators. Lubinsky first explored the connection with de Branges spaces \cite{LubJFA,LubinskyJourdAnalyse08} and Breuer \cite{breuer:2021} first observed the connection of bulk universality limits and operator scaling limits to constant canonical systems. Using the theory of canonical systems, a local sufficient condition was proved by Eichinger--Luki\'c--Simanek \cite{EichLukSimanek}: a strictly positive, finite nontangential limit of the Poisson transform of the measure at a point implies bulk universality at that point. This approach works with the continuous family of kernels $\{ K(t,z,w) \mid t \in [0,\infty) \}$ obtained by  piecewise linear interpolation,
\begin{equation}\label{eqnPiecewiseLinear1}
K(t,z,w) = K(\lfloor t\rfloor,z,w)  + (t - \lfloor t \rfloor) (K(\lfloor t \rfloor + 1,z,w) - K({\lfloor t \rfloor} ,z,w) ),
\end{equation}
which naturally appears through the reduction of a Jacobi matrix to a canonical system.

We digress to say that similar scaling limits of CD kernels are motivated by random matrix theory; the eigenvalues of random matrix ensembles with a unitary conjugation invariance are a determinantal point process whose correlation kernel is precisely the CD kernel \cite{GaudinMehta1960,DeiftOPandRM}, so certain limits of CD kernels encode local eigenvalue statistics of the random matrices. In that setting, an explicit $n$-dependence is naturally placed in the measure, so this is often referred to as a varying measure limit, with pioneering work by Bleher--Its \cite{BleherIts}, Pastur--Shcherbina \cite{PasturShcherbina}, and Deift--Kriecherbauer--McLaughlin--Venakides--Zhou \cite{DKMVZ1,DKMVZ2,DKMVZ3}, see also \cite{DeiftOPandRM} and the survey of Lubinsky \cite{LubSigma16}.

Returning to the "fixed measure" bulk universality limit \eqref{eqn:scalinglimit}, we obtain necessary and sufficient conditions for bulk universality. We first state the special case of our results with $\tau_n = \eta K(n,\xi,\xi)$, which we show is linked with the derivative of the measure:

\begin{theorem}\label{theorem1}
Let $\mu$ be a measure on $\bbR$ with finite moments, not supported on a finite set, with a determinate moment problem. For any $\xi \in\bbR$ and any $\eta \in (0,\infty)$, the following are equivalent:
\begin{enumerate}[{\rm (i)}]
\item 
\begin{equation}\label{theorem1condition}
\eta = \lim_{\epsilon\downarrow 0} \frac{ \mu((\xi - \epsilon, \xi))}\epsilon =  \lim_{\epsilon\downarrow 0} \frac{ \mu([\xi, \xi+ \epsilon))}\epsilon
\end{equation}
\item Uniformly on compact subsets of $(z,w) \in \bbC \times \bbC$,
\begin{equation}\label{eqntheorem1limit}
\lim_{t \to \infty} \frac{ K\left( t, \xi + \frac {z}{\eta K(t,\xi,\xi)} ,  \xi + \frac {w}{\eta K(t,\xi,\xi)} \right) }{K(t,\xi,\xi)} = \frac{ \sin (\pi(z - \ol{w})) }{ \pi (z - \ol{w})}.
\end{equation}

\item Uniformly on compact subsets of $(z,w) \in \bbC \times \bbC$,
\begin{equation}\label{eqntheorem1limitn}
\lim_{n \to \infty} \frac{ K\left(n, \xi + \frac {z}{\eta K(n,\xi,\xi)} ,  \xi + \frac {w}{\eta K(n,\xi,\xi)} \right) }{K(n,\xi,\xi)} = \frac{ \sin (\pi(z - \ol{w})) }{ \pi (z - \ol{w})}
\end{equation}
and
\begin{equation}\label{eqnNevaiConditionRelated}
\lim_{n\to\infty} \frac{ K(n+1,\xi,\xi) }{ K(n,\xi,\xi)} = 1.
\end{equation}
\end{enumerate}
\end{theorem}

Note that we consistently use $n\in\bbN$ as a discrete parameter and $t \in [0,\infty)$ as a continuous parameter; in particular, the equivalence of (ii) and (iii) above relates the sine kernel asymptotics for  the continuous family of linearly interpolated kernels \eqref{eqnPiecewiseLinear1} to that for the original sequence of CD kernels \eqref{eqnCDkernel1}.  

\begin{remark}
\begin{enumerate}[(i)]
\item The determinate moment problem condition means that $\mu$ is uniquely determined by its moments  $\int \xi^n \,d\mu(\xi)$, $n=0,1,2,\dots$. A sufficient condition is exponential decay of the tails, $\int e^{\epsilon\lvert \xi\rvert}\, d\mu(\xi) < \infty$ for some $\epsilon >0$.

\item Theorem~\ref{theorem1} describes bulk universality at the scale $\tau_n = \eta K(n,\xi,\xi)$. The reciprocal of $K(n,\xi,\xi)$ is known as the Christoffel function and its asymptotic behavior is widely studied. By work of M\'at\'e--Nevai--Totik \cite{MateNevaiTotik1991} and Totik \cite{Totik2000}, Stahl--Totik regularity and local Lebesgue point/local Szeg\H o conditions at the point imply that $K(n,\xi,\xi)$ grows linearly with $n$. This contains previous bulk universality results for compactly supported measures with an a.c. part; thus, although those results were formulated at the explicit scale $\tau_n = c n$, this is equivalent to the scaling limit \eqref{eqntheorem1limit}.

\item The condition \eqref{eqnNevaiConditionRelated} is equivalent to
\[
\lim_{n\to\infty} \frac{ p_n(\xi)^2}{ \sum_{j=0}^{n-1} p_j(\xi)^2 } = 0
\]
and sometimes described as subexponential growth of orthogonal polynomials $p_n(\xi)$. It is closely related to the Nevai condition \cite{BreuerLastSimon2010,BreuerDuits14,swiderski:2026}.

\item  Most prior results were based on a mix of global and local assumptions of the measure, and the local assumptions included Lebesgue point conditions on the measure; in particular, they required presence of an a.c.\ part of the measure. Theorem~\ref{theorem1} is completely local, and the local condition \eqref{theorem1condition} is weaker than a Lebesgue point condition; in particular, it makes it obvious that bulk universality at a single point can even be achieved for a pure point measure (see Lemma~\ref{lemmaBulkUniversalityPurePointMeasure}). 
\item It was proved in \cite[Theorem 1.2]{EichLukSimanek} that if for some $\alpha \in (0,\pi/2)$,
\begin{equation}\label{ELScriterionNew}
\lim_{\substack{z\to\xi \\ \alpha \le \arg(z-\xi) \le \pi - \alpha}} \frac 1\pi \Im \int \frac 1{\lambda -z} \,d\mu(\lambda) = \eta
\end{equation}
then \eqref{eqntheorem1limit} holds. For any $\alpha \in (0,\pi/2)$ and $\eta \in (0,\infty)$, this sufficient condition \eqref{ELScriterionNew} is equivalent to \eqref{theorem1condition}, by a general result of Loomis~\cite{Loomis43} for positive harmonic functions (see also 
 \cite{RameyUllrich88} which gives a proof related to our rescaled Weyl functions). Thus, Theorem~\ref{theorem1} shows that the implication in \cite{EichLukSimanek} is optimal; however, the opposite implication (ii)$\implies$(i) of Theorem~\ref{theorem1} is outside the scope of the method in \cite{EichLukSimanek}.
\end{enumerate}
\end{remark}

The approach in \cite{EichLukSimanek} is based on a homeomorphism between  trace-normalized limit circle-limit point Hamiltonians and Nevanlinna functions (analytic maps $\bbC_+ \to \ol{\bbC_+}$, where $\bbC_+ = \{z \in \bbC \mid \Im z > 0\}$ and  $\ol{\bbC_+} = \bbC_+ \cup \bbR \cup \{\infty\}$). In particular, the implication \eqref{ELScriterionNew}$\implies$\eqref{eqntheorem1limit} was proved by a shifted rescaling trick which does not give the converse implication and does not easily generalize to other situations.  The approach in this paper is different, and at its core is a bijection between certain measures and certain chains of de Branges spaces. We define a notion of convergence of chains and prove continuity properties of this bijection. This point of view is better suited for the study of convergence of kernels, and necessary for statements such as the implication (ii)$\implies$(i) of Theorem~\ref{theorem1}.

\subsubsection*{Stage 2: Bulk universality with regularly varying scaling}
\phantom{}\\[4pt]
We turn to more general sine kernel asymptotics. For the next result, the derivative condition on the measure is replaced by 
a tangent measure condition. We provide the required definitions before stating the result.

For a locally finite measure $\mu$ in $\bbC$,  $\xi \in \bbC$ and $r > 0$, consider the affine pushforwards of $\mu$ defined by $\mu_{\xi,r}(A) = \mu(\xi + A/r)$ for Borel sets $A$. A measure $\nu$ is a tangent measure of $\mu$ at $\xi$ if $\nu$ is locally finite, $\nu(\bbC) > 0$, and there exist positive sequences $c_n, r_n$ with $r_n \to \infty$ and $c_n \mu_{\xi, r_n} \to \nu$ weakly in $C_c(\bbC)^*$ as $n\to\infty$. The set of tangent measures of $\mu$ at $\xi$ is denoted $\Tan(\mu,\xi)$.  This notion was introduced in geometric measure theory by Preiss~\cite{Preiss87}, see also \cite{MattilaBook}.

The set $\Tan(\mu,\xi)$ is closed under multiplication by a positive scalar. It is said that $\mu$ has a unique tangent measure at $\xi$ if there exists $\nu$ such that $\Tan(\mu,\xi) = \{ c \nu \mid c \in (0,\infty)\}$.

A measurable function $g:(0,\infty) \to (0,\infty)$ is said to be regularly varying (at $\infty$) with index $\beta$ if for all $c \in (0,\infty)$, $g(c r) / g(r) \to c^\beta$ as $r \to \infty$.  Regularly varying functions were introduced by Karamata \cite{Karamata30,Karamata33} and play an important role in Abelian and Tauberian theorems; see also \cite{bingham.goldie.teugels:1989}, and applications to spectral theory \cite{eckhardt.kostenko.teschl:2018,langer.woracek:kara,langer.pruckner.woracek:arXiv,pruckner.woracek:1}.  
 Another regularly varying function $h$ is said to be an asymptotic inverse of $g$ if $h(g(r)) / r \to 1$ and $g(h(r))/r \to 1$ as $r \to \infty$. Every regularly varying function $g$ of order $\beta > 0$ has an asymptotic inverse $h$ of order $1/\beta$.

\begin{theorem}\label{theorem2}
Let $\mu$ be a measure on $\bbR$ with finite moments, not supported on a finite set, with a determinate moment problem.  For any $\xi \in\bbR$, the following are equivalent:
\begin{enumerate}[{\rm (i)}]
\item There exists $\scrg$ regularly varying with index $1$ such that 
\[
1 = \lim_{r \to \infty}  \scrg(r) \mu\left(\left(\xi - \tfrac 1r, \xi\right)\right) =  \lim_{r\to \infty} \scrg(r) \mu\left(\left[\xi, \xi+ \tfrac 1r\right)\right)
\]
\item $\Tan(\mu,\xi) = \{ c m \mid c \in (0,\infty)\}$ where $m$ denotes Lebesgue measure on $\bbR$
\item There exists $\scrh$ regularly varying with index $1$ such that uniformly on compact subsets of $(z,w) \in \bbC \times \bbC$,
\begin{equation}\label{RegularlyVaryingBulkUniversalityLimit}
\lim_{t\to\infty}\frac{K\left(t,\xi + \frac{z}{\scrh(K(t,\xi,\xi))}, \xi + \frac{w}{\scrh(K(t,\xi,\xi))}\right)}{K(t,\xi,\xi)} 
 = \frac{ \sin (\pi(z - \ol{w})) }{ \pi (z - \ol{w})}.
\end{equation}
\item There exists $\scrh$ regularly varying with index $1$ such that uniformly on compact subsets of $(z,w) \in \bbC \times \bbC$,
\[
\lim_{n \to \infty} \frac{ K\left(n, \xi + \frac {z}{\scrh( K(n,\xi,\xi))} ,  \xi + \frac {w}{\scrh( K(n,\xi,\xi))} \right) }{K(n,\xi,\xi)} = \frac{ \sin (\pi(z - \ol{w})) }{ \pi (z - \ol{w})}
\]
and \eqref{eqnNevaiConditionRelated} holds.
\end{enumerate}
In this case, $\scrh$ is an asymptotic inverse of $\scrg$.
\end{theorem}

\begin{remark}[Scaling functions and spectral type]
\begin{enumerate}[(i)]
\item Scaling by regularly varying functions is significantly more general than power law scaling; it allows, e.g., additional logarithmic factors, and scaling functions such as $\scrg(r) = r \log^\kappa (1+r)$, $\kappa\in \bbR$. Thus, a sine kernel limit can exist even where $\mu$ has zero or infinite derivative with respect to Lebesgue measure.
\item Bulk universality with regularly varying scaling on a set implies $1$-dimensionality of the measure $\mu$ on this set (see  Theorem~\ref{thmBulkUniversality1dimensional}).
\item Our work generalizes past results based on a set of assumptions on the measure $\mu$. Breuer's result in \cite{breuer:2011} is not of this form, as it is based on assumptions on the Jacobi coefficients: it shows that very sparse decaying discrete Schr\"odinger operators have bulk universality limits. By combining this result with our Theorem~\ref{theorem2}, we can conclude that very sparse decaying discrete Schr\"odinger operators have the following precise behavior of the spectral measure $\mu$: 
for every $\xi \in (-2,2)$,
\begin{equation}\label{eqn14oct1}
\lim_{n\to\infty} \scrg_\xi(n) \mu\left(\left( \xi - \tfrac 1{n} , \xi \right)\right) = \lim_{n\to\infty} \scrg_\xi(n) \mu\left(\left[\xi, \xi + \tfrac 1{n} \right)\right) = 1,
\end{equation}
with an explicit function $\scrg_\xi(n)$. This class is notable because it contains singular spectral measures. 
See Section~\ref{sec:SpectralType} for proofs and further discussion.
\item 
Although it is more common to think about sufficient conditions for scaling limits of kernels,  the above discussion relied on the implication Theorem~\ref{theorem2}(iv)$\implies$(i) from scaling limits of kernels to local properties of the measure. There are only a few prior results which go in this direction \cite{BLS2,BreuerDuits14}.
\item There exist finite measures $\mu$ on $[0,1]$ which are singular with respect to Lebesgue measure, but have Lebesgue measure as the unique tangent measure at every $\xi \in (0,1)$ (see \cite[Example 5.9]{Preiss87}, \cite{FreedmanPitman90}). By Theorem~\ref{theorem2}, this provides new examples of singular measures for which bulk universality holds at every $\xi \in (0,1)$. Unlike the examples in \cite{breuer:2011}, no discrete spectrum is needed.
\end{enumerate}
\end{remark}

\subsubsection*{Stage 3: Moving away from the sine kernel}
\phantom{}\\[4pt]
The most prominent example of a universality class other than bulk universality is hard edge universality, 
which has traditionally been studied at the edge of the support of $\mu$ and is characterized by a limiting Bessel kernel.  
In particular, for Jacobi-type measures with analytic weights on $[-1,1]$,  at the endpoint $\xi = 1$, Kuijlaars--Vanlessen \cite{KuiljaarsVanlessenIMRN02,KMVV04} proved hard edge universality by Riemann--Hilbert methods; Lubinsky \cite{LubinskyCM08} generalized this to a class of Stahl--Totik regular measures on $[-1,1]$, and proved in \cite{LubIMRN08} a conditional statement at a gap edge of the support.

Other limit kernels have been found for other power-law behaviors of the weight $d\mu(\lambda) / d\lambda$. For Fisher--Hartwig singularities, Vanlessen \cite{vanlessen:03} proved strong asymptotics and Danka \cite{Danka17} proved a universality limit for a class of Stahl--Totik regular measures (see also Kuijlaars--Vanlessen \cite{kuijlaarsvanlessen:03}). For a class of step-like analytic weights, Foulqui\'e~Moreno--Mart\'inez-Finkelstein--Sousa \cite{MorenoFinkelSousaConstr} proved a hypergeometric kernel scaling limit.

We characterize the limiting behavior of the kernel when the measure has a local behavior of the form
\begin{equation}\label{eqn:muLocalBehavior}
\lim_{r \to \infty} \scrg(r) \mu\left(\left(\xi- \tfrac 1r,\xi\right)\right) =\sigma_-,\quad \lim_{r \to \infty} \scrg(r) \mu\left(\left[ \xi, \xi + \tfrac 1r \right)\right) =\sigma_+.
\end{equation}
for some regularly varying function $\scrg$ with index $\beta>0$ and some $(\sigma_-,\sigma_+)\in [0,\infty)\times[0,\infty)\setminus\{(0,0)\}$. The limit kernel will be a function of $\sigma_\pm$ and $\beta$, as follows:

\begin{definition} \label{defn:LimitKernels}
Let $(\sigma_-, \sigma_+) \in [0,\infty)^2 \setminus \{(0,0)\}$ and $\beta > 0$. Recall:
\[
	M(\alpha,\beta,z):=\sum_{n=0}^\infty\frac{(\alpha)_n}{(\beta)_n}\cdot\frac{z^n}{n!}
	,\quad
	\prescript{}{0}{F}_1(\beta,z):=\sum_{n=0}^\infty\frac 1{(\beta)_n}\cdot\frac{z^n}{n!}
	,
\]
where $\alpha,z\in\bbC$ and $\beta\in\bbC\setminus(-\bbN_0)$. The symbol $(\Dummy)_n$ denotes the rising factorial, i.e., 
\[
	(\alpha)_0=1,\quad (\alpha)_{n+1}=(\alpha)_n(\alpha+n)\quad\text{for }n\in\bbN_0
	.
\]
We define functions $A,B$ by distinguishing two cases.
	\begin{enumerate}[{\rm(i)}]
	\item Assume that $\sigma_+,\sigma_->0$. Define 
		\begin{align*}
			& \alpha:=\frac i{2\pi}\log\frac{\sigma_-}{\sigma_+}+\frac{\beta-1}2
			,\quad 
			\kappa:=\frac 12
			\Big(\frac{2\Gamma(\beta+1)^2\pi\sqrt{\sigma_+\sigma_-}}{|\Gamma(\alpha+1)|^2}\Big)^{\frac 1\beta}
			,
			\\
			& A(z):=e^{i\kappa z}\frac{M(\alpha,\beta,-2i\kappa z)+M(\alpha+1,\beta,-2i\kappa z)}2
			,
			\\
			& B(z):=z  e^{i\kappa z}M(\alpha+1,\beta+1,-2i\kappa z)
			.
		\end{align*}
	\item Assume that $\sigma_+=0$ or $\sigma_-=0$. Define 
		\begin{align*}
			& \sigma:=
			\begin{cases}
				\big(\sigma_+\Gamma(\beta+1)^2\big)^{\frac 1\beta} &\text{if}\ \sigma_+>0
				,
				\\
				-\big(\sigma_-\Gamma(\beta+1)^2\big)^{\frac 1\beta} &\text{if}\ \sigma_->0
				,
			\end{cases}
			\\
			& A(z):=\prescript{}{0}{F}_1(\beta,-\sigma z)
			,\quad 
			B(z):=z\cdot\prescript{}{0}{F}_1(\beta+1,-\sigma z)
			.
		\end{align*}
	\end{enumerate}
	Now set 
	\[
		K_{\sigma_-, \sigma_+,\beta}(z,w):=\frac{B(z)A(\overline w)-A(z)B(\overline w)}{z-\overline w}.
	\]
\end{definition}
This kernel is expressed in terms of entire functions. When $\sigma_+ = 0$ or $\sigma_-=0$, the kernel can be rewritten in terms of Bessel functions, as is customary in the hard edge literature (see \cite[Remark 4.2]{eichinger.woracek:homo-arXiv}); likewise, when $\sigma_+ = \sigma_-$, it can be rewritten in terms of Bessel functions (see Lemma~\ref{lemma:LimitKernelFisherHartwig}).

\begin{theorem}\label{theorem3}
Let $\mu$ be a measure on $\bbR$ with finite moments, not supported on a finite set, with a determinate moment problem.  For any $\xi \in\bbR$, the following are equivalent:
		\begin{enumerate}[{\rm (i)}]
			\item There exists a regularly varying function $\scrg$ with  index $\beta>0$ and $(\sigma_-,\sigma_+)\in [0,\infty)\times[0,\infty)\setminus\{(0,0)\}$ such that  \eqref{eqn:muLocalBehavior} holds.
			\item $\mu$ has a unique tangent measure at $\xi$, which is not the Dirac measure $\delta_0$.
			\item There exists $\scrh$ regularly varying with index $\rho>0$ such that uniformly on compact subsets of $\bbC\times\bbC$
			\begin{equation}\label{eqn:Kcontlimit1}
			\lim\limits_{t\to\infty}\frac{K\left(t,\xi+\frac{z}{\scrh(K(t,\xi,\xi))},\xi+\frac{w}{\scrh(K(t,\xi,\xi))}\right)}{K(t,\xi,\xi)}=K_\infty(z,w),
			\end{equation}
			and $K_\infty\not\equiv 1$.
			\item There exists $\scrh$ regularly varying at $\infty$ with index $\rho>0$ such that uniformly on compact subsets of $\bbC\times\bbC$
\begin{equation}\label{eqnTheorem3n}
			\lim_{n\to\infty}\frac{K\left(n,\xi+\frac{z}{\scrh(K(n,\xi,\xi))},\xi+\frac{w}{\scrh(K(n,\xi,\xi))}\right)}{K(n,\xi,\xi)}=K_\infty(z,w)
			\end{equation}
			with $K_\infty\not\equiv 1$ and \eqref{eqnNevaiConditionRelated} holds.
\end{enumerate}
		In this case $\rho=1/\beta$,  $\scrh$ is the asymptotic inverse of $\scrg$, and $K_\infty = K_{\sigma_-, \sigma_+,\beta}$. 
\end{theorem}

The assumption $(\sigma_-,\sigma_+) \neq (0,0)$ is used to rule out a trivial limit obtained by a scaling function $\scrg$ which is too small, since such a trivial limit would carry no information; likewise for the assumption $K_\infty \not \equiv 1$.

We compare Theorem~\ref{theorem3} with prior literature:

\begin{remark}
\begin{enumerate}[(i)]
\item Theorem~\ref{theorem3} contains as special cases several universality classes studied separately in the literature. Most notably, $\sigma_- = \sigma_+$ and $\beta=1$ is bulk universality, $\sigma_+ = 0$ is hard edge universality, $\sigma_- = \sigma_+$ is a Fisher--Hartwig singularity, and $\sigma_- \neq \sigma_+$ with $\beta = 1$ is a jump discontinuity.
\item For any of those limit kernels except the sine kernel, prior literature required analyticity of the weight or Stahl--Totik regularity of the measure, and Theorem~\ref{theorem3} is the first completely local sufficient condition for a scaling limit.
\item Even if Theorem~\ref{theorem3} is specialized to the power law case $\scrg(r) = r^\beta$, the local condition is still weaker than the local assumptions in the prior literature; prior literature always assumed power law scaling of the weight $d\mu(\lambda) / d\lambda$.
\item In an indeterminate moment problem, there are many qualitatively different measures corresponding to the same sequence of moments and therefore orthogonal polynomials and CD kernels. Thus, in an indeterminate moment problem, no general result such as the above equivalence is possible; more specialized results exist, see \cite{CKM}.
\end{enumerate}
\end{remark}

For the study of bulk universality as in Theorems~\ref{theorem1} and \ref{theorem2}, an important realization was that although one starts from a probability measure $\mu$, convergence should be viewed in the larger space of Poisson-finite measures/Nevanlinna functions. This is motivated by the fact that bulk universality corresponds to having Lebesgue measure as the tangent measure.  For the setting of Theorem~\ref{theorem3}, extending to an even larger class of measures/functions is necessary. Namely, the local behavior \eqref{eqn:muLocalBehavior} corresponds to a tangent measure with a power law scaling; this measure need not be Poisson-finite, but it has power law growth at $\infty$ (see \cite{mattila:2005} and Lemma~\ref{lemmaA2}). Accordingly, the core of our approach is a homeomorphism between power-bounded measures on $\bbR$ (measures $\mu$ such that $\int (1+\lambda^2)^{-N} \,d\mu(\lambda) < \infty$ for some $N$) and a certain class of chains of de Branges spaces.  On the function theoretic side, going beyond Poisson-finite measures, is reflected in passing to reproducing kernels with a finite number of negative eigenvalues. The Weyl functions are no longer Nevanlinna functions, but are in the larger class of generalized Nevanlinna functions suitable for the indefinite setting and introduced by Krein--Langer \cite{krein.langer:1977}.

It has been observed before \cite{LubSigma16,BariczDanka} that the proof of the Freud--Levin theorem extends to other limiting kernels (see Theorem~\ref{thm:FreudLevinGeneral}). Combining that proof with Theorem~\ref{theorem3} gives the following description of the local configuration of zeros of orthogonal polynomials:

\begin{corollary}\label{cor:ZeroDistributionPolynomial}
Let $\mu$ be a measure on $\bbR$ with finite moments, not supported on a finite set, with a determinate moment problem.  Let $\xi \in\bbR$. If there exists a regularly varying function $\scrg$ with index $\beta>0$ and $\sigma_- \in [0,\infty)$, $\sigma_+ \in (0,\infty)$ such that  \eqref{eqn:muLocalBehavior} holds, then:
\begin{enumerate}[(i)]
\item For every $k \ge 0$, for all large enough $n$, $p_n$ has at least $k$ zeros larger than $\xi$; in other words the $k$-th zero to the right of $\xi$, denoted $\xi_k^{(n)}$, is well-defined for all large enough $n$.
\item The function $K_{\sigma_-, \sigma_+,\beta} (\cdot, 0)$ has infinitely many positive zeros. Denoting by $\theta$ its smallest positive zero,
\[
\limsup_{n\to\infty} \scrh(K_n(\xi,\xi)) ( \xi_1^{(n)} - \xi )  \le \theta.
\]
\item Fix a sequence $n_j$ such that the limit
\[
 \lim_{j\to\infty} \scrh(K_{n_j}(\xi,\xi)) ( \xi_1^{(n_j)} - \xi )
\]
exists. Denote its value by $\kappa_1$ and denote by $\kappa_2 < \kappa_3 < \dots$ all the zeros of $K_{\sigma_-, \sigma_+,\beta} (\cdot, \kappa_1)$ in $(\kappa_1, \infty)$. 
Then for every $k \in \bbN$,
\[
\lim_{j\to\infty}  \scrh(K_{n_j}(\xi,\xi)) ( \xi_k^{(n_j)} - \xi ) = \kappa_{k}.
\]
\end{enumerate}
\end{corollary}

In the special case of the sine kernel limit, the differences $\kappa_{k+1} - \kappa_k$ are independent of $\kappa_1$ or $k$, which is why clock behavior has a more elegant formulation. However, the conclusions are of the same strength: they describe the local zero configuration up to one free parameter (location of the nearest zero to the right of $\xi$).

Convergent subsequences as in Corollary~\ref{cor:ZeroDistributionPolynomial}(iii) exist by compactness of $[0,\theta]$, but in general one cannot expect convergence of the sequence $\scrh(K_n(\xi,\xi)) ( \xi_1^{(n)} - \xi )$. Such convergence, however, holds in the hard edge case; to state this, we denote positive zeros of the Bessel function  $J_\nu$ of the first kind and order $\nu$ by $j_{\nu,1}$, $j_{\nu,2}$, \dots, so that
\begin{equation}\label{eqn:BesselFunctionZeros}
0 < j_{\nu,1} < j_{\nu,2} < \dots
\end{equation}

\begin{theorem}\label{thm:ZerosHardEdge}
Let $\mu$ be a probability measure on $[0,\infty)$ with a determinate Stieltjes moment problem. 
Denote by $\xi_1^{(n)} < \dots < \xi_n^{(n)}$ the zeros of $p_n$. If the function
\[
\scrg(r) = 1 / \mu([0,1/r))
\]
is a regularly varying function with index $\beta > 0$, then for every $k \in \bbN$,
\begin{equation}\label{eqn:HardEdgeZerosResult}
\lim_{n\to\infty}  \scrh(K(n,0,0)) \xi_k^{(n)} =  \frac{1}{4 \Gamma(\beta+1)^{2/\beta}}   j_{\beta-1,k}^2,
\end{equation}
where  $\scrh$ denotes an asymptotic inverse of $\scrg$. 
\end{theorem}

This was previously proved by Levin--Lubinsky \cite[Theorem 1.2]{LevinLubinsky08} in the special case of  Stahl--Totik regular measures whose essential spectrum is a compact interval, and for which $\mu$ is purely a.c.\ on some subinterval $[0,\rho]$, with weight $d\mu(\lambda) / d\lambda \sim \lambda^{\beta-1}$ as $\lambda \to 0$. We note that determinacy of the Stieltjes moment problem follows, e.g., from $\int e^{\epsilon \sqrt \lambda} \,d\mu(\lambda) < \infty$ for some $\epsilon > 0$.

\subsubsection*{Final stage: Natural environment for universality with regularly varying scale}
\phantom{}\\[4pt]
The natural environment for our results is the setting of $J$-decreasing transfer matrices. To explain this, we must set the following notation. To denote the action of a fractional linear transformation on the Riemann sphere, for a matrix $M=(m_{ij})_{i,j=1,2}$ with $\det M\neq 0$ and a point $\tau\in\bbC\cup\{\infty\}$ we set, 
with the usual conventions concerning algebra in $\bbC\cup\{\infty\}$,
\begin{equation}\label{eqn:Mobius}
M\star\tau:=\frac{m_{11}\tau+m_{12}}{m_{21}\tau+m_{22}}.
\end{equation}
For an entire function $f$ we denote
\[
f^\sharp(z) = \ol{ f( \ol z)}
\]
and say $f$ is real if $f = f^\sharp$. We denote
\[
J = \begin{pmatrix} 0 & -1 \\ 1 & 0 \end{pmatrix}.
\]

\begin{definition}\label{def:JInner}
An entire matrix function $W: \bbC \to \bbC^{2\times 2}$ with real entries and $\det W = 1$ is $J$-inner if  
\[
\frac{W(z) J W(w)^* - J}{z-\ol w}
\]
is a positive kernel on $\bbC$. A family $\{ W(t,z) \mid a \le t < b \}$ of such functions is $J$-decreasing if $W(t_1,z)^{-1} W(t_2,z)$ is $J$-inner whenever $t_1 < t_2$.

Decreasing families of $J$-inner matrix functions are another perspective on canonical systems and de Branges spaces, which are an active area of research \cite{belov.borichev:2023,bessonov.denisov:2023,makarov.poltoratski:2023,makarov.poltoratskii:2025,poltoratski:2024,romanov.woracek:2020}.
Such a family is in the limit point case if for every $z\in \bbC_+$ and $\tau\in \ol{\bbC_+}$ the limit
\begin{align}\label{eq:36intro}
q(z):=\lim\limits_{t\to b}\Big[W(t,z)\star\tau\Big]
\end{align}
exists, and its value is independent of $\tau$. The function $q$ is a Nevanlinna function. Thus, if $q \not\equiv \infty$, there exists $\a\in\bbR$, $\b\geq 0$, and a positive Borel measure $\mu$ with 
\begin{equation}
\label{eq:957}
	\int_\bbR\frac{d\mu(\xi)}{1+\xi^2}<\infty,
\end{equation}
such that
\begin{equation}
\label{eq:997}
	q(z)=\a+\b z+\int_\bbR\left(\frac{1}{\xi-z}-\frac{\xi}{1+\xi^2}\right)d\mu(\xi).
\end{equation}
\end{definition}
We denote the entries of $W(t,z)$ by $w_{ij}(t,z)$, $1\le i,j \le 2$. The $J$-decreasing family $W(t,z)$ also generates the reproducing kernels
\begin{equation}\label{eq:59intro}
K(t,z,w):= \frac{w_{22}(t,z)\overline{w_{21}(t,w)}-w_{21}(t,z)\overline{w_{22}(t,w)}}{z-\overline{w}}.
\end{equation}

\begin{theorem}\label{theorem:main} 
For any continuous $J$-decreasing family of entire matrix functions $\{ W(t,z) \mid a \le  t < b\}$ in the limit point case with $W(a,z)=I$ and with $q \not\equiv \infty$, and the measure $\mu$ and kernels $\{K(t,\cdot,\cdot) \mid a \le t < b\}$ determined by \eqref{eq:997}, \eqref{eq:59intro}, for any $\xi \in\bbR$, the following are equivalent:
		\begin{enumerate}[{\rm (i)}]
			\item There exist $(\sigma_-,\sigma_+)\in [0,\infty)\times[0,\infty)\setminus\{(0,0)\}$ and a regularly varying function $\scrg$ with  index $\beta>0$ such that \eqref{eqn:muLocalBehavior} holds.
			\item $\mu$ has a unique tangent measure, which is not the Dirac measure $\delta_0$.
			\item There exists $\scrh$ regularly varying at $\infty$ with index $\rho>0$ such that uniformly on compact subsets of $\bbC\times\bbC$
			\begin{equation}\label{eqn:ScalingLimitKinftyb}
			\lim\limits_{t\to b}\frac{K\left(t,\xi+\frac{z}{\scrh(K(t,\xi,\xi))},\xi+\frac{w}{\scrh(K(t,\xi,\xi))}\right)}{K(t,\xi,\xi)}=K_\infty(z,w),
			\end{equation}
			and $K_\infty\not\equiv 1$.
\end{enumerate}
		In this case $\rho=1/\beta$,  $\scrh$ is the asymptotic inverse of $\scrg$, and $K_\infty = K_{\sigma_-, \sigma_+,\beta}$.
\end{theorem}

The results for orthogonal polynomials described above are applications of this result. 
Theorem~\ref{theorem:main} can also be applied to other models in spectral theory.

One are Schr\"odinger operators $-\frac{d^2}{dx^2} + V$ on an interval $[a,b)$ with a regular endpoint at $a$ and in the limit point case at $b$. Bulk universality limits for Schr\"odinger operators were first studied in \cite{maltsev:2010} and later proved under more general assumptions in \cite{eichinger:2024,EichLukSimanek}. We can allow potentials $V \in H^{-1}_\loc([a,b))$, via the decomposition $V = s' + q$, $s\in L^2_\loc([a,b))$, $q \in L^1_\loc([a,b))$; the most commonly discussed case $V \in L^1_\loc([a,b))$ corresponds to $s=0$.  Fixing a boundary condition $\cos \beta u(a) + \sin\beta u^{[1]}(a) = 0$, where $u^{[1]}(a) = u'(a) - s(a) u(a)$ is the quasiderivative, gives a self-adjoint operator  with a canonical choice of spectral measure $\mu$ \cite{TeschlMathMedQuant,LukicBook,LukicSukhtaievWang}. Consider the eigensolution $u(x,z)$ given by
\[
-\partial_x^2 u + V u = zu, \qquad u(a,z) = \sin \beta, \qquad u^{[1]}(a,z) = -\cos \beta
\]
and the reproducing kernels
\begin{equation}\label{eqn:SchrKernel}
K(x,z,w) = \int_a^x u(y,z) \ol{ u(y,w)} \,dy = \frac{ u(x,z) \ol{ u^{[1]}(x,w)} - u^{[1]}(x,z) \ol{  u(x,w)} }{ z - \ol w}.
\end{equation}

\begin{theorem}\label{theorem:schr}
For any Schr\"odinger operator $H$ with a potential $V\in H^{-1}_\loc([a,b))$ with a regular endpoint at $a$ and limit point at $b$, its canonical spectral measure $\mu$, and the reproducing kernels \eqref{eqn:SchrKernel}, for any $\xi \in\bbR$, the following are equivalent:
		\begin{enumerate}[{\rm (i)}]
			\item There exist $(\sigma_-,\sigma_+)\in [0,\infty)\times[0,\infty)\setminus\{(0,0)\}$ and a regularly varying function $\scrg$ with  index $\beta>0$ such that \eqref{eqn:muLocalBehavior} holds.
			\item $\mu$ has a unique tangent measure, which is not the Dirac measure $\delta_0$
			\item There exists $\scrh$ regularly varying at $\infty$ with index $\rho>0$ such that \eqref{eqn:ScalingLimitKinftyb} uniformly on compact subsets of $\bbC\times\bbC$ and $K_\infty\not\equiv 1$.
\end{enumerate}
		In this case $\rho=1/\beta$,  $\scrh$ is the asymptotic inverse of $\scrg$, and $K_\infty = K_{\sigma_-, \sigma_+,\beta}$. 
\end{theorem}

Completely analogously, Theorem~\ref{theorem:main} applies to some other settings such as Sturm--Liouville and Dirac operators.

With some additional arguments, Theorem~\ref{theorem:main} also applies to universality limits for orthogonal polynomials on the unit circle. For a comprehensive treatment of orthogonal polynomials on the unit circle we refer to Simon's monograph \cite{SimonOPUC1}. Bulk universality in this setting were established by Levin-Lubinsky \cite{Levin.LubinskyOPUC:2007} and Simanek studied universality with singularities \cite{Simanek:2017}. To state our result, let $\nu$ be a probability measure on $\partial\bbD$ such that $\supp \nu$ is not a finite set. Orthogonal polynomials $\{\varphi_n\}_{n=0}^\infty$ are obtained from the sequence $\{\zeta^n\}_{n=0}^\infty$  by the Gram--Schmidt process in $L^2(\partial\bbD,d\nu)$, and they obey
\[
\int_{\partial\bbD} \overline{ \varphi_m(\zeta) } \varphi_n(\zeta) \,d\nu(\zeta) = \delta_{m,n}.
\]
The corresponding CD kernels are defined by
\begin{equation}\label{eqn:OPUCCD}
k_n(\zeta,\omega) = \sum_{j=0}^{n-1} \varphi_j(\zeta) \ol{ \varphi_j(\omega)}.
\end{equation}

\begin{theorem}\label{theorem:OPUC}
Let $\nu$ be a probability measure on $\partial\bbD$ such that $\supp\nu$ is not finite. For any $\xi \in\bbR$, the following are equivalent:
		\begin{enumerate}[{\rm (i)}]
			\item There exists a regularly varying function $u$ with  index $\beta>0$ and $(\sigma_-,\sigma_+)\in [0,\infty)\times[0,\infty)\setminus\{(0,0)\}$ such that 
			\[
			\lim_{r\to \infty} u(r) \nu\left(\left\{e^{it} \mid t \in \left(\xi-\tfrac 1r,\xi \right) \right\} \right) =\sigma_-,\quad \lim_{r\to \infty} u(r) \nu\left(\left\{e^{it} \mid t \in \left[\xi,\xi+\tfrac 1r \right) \right\} \right) =\sigma_+.
			\]
			\item $\nu$ has a unique tangent measure at $e^{i\xi}$, which is not the Dirac measure $\delta_0$
			\item There exists a regularly varying function $f$ with index $\rho>0$ such that uniformly on compact subsets of $\bbC\times\bbC$
\begin{equation}\label{eqnTheorem5n}
			\lim\limits_{n\to\infty} e^{-in \frac{z-\ol w}{2f(k_n(e^{i\xi},e^{i\xi}))}} \frac{k_n\left(e^{i\xi+\frac{iz}{f(k_n(e^{i\xi},e^{i\xi}))}},e^{i\xi+\frac{iw}{f(k_n(e^{i\xi},e^{i\xi}))}}\right)}{k_n(e^{i\xi},e^{i\xi})}=K_\infty(z,w)
			\end{equation}
			with $K_\infty\not\equiv 1$ and
\[
\lim_{n\to\infty} \frac{ k_{n+1}(e^{i\xi},e^{i\xi}) }{ k_n(e^{i\xi},e^{i\xi})} = 1.
\]
\end{enumerate}
		In this case $\rho=1/\beta$,  $f$ is the asymptotic inverse of $u$, and $K_\infty = K_{\sigma_-, \sigma_+,\beta}$. 
\end{theorem}

\subsection*{A very brief outline of the subsequent sections}
\phantom{}\\[1mm]
In Section~\ref{sec:2}, we recall aspects of de Branges' theory of Hilbert spaces of entire functions and its relation to canonical systems. In Section~\ref{sec:3}, we study structure Hamiltonians. In Section~\ref{sec:4}, we axiomatize the notion of a chain of de Branges spaces, and develop a notion of convergence of chains of de Branges spaces. In Section~\ref{sec:5}, we relate this to the measures associated to unbounded chains. In Section~\ref{sec:7}, we apply this to study rescaling limits of reproducing kernels, culminating in the proof of Theorem~\ref{theorem:main}. In Section~\ref{section:TwoConventions}, we address different conventions in the literature and prove the application to Schr\"odinger operators (Theorem~\ref{theorem:schr}). In Section~\ref{section:OP}, we address applications to orthogonal polynomials and prove Theorems~\ref{theorem1}, \ref{theorem2}, \ref{theorem3}, \ref{theorem:OPUC}. In Section~\ref{sec:SpectralType}, we discuss the connections between bulk universality and spectral type. In Section~\ref{sec:FreudLevin}, we describe a generalization of the Freud--Levin theorem to reproducing kernels of de Branges spaces and prove Corollary~\ref{cor:ZeroDistributionPolynomial} and Theorem~\ref{thm:ZerosHardEdge}.

%
%
%
\newpage
\section{De Branges spaces and canonical systems; a reminder}
\label{sec:2}
%
%
%

This section is of preliminary nature. We recall facts of de~Branges' theory of Hilbert spaces of entire functions 
and its relation to two-dimensional canonical systems. Standard references are 
\cite{debranges:1968,romanov:1408.6022v1,remling:2018,dym.mckean:1976}.
The underlying basis for the theory of de~Branges spaces is the notion of reproducing kernel Hilbert spaces. 
Our standard reference in this context is the seminal paper \cite{aronszajn:1950}. 

All content of this section is extracted from the named references.

%
\subsection{Reproducing kernel Hilbert spaces of entire functions}
\label{sec:2-1}
%

\begin{definition}
	Let $\Omega$ be a nonempty set. 
	A Hilbert space $(\cH, \langle \cdot,\cdot\rangle)$ of complex valued functions on $\Omega$ is called a 
	\emph{reproducing kernel Hilbert space}, if for each $w \in\Omega$ the point evaluation functional 
	$F\mapsto F(w)$, $F\in\cH$ is continuous. 
\end{definition}

If $\cH$ is a reproducing kernel Hilbert space of functions on $\Omega$, there exists a unique function 
$K_\cH:\Omega\times\Omega\to\bbC$ which satisfies:
\begin{enumerate}[{\rm(i)}]
	\item For each $w\in\Omega$, $K_\cH(\cdot, w)\in \cH$;
	\item For each $F\in\cH$ and $w\in\Omega$ we have
	\begin{align*}
	F(w)=\langle F, K_\cH(\cdot, w)\rangle. 
	\end{align*}
\end{enumerate}
This function is called the \emph{reproducing kernel} of $\cH$. 

It directly follows that 
\begin{align}\label{eq:1}
\langle K_\cH(\cdot, w),K_\cH(\cdot, z)\rangle =K_\cH(z,w).
\end{align}
In particular, the norm $\Delta_\cH(w)$ of the point evaluation functional at a point $w$ is given by
\[
	\Delta_\cH(w)=\sqrt{K_\cH(w,w)}. 
\]
A function $K:\Omega\times\Omega\to\bbC$ is called a \emph{positive kernel}, if $K(z,w)=\overline{K(w,z)}$ for all
$z,w\in\Omega$, and 
for any finite collection $(z_j)_{j=1}^N\in\Omega^N$ the matrix $(K_\cH(z_i,z_j))_{i,j=1}^N$ is positive semidefinite. 
The reproducing kernel of some reproducing kernel Hilbert space always is a positive kernel, and conversely, 
for every positive kernel $K$ there exists a unique reproducing kernel Hilbert space $\cH$ so that $K$ is its 
reproducing kernel. We denote this space as $\cH(K)$. 

\begin{definition}\label{def:5}
	Let $\Omega$ be a nonempty set, and $\cH,\tilde\cH$ reproducing kernel Hilbert spaces on $\Omega$. 
	\begin{enumerate}[{\rm(i)}]
		\item We say that $\cH$ is \emph{isometrically contained} in $\tilde \cH$ and write $\cH\subseteq_i\tilde\cH$, if 
		\begin{align*}
		\forall F\in \cH: F\in\tilde\cH\wedge \|F\|_{\tilde \cH}=\|F\|_{\cH}. 
		\end{align*}
		\item\label{it:1_2}  We say that $\cH$ is \emph{contractively contained} in $\tilde \cH$ and write 
		$\cH\subseteq_c\tilde\cH$, if
		\begin{align*}
		\forall F\in \cH: F\in\tilde\cH\wedge\|F\|_{\tilde \cH}\leq \|F\|_{\cH}. 
		\end{align*}
	\end{enumerate}
\end{definition}

Note that 
\begin{align}\label{eq:2}
\cH\subseteq_c\tilde \cH\implies \forall w\in\bbC: \Delta_{\cH}(w)\leq \Delta_{\tilde\cH}(w).
\end{align}
Contractive inclusion is equivalent to a property of reproducing kernels:
For two positive kernels $K,\tilde K$ defined on the same set $\Omega$ write 
\begin{equation}
\label{eq:47}
	K\leq\tilde K
\end{equation}
if $\tilde K -K$ is a positive kernel. Then
\begin{equation}
\label{eq:16}
	\cH\subseteq_c\tilde \cH\iff K_{\cH}\leq K_{\tilde\cH}.
\end{equation}
If $\Omega\subseteq\bbC$ and $\cH$ is a reproducing kernel Hilbert space of functions on $\Omega$, we have the operator of 
multiplication by the independent variable defined on its natural maximal domain $\{F\in\cH\mid zF(z)\in\cH\}$. 
We denote the closure of this domain as
\begin{equation}
\label{eq:18}
	\clH:=\clos\big(\{F\in\cH\mid zF(z)\in\cH\}\big).
\end{equation}
The following partial order, which lies in between contractive and isometric inclusion, is crucial. 

\begin{definition}
	Let $\Omega\subseteq\bbC$ be a nonempty set and $\cH, \tilde \cH$ reproducing kernel Hilbert spaces on $\Omega$. 
	We say that $\cH$ is \emph{almost isometrically contained} in $\tilde\cH$ and write $\cH\sqsubseteq\tilde\cH$, if 
	\begin{align*}
	\cH\subseteq_c\tilde\cH\quad \wedge\quad\clH\subseteq_i\tilde\cH. 
	\end{align*}
\end{definition}

Now we turn our attention to reproducing kernel Hilbert spaces whose elements are analytic functions. 
For an open and nonempty subset $\Omega\subseteq\bbC^n$ we denote
\[
\Hol(\Omega):=\{F:\Omega\to \bbC\mid F\text{ is analytic in }\Omega\}, 
\]
and endow $\Hol(\Omega)$ with the topology of locally uniform convergence. Recall that this topology is metrizable: 
Let $S_n\subseteq\Omega$, $n\in\bbN_0$, be compact such that $S_n \subseteq \Int S_{n+1}$ and $\bigcup_{n\in\bbN_0} S_n = \Omega$. 
Then $\Hol(\Omega)$ becomes a Fr\'echet space with the metric
\begin{align}\label{eq:41}
d(f,g) := \sum_{n=1}^\infty 2^{-n} \min \big\{ 1, \sup_{z\in S_n}|f(z)-g(z)|\big\},
\end{align}
and this metric induces locally uniform convergence. 

\begin{definition}
	If $\cH$ is a reproducing kernel Hilbert space on $\Omega$ and $\cH\subseteq\Hol(\Omega)$, then we call $\cH$ a
	\emph{reproducing kernel Hilbert space of analytic functions} on $\Omega$. 
	We denote the set of all reproducing kernel Hilbert spaces of analytic functions on $\Omega$ by $\RK(\Omega)$.
	If $\Omega=\bbC$ we speak of a \emph{reproducing kernel Hilbert space of entire functions} and write $\RK$ 
	for the set of all such spaces.
\end{definition}

Analyticity of the elements of a reproducing kernel Hilbert space $\cH$ can be characterized in terms of its reproducing kernel
$K_\cH$: we have $\cH\in\RK(\Omega)$ if and only if $K_\cH(z,\overline{w})\in\Hol(\{(z,w)\mid z,\overline w\in\Omega\})$.
In particular, for $\cH\in\RK(\Omega)$, the norm of the point evaluation functional is locally bounded, and hence convergence in
the norm of $\cH$ implies convergence in $\Hol(\Omega)$. 

The map 
\begin{align*}
\cI:\left\{\begin{array}{ccc}
\RK(\Omega)&\to &\Hol(\{(z,w)\mid z,\overline w\in\Omega\})\\[2mm]
\cH&\mapsto & ((z,w)\mapsto K_\cH(z,\overline{w}))
\end{array}
\right.
\end{align*}
is injective, and we topologize $\RK(\Omega)$ by pulling back the metric from \cref{eq:41} from 
$\Hol(\{(z,w)\mid z,\overline w\in\Omega\})$ to $\RK(\Omega)$ via the map $\cI$.
Thus convergence of spaces means locally uniform convergence of their reproducing kernels. 
Obviously, the set of positive kernels is closed under locally uniform (even pointwise) limits, and thus 
$\RK(\Omega)$ is a complete metric space.

\begin{lemma}\label{lem:8}
	Let $\Omega\subseteq\bbC$ be open and nonempty and $\tilde\cH\in\RK(\Omega)$. Then 
	\begin{align}\label{eq:17}
	\{\cH\in\RK(\Omega)\mid \cH\subseteq_c\tilde\cH\}
	\end{align}
	is compact.
\end{lemma}
\begin{proof}
	Consider the set 
	\[
	M:=\{K\mid K \text{ is a positive kernel on }\Omega, K\leq K_{\tilde \cH} \}.
	\]
	By \cref{eq:16} it follows that
	\[
	\cI(\{\cH\in\RK(\Omega)\mid \cH\subseteq_c\tilde\cH\})=M.
	\]
	Since $\cI$ is a homeomorphism, it suffices to show that $M$ is compact. 
	
	For $K\in M$ the Cauchy-Schwarz inequality gives
	\[
	|K(z,w)|^2\leq K(z,z)K(w,w)\leq K_{\tilde\cH}(z,z)K_{\tilde\cH}(w,w),
	\]
	and Montel's theorem implies that $M$ is a normal family. 
	Since the inequality in the definition of $M$ is preserved by taking limits, $M$ is also closed and thus compact. 
\end{proof}

\begin{example}
	Two classical examples of reproducing kernel Hilbert spaces of entire functions 
	are the following.
	\begin{enumerate}[{\rm(i)}]
	\item Let $a>0$. The \emph{Paley-Wiener space} $\cP\cW_a$ is the space of all entire functions of 
		exponential type at most $a$ which are square integrable on $\bbR$ endowed with the $L^2(\bbR)$-scalar product. 
		It follows by direct verification using the Paley-Wiener theorem and properties of the Fourier transform 
		that $\cP\cW_a\in\RK$. Moreover, $(\cP\cW_a)^\flat=\cP\cW_a$.
	\item For $m\in\bbN_0$ we denote by $\cP_m$ the set of polynomials of degree at most $m$, and formally set 
		$\cP_{-1}:=\{0\}$.
		Let $n\in\bbN_0$ and $\mu$ a positive Borel measure on $\bbR$ which has at least $n$ finite moments and whose
		support contains at least $n+1$ points. For each $m\leq n$ the space $\cP_m$ endowed 
		with the $L^2(\mu)$-scalar product belongs to $\RK$. Moreover, $(\cP_n)^\flat=\cP_{n-1}$.
	\end{enumerate}
\end{example}

%
\subsection{De~Branges spaces}
\label{sec:2-2}
%

De~Branges spaces are reproducing kernel Hilbert spaces of entire functions that satisfy certain additional axioms. 

Throughout the following, we denote for an entire function $F$
\begin{equation}
\label{eq:979}
F^\#(z):=\overline{F(\overline z)},
\end{equation}
and say that $F$ is a \emph{real entire function} if $F=F^\#$. 

\begin{definition}\label{def:1}
	A \emph{de~Branges space} $\cH$ (\emph{dB-space}, for short) is a Hilbert space which satisfies: 
	\begin{enumerate}[{\rm(i)}]
		\item\label{it:2-1} $\cH\in \RK\setminus \{\{0\}\}$;
		\item\label{it:2-2} For each $F\in\cH$, also $F^\#\in\cH$ and $\|F\|_\cH=\|F^\#\|_\cH$;
		\item\label{it:2-3} If $w\in\bbC\setminus\bbR$ and $F\in\cH$ with $F(w)=0$, then 
		\begin{align*}
		\frac{F(z)}{z-w}\in \cH,\quad\text{and}\quad \left\|\frac{z-\overline{w}}{z-w}F(z)\right\|_\cH=\|F\|_\cH.
		\end{align*}
		Note here that $\frac{z-\overline{w}}{z-w}F(z) = F(z) + \frac{w- \overline{w}}{z-w} F(z) \in \cH$.
	\end{enumerate}	
	We denote the set of all dB-spaces by $\DB$. The set of all those dB-spaces which satisfy in addition:
	\begin{enumerate}
		\item[(iv)]\label{it:2-4}If $w\in\bbR$ and $F\in\cH$ with $F(w)=0$, then 
		\[
		\frac{F(z)}{z-w}\in\cH;
		\]
	\end{enumerate}
	is denoted as $\DB^\ast$.
\end{definition}

Those subspaces of a dB-space which are with the inner product inherited from $\cH$ 
themselves dB-spaces play an outstanding role and are discussed in detail in \Cref{sec:3}. We call such a subspace a 
\emph{dB-subspace} of $\cH$. 

In this place we only observe the following property: if $\cL$ is a linear subspace of $\cH$ which is closed under the
operations in \Cref{def:1}(ii),(iii), then the closure of $\cL$ in $\cH$ is a dB-subspace of $\cH$. 
This has two consequences, which we also state explicitly:
\begin{enumerate}[{\rm(i)}]
\item For every dB-space $\cH$, the space $\cH^\flat$ is a dB-subspace of $\cH$.
\item If $(\cH,\langle\cdot,\cdot\rangle_\cH)$ is a dB-space, $\cL$ is a closed linear subspace of $\cH$, and 
	$\langle\cdot,\cdot\rangle_\cL$ is a scalar product on $\cL$ such that $(\cL,\langle\cdot,\cdot\rangle_\cL)$ is a
	dB-space, then $(\cL,\langle\cdot,\cdot\rangle_\cH)$ is a dB-space. 
\end{enumerate}
In this context, let us also recall that for a dB-space $\cH$
\begin{equation}\label{eq:20}
	\dim\big(\cH/\cH^\flat\big)\leq 1.
\end{equation}

A dB-space can be generated from one single entire function. This follows since the reproducing kernel of a dB-space is of a
particular form. To explain the connection, recall the notion of Hermite-Biehler functions. 

\begin{definition}
	A \emph{Hermite-Biehler function} is an entire function $E$ which satisfies
	\begin{align}\label{eq:37}
	\forall z\in\bbC_+: |E(\overline z)|< |E(z)|.
	\end{align}
	We denote the set of all Hermite-Biehler functions by $\HB$. 
	The set of all those Hermite-Biehler functions which have no real zeros is denoted by $\HB^\ast$. 
\end{definition}

For an entire function $E$ we denote its real and imaginary part in the sense of the involution \cref{eq:979} by 
\[
A:=\frac{E+E^\#}{2},\quad B:=i\frac{E-E^\#}{2}.
\]
Then $A=A^\#$, $B=B^\#$, and $E=A-iB$. In particular, the assignment $E\mapsto(A,\,B)$ is injective. 
We freely apply the convention that $E,A,B$ are related in this way, if the meaning is clear from the context.
Another useful notation in this context is the following: if $E$ is entire and $M$ is a $2\times 2$-matrix function with real
entire entries, we set 
\[
E\ltimes M:=\tilde A-i\tilde B\quad\text{where}\quad (\tilde A,\,\tilde B):=(A,\, B)M
.
\]
Given $E\in \HB$ we define 
\begin{align}\label{eq:25}
K_E(z,w):=\frac i{2}\frac{E(z)E^\#(\overline w)-E(\overline w)E^\#(z)}{z-\overline w}
,
\end{align}
which should be appropriately interpreted in terms of derivatives if $z=\overline{w}$. 
Using the functions $A,B$, the kernel $K_E$ writes as 
\begin{equation}
\label{eq:978}
K_E(z,w)=\frac{(A(z), B(z)) J\begin{pmatrix} A(\overline{w})\\ B(\overline{w})\end{pmatrix}}{z-\overline{w}}
=\frac{B(z)\overline{A(w)}-A(z)\overline{B(w)}}{z-\overline{w}},
\end{equation}
where $J=\big(\begin{smallmatrix} 0&-1\\ 1&0 \end{smallmatrix}\big)$. 

In the following theorem we summarize the connection between dB-spaces and Hermite-Biehler functions. 
Here we endow $\HB$ with the subspace topology of $\Hol(\bbC)$, and $\DB$ with the subspace topology of $\RK$. 

\begin{theorem}\label{prop:1}
	For any $E\in\HB$ the function $K_E$ is a positive kernel and the reproducing kernel space $\cH(K_E)$ 
	generated by $K_E$ is a dB-space. 
	
	Let $\cB:\HB\to\DB$ be the map $E\mapsto\cH(K_E)$. Then
	\begin{enumerate}[{\rm(i)}]
		\item\label{it:4-1} $\cB$ is surjective;
		\item\label{it:4-2} $\cB(E_1)=\cB(E_2)$ if and only if there exists $M\in \SL(2,\bbR)$ such that
		$E_2=E_1\ltimes M$;
		\item\label{it:4-3} $\cB(\HB^\ast)=\DB^\ast$;
		\item\label{it:4-4} $\cB$ is continuous and has a continuous right inverse.
	\end{enumerate}
\end{theorem}

In this context note the formula 
\[
E(z):=-i\sqrt{\frac{1}{K(i,i)}}(z+i)K(z,i),
\]
which determines one possible choice of $E$ given the kernel $K(.,.)$ of the space. 
This formula also implies that the closure of $\DB$ in $\RK$ is equal to $\DB\cup\{\{0\}\}$. 
Also note the formula for the norm of point evaluation at a nonreal point $w$:
\[
\Delta_{\cB(E)}(w)= \sqrt{K_E(w,w)} = \Big(\frac{|E(w)|^2-|E(\overline w)|^2}{4\Im w}\Big)^{\frac 12}
.
\]
The \emph{real zero divisor} of an entire function $F$ which does not vanish identically is the function 
$\vartheta_F:\bbR\to \bbN_0$ defined by 
\[
\vartheta_F(x):=\min\{n\in\bbN_0\mid F^{(n)}(x)\neq 0\}.
\]
For $\cH\in \DB$ we set 
\[
\vartheta_{\cH}:=\vartheta_{E},
\]
where $E\in\HB$ is such that $\cB(E)=\cH$. It follows from \Cref{prop:1}(ii) that this definition does not 
depend on the choice of $E$. It holds that 
\[
\vartheta_{\cH}(x)=\min\big\{\vartheta_F(x)\mid F\in\cH\setminus\{0\}\big\}
.
\]
For many purposes, it suffices to study Hermite-Biehler functions without real zeros, due to the following simple fact.

\begin{lemma}\label{lem:26}
	Let $E\in \HB$ and $C$ a real entire function without nonreal zeros such that $\vartheta_C\leq\vartheta_E$ (here 
	``$\leq$'' is understood pointwise). Then, $\frac{E}{C}\in\HB$ and the map 
	\begin{align*}
		\left\{
		\begin{array}{ccc}
			\cB(E) & \to & \cB\left(\frac{E}{C}\right)
			\\[2mm]
			F & \mapsto & \frac{F}{C}
		\end{array}
		\right.
	\end{align*}
	is an isometric isomorphism. We have 
	\[
		\cB\big(\tfrac EC\big)=\tfrac 1C\cB(E),\qquad
		\tfrac 1C\big(\cB(E)^\flat\big)=\big(\cB\big(\tfrac EC\big)\big)^\flat
		.
	\]
\end{lemma}
\begin{proof}
	That $\frac{E}{C}\in\HB$ is clear. The other assertions follow directly from 
	\[
	K_E(z,w)=C(z)K_{E/C}(z,w)\overline{C(w)}.
	\]
\end{proof}

Similar to the scalar case, a function $K:\Omega\times\Omega\to\bbC^{n\times n}$ is called a \emph{positive kernel}, 
if $K(z,w)=K(w,z)^*$ for all
$z,w\in\Omega$, and for any finite collections $(z_j)_{j=1}^N\in\Omega^N$ and $(a_j)_{j=1}^N\in(\bbC^n)^N$, the matrix 
$(a_j^*K_\cH(z_i,z_j)a_i)_{i,j=1}^N$ is positive semidefinite. 

Recall \Cref{def:JInner} of $J$-inner matrix functions.

\begin{theorem}
\label{thm:996}
	The following statements hold.
	\begin{enumerate}[{\rm(i)}]
	\item Let $E\in\HB$, $W$ be $J$-inner, and set $\tilde E:=E\ltimes W$. Then
		\[
			\tilde E\in\HB,\ \vartheta_{\tilde E}=\vartheta_E,\quad\cB(E)\subseteq_c\cB(\tilde E)
			.
		\]
	\item If $E,\tilde E\in\HB$, $\vartheta_{\tilde E}=\vartheta_E$, and $\cB(E)\subseteq_i\cB(\tilde E)$, 
		then there exists a unique $J$-inner matrix function $W$ such that $\tilde E=E\ltimes W$. 
	\end{enumerate}
\end{theorem}

\begin{example}
\label{917}
	The simplest example for nonconstant $J$-inner matrix functions are linear polynomials of a specific form. 
	For $\a\in\bbR$ we denote 
	\begin{align}\label{eq:28}
	e_\a=\begin{pmatrix}
	\cos\a\\\sin\a
	\end{pmatrix}.
	\end{align}
	Then, for $\a\in\bbR$ and $\ell\geq 0$, the matrix function $W:=I-\ell ze_\a e_\a^*J$ is $J$-inner. 
	
	Assume that $E\in \HB$. Then by a direct computation 
	\begin{align}\label{eq:61}
	K_{E\ltimes W}(z,w)=K_E(z,w)+\ell\cdot \big[(A(z),\ B(z)) e_\a\big]\big[(A(w),\ B(w)) e_\a\big]^*.
	\end{align}
	Provided that $G(z):=A(z)\cos\a+B(z)\sin\a\notin\cB(E)$, we have
	\begin{equation}\label{eq:50}
		\begin{aligned}
			& \cB(E\ltimes W)=\cB(E)\oplus\spann\{G\} \ \text{and}\ \|G\|_{\cB(E\ltimes W)}=\frac{1}{\sqrt{\ell}},
			\\
			& \cB(E\ltimes W)^\flat=\cB(E).
		\end{aligned}
	\end{equation}
	If $G\in\cB(E)$, then the inclusion map of $\cB(E)$ into $\cB(E\ltimes W)$ is bijective and a nonisometric contraction. 
\end{example}

$J$-inner matrix functions are related to de~Branges spaces also in another way. Assume we have 
$W=(w_{ij})_{i,j=1,2}$ with these properties, then the function 
\begin{equation}\label{eq:995}
E:=w_{22}+iw_{21}
\end{equation}
is a Hermite-Biehler function. In fact, a $J$-inner matrix function generates a whole family of Hermite-Biehler functions, but
\eqref{eq:995} is the one we work with. 

%
\subsection{Canonical systems}
\label{sec:2-3}
%

In this subsection we recall some facts and basic definitions about canonical systems.

\begin{definition}\label{def:899}
	Let $-\infty\leq a<b\leq\infty$, and let $H:(a,b)\to\bbR^{2\times 2}$ be a function with 
	\begin{equation}
	\label{eq:959}
		H\in L^1_\loc((a,b)),\qquad H(t)\geq 0\text{ for a.a.\ }t\in(a,b)
		.
	\end{equation}
	\begin{enumerate}[{\rm(i)}]
	\item We say that $H$ is in \emph{limit circle case} at the endpoint $a$ (\emph{lc} at $a$, for short), if for one 
		(and hence for all) $c\in (a,b)$ 
		\[
			\int_a^c\tr H(t)dt<\infty. 
		\]
		Otherwise, $H$ is in \emph{limit point case} at $a$ (\emph{lp} at $a$, for short). 
		Analogous definitions apply to the endpoint $b$. 
	\item We say that $H$ is a \emph{Hamiltonian}, if $H(t)\neq 0$ for a.a. $t\in(a,b)$. 
		The set of all Hamiltonians defined on the interval $(a,b)$ is denoted as $\bbH_{a,b}$. 
	\item We call $H$ \emph{trace normalized}, if $\tr H(t)=1$ for a.a.\ $t\in(a,b)$, and denote the set of all 
		such functions as $\bbHe_{a,b}$. 
	\item We denote the set of all functions $H$ which satisfy $\tr H(t)\leq 1$ for a.a.\ 
		$t\in(a,b)$ as $\bbHle_{a,b}$.
	\end{enumerate}
\end{definition}

The set $\bbHle_{a,b}$ can be topologized in a natural way, see for example the exposition in 
\cite{pruckner.woracek:limp}. This topology is compact and metrizable, for example in the following way: 
choose sequences $(a_n)_{n\in\bbN}$ and $(b_n)_{n\in\bbN}$ such that $a<a_n<b_n<b$ for all $n\in\bbN$ and 
$\lim_{n\to\infty}a_n=a$, $\lim_{n\to\infty}b_n=b$, and use the metric
\begin{equation}
\label{eq:26}
	d(H_1,H_2):=\sum_{n=1}^\infty2^{-n}\min\left\{1,
	\sup_{a_n\leq c\leq d\leq b_n}\bigg\|\int_c^d\big(H_1(t)-H_2(t)\big)dt\bigg\|
	\right\},
\end{equation}
for $H_1,H_2\in\bbHle_{a,b}$.

With a function $H$ subject to \cref{eq:959} we associate a differential equation for a $2$-vector valued function $y$, namely
\begin{equation}\label{eq:51}
	J\partial_xy(x)=-zH(x)y(x), \quad x\in(a,b)\text{ a.e.},
\end{equation}
which is called a \emph{canonical system}.
For $a<c\leq x<b$, we define the \emph{transfer matrix} at $c$ as the matrix solution of the initial value problem 
\begin{equation}
\label{eq:9}
	\partial_xW_H(c,x,z)J=zW_H(c,x,z)H(x),\quad 
	W_H(c,c,z)=I.
\end{equation} 
Observe that \cref{eq:9} is transposed compared to \cref{eq:51}. We use \cref{eq:9} since this is practical in many respects
and was the convention that de~Branges used in \cite{debranges:1968} on which we heavily rely in this paper. 
Note that, by uniqueness of solutions of the above differential equation, transfer matrices are multiplicative in the sense that 
\begin{equation}
\label{eq:948}
	W(c,x,.)=W(c,d,.)W(d,x,.),\quad d\in[c,x]
	.
\end{equation}
Assume that $a$ is a lc endpoint. Then the solutions of \cref{eq:51} can be continuously extended to $a$ and we define the 
\emph{fundamental solution} associated to $H$ by 
\begin{equation}
\label{eq:38}
	W_H(x,z):=W_H(a,x,z).
\end{equation}
Based on \cref{eq:995} we obtain a family of Hermite-Biehler functions with no real zeros. Namely, writing 
$W_H(t,z)=(w_{ij}(t,z))_{i,j=1,2}$, we have 
\begin{equation}
\label{eq:966}
	E(t,z):=w_{22}(t,z)+iw_{21}(t,z)\in\HB^*.
\end{equation}
Moreover, we introduce the notation 
\begin{equation}\label{eq:59}
K_H(t,z,w):=K_{E(t,\cdot)}(z,w)
=\frac{w_{22}(t,z)\overline{w_{21}(t,w)}-w_{21}(t,z)\overline{w_{22}(t,w)}}{z-\overline{w}},
\end{equation}
cf. \cref{eq:978}. Note that the kernels $K_H(t,.,.)$ depend continuously on $t$. Moreover, if $H$ is lp at $b$, then 
\begin{equation}
\label{918}
	\forall z\in\bbC\setminus\bbR: \lim_{t\to b}K_H(t,z,z)=\infty
	,
\end{equation}
cf.\ \cite{debranges:1968}. Since $K_H(t,z,z)$ is nondecreasing in $t$ for each fixed $z$, Dini's theorem implies that 
this limit is attained uniformly on every compact subset of $\bbC\setminus\bbR$. 

\begin{definition}
	Let $\bbTM$ denote the set of all matrix functions $W$ which are $J$-inner and satisfy the 
	normalization condition $W(0)=I$. Moreover, we define a function $\mathfrak t:\bbTM\to[0,\infty)$ as 
	\[
		\mathfrak{t}(W):=\tr\big(\partial_zW(0)J\big).
	\]
\end{definition}

We equip $\bbTM$ with the product topology of $\Hol(\bbC)$ in every component. Note that $\bbTM$ is closed in the
space of all entire matrix functions. 
For all $W\in\bbTM$ we have $\det W=1$ since $\det W$ is real, entire, $(\det W(z))^2=1$ for $z\in\bbR$ and $\det W(0)=1$.

\begin{lemma}\label{lem:7}
	For any $c>0$, the set 
	\[
	\bbTM_c:=\{W\in \bbTM\mid\mathfrak{t}(W)\leq c \}
	\]
	is compact. 
\end{lemma}
\begin{proof}
	A calculation shows that $\{\|W(z)\|\mid W\in\bbTM_c\}$ is uniformly bounded on compact subsets, 
	and thus $\bbTM_c$ is a normal family by Montel's theorem. It is closed because $\bbTM$ is closed and 
	$\mathfrak{t}$ is continuous.
\end{proof}

Let $H:(a,b)\to\bbR^{2\times 2}$ be a function subject to \cref{eq:959}. 
Differentiating the form $W_H(c,x,z)JW_H(c,x;w)^*$, which is possible a.e., yields 
\begin{align}\label{eq:8}
	W_H(c,d,z)JW_H(c,d,w)^*-J=(z-\overline{w})\int_c^dW_H(c,s,z)H(s)W_H(c,s,w)^*ds
	.
\end{align}
Since $H\geq 0$, this shows that $W_H(c,d,z)$ is $J$-inner whenever $c\leq d$. Clearly, all entries of a transfer matrix are
real and entire functions in the variable $z$, and $W_H(c,d,0)=I$. Hence $W_H(c,d,z)\in\bbTM$. The meaning of the function
$\mathfrak t$ in this context is that 
\begin{align}\label{eq:15}
	\mathfrak{t}\big(W_H(c,d,.)\big)=\int_c^d\tr H(x)dx
\end{align}
which follows from \cref{eq:9} and the fact that $W(c,x,0)=I$.

The following is a fundamental result. 

\begin{theorem}\label{thm:958}
	For each $c>0$ the map
	\[
		\left\{
		\begin{array}{ccc}
			\bbHe_{0,c} & \to & \bbTM_c
			\\[2mm]
			H & \mapsto & W_H(c,\cdot)
		\end{array}
		\right.
	\]
	is a homeomorphism. 
\end{theorem}

Making a change of variable in the time-parameter will not affect essential properties of the solution of a canonical system. 
To formalize this, the following notion is used.

\begin{definition}
	Let $H_1\in\bbH_{a_1,b_1}$ and $H_2\in\bbH_{a_2,b_2}$. 
	We say that $H_2$ is a \emph{reparametrization} of $H_1$ and write $H_2\sim  H_1$, if there exists an increasing 
	bijection $\g:(a_2,b_2)\to(a_1,b_1)$ such that $\g$ and $\g^{-1}$ are absolutely continuous and
	\begin{equation}\label{eq:999}
	H_2(t)=H_1(\g(t))\g'(t).
	\end{equation}
\end{definition}

Fundamental solutions behave well when performing a reparameterization:
a direct computation shows that \cref{eq:999} implies 
\begin{align}\label{eq:35}
W_{H_2}(c,d,z)=W_{H_1}(\gamma(c),\gamma(d),z)
\end{align}
for $a_2<c\leq d<b_2$. 

Clearly, the relation $\sim$ is an equivalence relation on the set $\bigcup_{-\infty\leq a<b\leq\infty}\bbH_{a,b}$ of all
Hamiltonians. Each equivalence class contains a trace normalized element. In fact, pick $c\in (a,b)$ and use 
\begin{equation}\label{eq:998}
\g(t)=\int_c^t\tr H(s)ds
\end{equation}
in \cref{eq:999} (with the convention $\int_c^t = - \int_{[t,c)}$ if $t < c$).

\begin{definition}
	Let $H:(a,b)\to\bbR^{2\times 2}$ be subject to \cref{eq:959}. 
	A nonempty interval $(c,d)\subseteq(a,b)$ is called \emph{indivisible} for $H$, 
	if for some angle $\a\in\bbR$ (recall \cref{eq:28})
	\[
	H(t)=\tr H(t)\cdot e_\a e_\a^*,\quad t\in(c,d).
	\]
	The number $\ell:=\int_c^d\tr H(t)dt$ is called the \emph{length} of the indivisible interval. 
	Unless $H(t)=0$ a.e.\ on $(c,d)$, i.e.\ $l=0$, the angle $\a$ is determined up to integer multiples of 
	$\pi$ and is called the \emph{type} of the indivisible interval.
	
	A point $t\in(a,b)$ is called \emph{regular}, if it is not an interior point of an indivisible interval. A point which
	is not regular is called \emph{singular}. We denote the set of all regular points of $H$ as $I_{\rm reg}$.
\end{definition}

Transfer matrices of indivisible intervals are linear polynomials:  
if $(c,d)$ is indivisible with length $\ell>0$ and of type $\a$, then 
\begin{equation}\label{eq:30}
W_H(c,d,z)=I-\ell ze_\a e_\a^*J=\begin{pmatrix}
1-\ell z\cos\a\sin\a&\ell z(\cos\a)^2\\-\ell z(\sin\a)^2&1+\ell z\cos\a\sin\a
\end{pmatrix}.
\end{equation}
Recall that we have already met matrices of this form in \Cref{917}. 

The following simple transformation rule for canonical systems is often practical.

\begin{lemma}\label{lem:6}
	Let $H:(a,b)\to\bbR^{2\times 2}$ be subject to \cref{eq:959}, let $M\in \SL(2,\bbR)$, and set
	\[
	(\cT_MH)(t):=MH(t)M^{*}.
	\]
	Then, for any $a<c\leq t<b$ and $z\in\bbC$,
	\[
	W_{\cT_MH}(c,t,z)=MW_{H}(c,t,z)M^{-1}.
	\]
\end{lemma}

%
\subsection{Canonical systems and Nevanlinna functions}
\label{sec:2-4}
%

We say that a function $q$ is a \emph{Nevanlinna function} (in the literature also called \emph{Herglotz function}), 
if it is defined and analytic in the open upper half-plane $\bbC_+$ and maps this half-plane into $\bbC_+\cup\bbR$. 
The set of all Nevanlinna functions is denoted by $\cN_0$. Often the class of Nevanlinna functions is augmented by the function
which is identically equal to infinity, and considered as a subclass of the analytic functions of $\bbC_+$ into the Riemann
sphere. We equip $\cN_0\cup\{\infty\}$ with the topology of local uniform convergence which is metrizable. 

For $q\in\cN_0$ there exists $\a\in\bbR$, $\b\geq 0$, and a positive Borel measure $\mu$ with \eqref{eq:957} such that \eqref{eq:997} holds. Conversely, every function of this form belongs to $\cN_0$. 
Note that the integral representation \cref{eq:997} can be rewritten as 
\[
	q(z)=\a+\b z+\int_{\bbR}\frac{1+\xi z}{\xi-z}\frac{d\mu(\xi)}{1+\xi^2}.
\]
The data $\a,\b,\mu$ in this integral representation is uniquely determined by $q$. 
First, the Stieltjes inversion formula says that 
\[
	\forall a,b\in\bbR,\,a<b:
	\mu((a,b))=\lim_{\d\to 0}\lim_{\e\to 0}\frac{1}{\pi}\int_{a+\d}^{b-\d}\Im q(t+i\e)dt,
\]
and hence $\mu$ is determined by $q$. Now $\a,\b$ can be computed as 
\[
	\a=\Re q(i),\quad \b=\Im q(i)-\int_\bbR\frac{d\mu(\xi)}{1+\xi^2}<\infty.
\]
Convergence of Nevanlinna functions translates to convergence of the data in their integral representation. 
This result is known as the Grommer-Hamburger theorem. 
Let $(q_j)_{j\in J}$ be a net in $\cN_0$, $q\in\cN_0$ and $\a_j,\b_j,\mu_j$ and $\a,\b,\mu$ be the data in the respective integral representations. 
Then 
\begin{align}
\label{eq:968}
\lim_{j\in J} q_j= q & \iff
\\
\nonumber
& \lim_{j\in J} \a_j=\a,
\\
\nonumber
& \lim_{j\in J}\Big(\b_j+\int_\bbR\frac{d\mu_j(\xi)}{1+\xi^2}\Big)=\b+\int_\bbR\frac{d\mu(\xi)}{1+\xi^2},
\\
\nonumber
& \lim_{j\in J}\mu_j=\mu,
\end{align}
where the limit of measures is understood in the $w^*$-topology of $C_c(\bbR)^*$. The Grommer-Hamburger theorem is typically formulated for sequences. However, since the conditions on $\mu_j$ imply that its total mass is eventually uniformly bounded, one can conclude the Grommer-Hamburger theorem in the stated version, see \Cref{thm:950b} and \Cref{lem:950b}.

Let us now explain the relation of Nevanlinna functions and canonical systems.
Recall the notation \eqref{eqn:Mobius} for fractional linear transformations. Let $H\in\bbH_{a,b}$ such that $a$ is lc and $b$ is lp. Then for every family $(\tau_t)_{t\in(a,b)}$ with 
$\tau_t\in\cN_0\cup\{\infty\}$ the limit 
\begin{align}\label{eq:36}
q_H(z):=\lim\limits_{t\to b}\Big[W_H(t,z)\star\tau_t(z)\Big]
\end{align}
exists, and its value is independent of the parameter family $(\tau_t)_{t\in(a,b)}$. 
The function $q_H$ either is a Nevanlinna function or identically equal to $\infty$. 

\begin{definition}
	Let $H\in\bbH_{a,b}$ such that $a$ is lc and $b$ is lp. 
	The function $q_H$ defined by \cref{eq:36} is called the \emph{Weyl coefficient} of $H$. 
	The measure in the integral representation of $q_H$ is called the \emph{spectral measure} of $H$, 
	and we denote it by $\mu_H$. 
\end{definition}

The relation \cref{eq:35} shows that $H_1\sim H_2$ implies $q_{H_1}=q_{H_2}$, and hence also $\mu_{H_1}=\mu_{H_2}$. 
Therefore, one can for many purposes restrict attention to $\bbHe_{0,\infty}$. 

The relation \cref{eq:36} establishes a map from Hamiltonians to functions.
The following is a fundamental result known as de~Branges' inverse spectral theorem. 

\begin{theorem}\label{thm:8}
	The map
	\[
	\left\{
	\begin{array}{ccc}
	\bbHe_{0,\infty}&\to& \cN_0\cup\{\infty\}
	\\[2mm]
	H&\mapsto& q_H
	\end{array}\right.
	\]
	is a homeomorphism. 
\end{theorem}

By the Grommer-Hamburger theorem convergence of Hamiltonians also implies convergence of spectral measures.

%
\subsection{Power bounded measures and generalized Nevanlinna functions}
\label{sec:2-5}
%

We already discussed the connection between Nevanlinna functions and Poisson integrable measures in \Cref{sec:2-4}:
the formula \cref{eq:997} establishes a bijection between the set $\cN_0$ and the set of all pairs $(\mu,p)$ where 
$\mu$ is a positive measure on the real line with \cref{eq:957} and $p(z)=\a+\b z$ with $\a\in\bbR$ and $\b\geq 0$. 
This correspondence has an analogy for a class of functions larger than $\cN_0$ and a 
class of measures being not anymore Poisson integrable.

To define this class of functions, we have to move away from positive definiteness, and consider sign-indefinite kernel
functions. 
Recall that a function $K:\Omega\times\Omega\to\bbC$ is called a \emph{kernel with $\kappa$ negative squares}, 
if $K(z,w)=\overline{K(w,z)}$ for all $z,w\in\Omega$, for any finite collection $(z_j)_{j=1}^N\in\Omega^N$ the matrix 
$(K_\cH(z_i,z_j))_{i,j=1}^N$ has at most $\kappa$ negative eigenvalues (counted according to their multiplicities), 
and for at least one choice of $(z_j)_{j=1}^N\in\Omega^N$ this bound is attained. 

Proofs of what we state below and more details can be found in \cite{langer.woracek:gpinf,langer.woracek:kara}.

\begin{definition}
\label{def:956}
	Let $\k\in\bbN_0$. We denote by $\cN_\k^{(\infty)}$ the set of all functions $q$ which are analytic in $\bbC_+$, 
	for which 
	\begin{equation}
	\label{eq:952}
		\frac{q(z_i)-\overline{q(z_j)}}{z_i-\overline{z_j}}
	\end{equation}
	is a kernel with $\k$ negative squares, and which satisfy
	\begin{equation}
	\label{eq:953}
		\lim_{y\to+\infty}\bigg|\frac{q(iy)}{y^{2\kappa-1}}\bigg| = \infty
		\qquad\text{or}\qquad
		\lim_{y\to+\infty}\frac{q(iy)}{(iy)^{2\kappa-1}}\in(-\infty,0).
	\end{equation}
	Moreover, we set $\cN_{\le\kappa}^{(\infty)}:=\bigcup\limits_{\kappa'=0}^\kappa \cN_{\kappa'}^{(\infty)}$
	and $\cN_{<\infty}^{(\infty)}:=\bigcup\limits_{\kappa'=0}^\infty \cN_{\kappa'}^{(\infty)}$.
\end{definition}

It follows from known properties of the asymptotic behaviour of Nevanlinna functions (e.g.\ \cite{kac.krein:1968a}) that 
$\cN_0^{(\infty)}=\cN_0$. One should view $\cN_{<\infty}^{(\infty)}$ as a sign-indefinite generalization of $\cN_0$ still
retaining analyticity in $\bbC_+$ and sign-definite behaviour along $\bbR$. 
This is ensured by the condition \cref{eq:953} which means that the sign-indefinite contribution to $q$ is concentrated 
at $\infty$ (for details see \cite{langer.woracek:ninfrep}).

\begin{definition}
\label{def:955}
	Let $\k\in\bbN_0$. For a positive Borel measure $\mu$ on $\bbR$ we set 
	\begin{equation}
	\label{908}
		\|\mu\|_\k:=\int_{\bbR}\frac{d\mu(t)}{(1+t^2)^{\k+1}}\in[0,\infty]
		.
	\end{equation}
	Moreover, let 
	\begin{align*}
		& \bbM_{\leq\k}=\big\{\mu\,\big|\,
		\mu\text{ positive Borel measure on }\bbR,\ \|\mu\|_\k<\infty\big\}
		,
		\\
		& \bbM_{<\infty}:=\bigcup\limits_{\k\in\bbN_0}\bbM_{\leq\k}
		.
	\end{align*}
	We call elements of $\bbM_{<\infty}$ \textit{power bounded} measures.
\end{definition}

\begin{definition}
\label{def:903}
	Let $\k\in\bbN_0$. We denote by $\bbE_{\leq \k}$ the set of all pairs $(\mu,p)$ where $\mu\in \bbM_{\leq\k}$,
	$p$ is a polynomial with real coefficients whose degree does not exceed $2\k+1$, and 
	\begin{equation}\label{eq:27}
		\frac{1}{(2\k+1)!}p^{(2\k+1)}(0)\geq \|\mu\|_\k
		.
	\end{equation}
\end{definition}

\begin{definition}
\label{def:954}
	Let $\k\in\bbN_0$. The $\k$-\textit{regularized Cauchy transform} is the map 
	\\ $\mathrm C_\k:\bbE_{\leq \k}\to \Hol(\bbC_+)$ defined by 
	\begin{align}\label{eq:71}
		\mathrm C_\k[\mu,p](z):= 
		p(z)+(1+z^2)^{\k+1}\int_\bbR\frac{1}{t-z}\frac{d\mu(t)}{(1+t^2)^{\k+1}},\quad z\in\bbC_+.
	\end{align}
\end{definition}

The name ``regularized Cauchy transform'' is explained by the identity
\begin{align*}
	(1+z^2)^{\k+1}\frac{1}{t-z}\frac{1}{(1+t^2)^{\k+1}}=\frac{1}{t-z}-(t+z)\sum_{j=0}^{\k}\frac{(1+z^2)^j}{(1+t^2)^{j+1}}.
\end{align*}
Let $\mu\in\bbM_{\leq 0}$, $\alpha\in\bbR$ and $\beta\geq 0$, and set 
$p(z):=\alpha+\big(\beta+\int_\bbR\frac{d\mu(t)}{1+t^2}\big)z$. Then, by the above formula, 
\begin{align*}
	\mathrm C_0[\mu,p](z)= &\, p(z)+\int_\bbR\left(\frac{1}{t-z}-\frac{t+z}{1+t^2}\right)d\mu(t)
	\\
	= &\, \a+\b z+\int_\bbR\left(\frac{1}{t-z}-\frac{t}{1+t^2}\right)d\mu(t).
\end{align*}
This shows that the operators $\mathrm C_\k$ constitute an extension \cref{eq:997} to power bounded measures, 
and also explains the role of \cref{eq:27} in the definition of $\bbE_{\leq\k}$. 

The next theorem says that power boundedness in the context of measures corresponds to sign indefiniteness in the context of their Cauchy-transforms. 

\begin{theorem}{\cite[Theorem 2.8]{langer.woracek:kara}}
\label{thm:950}
	Let $\k\in\bbN_0$. The map $\mathrm C_\k$ is a bijection from $\bbE_{\leq\k}$ onto $\cN_{\leq\k}^{(\infty)}$. 
		If $(\mu,p)\in\bbE_{\leq\k}$ and $q:=\mathrm C_\k[\mu,p]$, then $\mu$ can be recovered by 
		\begin{equation}
		\label{eq:949}
			\forall a,b\in\bbR,\,a<b:
			\mu((a,b))=\lim_{\d\to 0}\lim_{\e\to 0}\frac{1}{\pi}\int_{a+\d}^{b-\d}\Im q(t+i\e)dt.
		\end{equation}
		The polynomial $p$ can be recovered from the $2\k+2$ equations obtained by splitting real- and imaginary parts of 
		$q^{(k)}(i)=p^{(k)}(i)$, $k=0,\ldots,\k$. 
	\end{theorem}
In fact, a more general variant of \cref{eq:949} holds, the Stieltjes-Livshits inversion formula 
(see e.g.\ \cite[Theorem~1.2.4]{gorbachuk.gorbachuk:1997}). In the present context it can be formulated as follows: 
if $q:=\mathrm C_\k[\mu,p]$ and $f$ is a function which is analytic on some open set containing the real axis and takes real
values on $\bbR$, then 
\begin{equation}
	\label{eq:916}
	\forall a,b\in\bbR,\,a<b:
	\int_{(a,b)}f(t)d\mu(t)=\lim_{\d\to 0}\lim_{\e\to 0}\frac{1}{\pi}\int_{a+\d}^{b-\d}\Im\big[f(t)q(t+i\e)\big]dt.
\end{equation}

Using the appropriate topologies these mappings are continuous. For $\omega_k(t)=(1+t^2)^{\kappa+1}$, we define
$$C_0(\bbR,\omega_\kappa):=\left\{f\in C(\bbR)\mid \lim_{|t|\to\infty}\omega_\kappa(t)f(t)=0 \right\}.$$

	\begin{theorem}\label{thm:950b}Let $\k\in\bbN_0$. Let $((\mu_j,p_j))_{j\in J}$ be a net in $\bbE_{\leq\k}$ and $(\mu,p)\in\bbE_{\leq\k}$. Then the following are equivalent
		\begin{enumerate}[{\rm(i)}]
			\item  ${\displaystyle\lim_{j\in J}\mathrm C_\k[\mu_j,p_j]=\mathrm C_\k[\mu,p]}$
		\item ${\displaystyle\lim\limits_{j\in J}p_j=p}$ and
			$
			{\displaystyle\lim\limits_{j\in J}\mu_j=\mu}
			$ in the $w^*$-topology of $C_0(\bbR,\omega_\kappa)^*$
		\item ${\displaystyle\lim\limits_{j\in J}p_j=p}$ and
		$
		{\displaystyle\lim\limits_{j\in J}\mu_j=\mu}
		$ in the $w^*$-topology of $C_c(\bbR)^*$
			\item ${\displaystyle\lim\limits_{j\in J}p_j=p}$ and
			\begin{align*}
				\forall a,b\in\bbR,\,a<b,\,\mu(\{a\})=\mu(\{b\})=0:\ 
				\lim_{j\in J}\mu_j((a,b))=\mu((a,b))
		\end{align*}
	\end{enumerate}
\end{theorem}
\begin{proof}
	The equivalence of $(i)$ and $(ii)$ is the context of \cite[Proposition 2.10]{langer.woracek:kara}. 
	Clearly $(ii)$ implies $(iii)$. To show $(iii)\implies (ii)$, first note that  
	\[
	\|\mu_j\|_{C_0(\bbR,\omega_\kappa)^*}=\|\mu_j\|_\kappa\leq 	\frac{1}{(2\k+1)!}p_j^{(2\k+1)}(0).
	\]
	Hence, $\lim\limits_{j\in J}p_j=p$ implies that there is $j_0\in J$ and $c>0$ such that for $j\geq j_0$ we have $\|\mu_j\|_\kappa\leq c$. 
	Since $C_c(\bbR)$ is dense in $C_0(\bbR,\omega_\kappa)$, we get $(iii)$ implies $(ii)$. 
	
	The equivalence of $(iii)$ and $(iv)$ is proved in \Cref{lem:950b}.
\end{proof}

One can think of the formula \cref{eq:71} as an additive decomposition of a function $q\in\cN_{<\infty}^{(\infty)}$. 
There is also a multiplicative decomposition of such functions. The following result is shown in 
\cite{dijksma.langer.luger.shondin:2000}.

\begin{theorem}
\label{thm:951}
	Let $\k\in\bbN$, and $q\in\Hol(\bbC_+)$. Then $q\in\cN_\k^{(\infty)}$, 
	if and only if there exists a function $q_0\in\cN_0$ and points (not necessarily different) 
	$\b_1,\dots,\b_\k\in\bbC_+\cup\bbR$, such that 
	\[
		q(z)=q_0(z)\prod_{j=1}^\k(z-\b_j)(z-\overline{\b_j}).
	\]
\end{theorem}

%
\subsection{Generalized Nevanlinna functions and matrix families}
\label{sec:2-6}
%

Recall again \Cref{sec:2-4} where we saw that functions $q\in\cN_0\cup\{\infty\}$ correspond to Hamiltonians $H\in\bbH_{a,b}$
which are in limit circle case at $a$ and in limit point case at $b$. This correspondence is established via the fundamental
solution $W_H(t,z)$ of a Hamiltonian $H$ and Weyl's limit formula ($\tau_t\in\cN_0\cup\{\infty\}$)
\[
	\lim_{t\to\infty}W_H(t,.)\star\tau_t=q
	.
\]
A sign-indefinite analogue of the theory of canonical systems, dealing with functions $q$ for which the kernel \cref{eq:952} has
a finite number of negative squares, is developed in the series of papers 
\cite{kaltenbaeck.woracek:db}--\cite{kaltenbaeck.woracek:p6db}.%
\nocite{kaltenbaeck.woracek:p2db,kaltenbaeck.woracek:p3db,kaltenbaeck.woracek:p4db,kaltenbaeck.woracek:p5db}
A refined investigation of the class $\cN_{<\infty}^{(\infty)}$ in this context, and the connection with Hamiltonians with two
limit point endpoints, is undertaken in \cite{langer.woracek:gpinf} and \cite{langer.woracek:sinham}. 
We restate comprehensively what is needed from those papers for our present work. 

Recall here that an analytic function is called of
bounded type in some domain, if it can be written as a quotient of two bounded analytic functions in this domain. 

\begin{theorem}
\label{thm:906}
	Let $q\in\cN_{<\infty}^{(\infty)}$. Then there exist functions 
	\[
		W:(a,b)\times\bbC\to\bbC^{2\times 2},\quad H:(a,b)\to\bbR^{2\times 2},
	\]
	defined on some interval with $-\infty\leq a<b\leq\infty$, which possess the following properties (here we write 
	$W(t,z)=(w_{ij}(t,z))_{i,j=1}^2$).
	\begin{enumerate}[{\rm(i)}]
	\item For each $t\in(a,b)$ the functions $z\mapsto w_{ij}(t,z)$ are real entire and of bounded type in $\bbC_+$ and 
		$\bbC_-$. We have $W(t,0)=I$ and $\det W(t,z)=1$. 
	\item For each $t\in(a,b)$ the function 
		\[
			E(t,.):=w_{22}(t,.)+iw_{21}(t,.)
		\]
		belongs to $\HB^*$. 
	\item The function $H$ belongs to $\bbH_{a,b}$ and is in limit point case at $b$. For each $z\in\bbC$ the function 
		$t\mapsto W(t,z)$ is absolutely continuous, and 
		\begin{equation}
		\label{eq:915}
			\partial_tW(t,z)J=zW(t,z)H(t),\quad t\in(a,b)\text{ a.e.}
		\end{equation}
	\item We have 
		\[
			\lim_{t\to a}(0,1)W(t,z)=(0,1)
		\]
		locally uniformly for $z\in\bbC$.
	\item For every family $(\tau_t)_{t\in(a,b)}$ with $\tau_t\in\cN_0\cup\{\infty\}$ we have 
		\[
			\lim_{t\to b}\big[W(t,z)\star\tau_t(z)\big]=q(z)
		\]
		locally uniformly for $z\in\bbC_+$. 
	\end{enumerate}
\end{theorem}

If $q,W,H$ are as in the theorem, we say that $W$ is a \emph{matrix family for $q$ with Hamiltonian $H$}.

The connection between $q,W,H$ is known to have several additional properties, and we state some of them.

\begin{proposition}
\label{pro:914}
	Let $q\in\cN_{<\infty}^{(\infty)}$, and let $W$ be a matrix family for $q$ with Hamiltonian $H$ (defined on and interval 
	$(a,b)$). Then the following statements hold. 
	\begin{enumerate}[{\rm(i)}]
	\item The Hamiltonian $H$ is in limit circle case at $a$, if and only if $q=p+q_0$ with some $q_0\in\cN_0$ and a
		polynomial $p$ with real coefficients.
	\item If $H$ is in limit circle case at $a$, then 
		\[
			W(t,z)=\begin{pmatrix} 1 & p(z)\\ 0 & 1\end{pmatrix}W_H(t,z),\quad t\in(a,b),z\in\bbC
			,
		\]
		with some polynomial $p$ with real coefficients.
	\item If $q$ is not a polynomial with real coefficients, then there exist $s,t\in(a,b)$, $s\leq t$, such that 
		$\cB(E(s,.))=\cB(E(t,.))^\flat$.
	\end{enumerate}
\end{proposition}

For later reference, let us make the following simple formula explicit.

\begin{lemma}
\label{lem:913}
	Let $q\in\cN_{<\infty}^{(\infty)}$, and let $W$ be a matrix family for $q$ with Hamiltonian $H$ (defined on $(a,b)$).
	Then, for each $c\in(a,b)$, we have 
	\begin{equation}
	\label{eq:912}
		W(c,.)\star q_{H|_{(c,b)}}=q
		.
	\end{equation}
\end{lemma}
\begin{proof}
	The function $W(t,.)$ solves \cref{eq:915}, and therefore for all $t\in[c,b)$ 
	\[
		W(t,.)=W(c,.)W_{H|_{(c,b)}}(t,.)
		.
	\]
	This implies 
	\[
		q=\lim_{t\to b}\big[W(t,.)\star 0\big]=W(c,.)\star\lim_{t\to b}\big[W_{H|_{(c,b)}}(t,.)\star 0\big]
		=W(c,.)\star q_{H|_{(c,b)}}
		.
	\]
\end{proof}

%
%
%
\newpage
\section{The structure Hamiltonian}
\label{sec:3}
%
%
%

The connection between dB-spaces on the one hand and canonical systems on the other is a core feature in de~Branges' work. 
Our aim in this section is to present it in a language adapted to our present needs and to prove some additions to 
\cite{debranges:1968}. 

The link between the two objects is the set of subspaces of a given dB-space which are themselves dB-spaces.

\begin{definition}
\label{def:960}
	Let $\cH\in\DB$. Then we denote 
	\[
		\scrC(\cH):=\big\{\{0\}\big\}\cup\big\{\cL\in\DB\mid \cL\sqsubseteq \cH,\vartheta_\cL=\vartheta_\cH\big\}
		.
	\]
\end{definition}

The comprehensive result now reads as follows. 

\begin{theorem}
\label{thm:1}
	Let $E\in\HB$. Then there exists $H\in\bbHe_{-\infty,0}$ with $I_{\rm reg}\neq\emptyset$, 
	such that the solution of 
	\begin{equation}
	\label{eq:4}
		\begin{aligned}
			&\partial_t(A(t,z),\ B(t,z))J =z(A(t,z),\ B(t,z))H(t),\quad t\leq 0\text{ a.e.},
			\\
			& (A(0,z),\ B(0,z))=(A(z),\ B(z)),
		\end{aligned}
	\end{equation}
	has the following properties.
	\begin{enumerate}[{\rm(i)}]
	\item Set $T_-:=\inf I_{\rm reg}\in[-\infty,0)$. For $t\in(T_-,0]$ we have $E(t,\cdot)\in\HB$, and for 
		$t\in(-\infty,T_-]$ the function $E(t,\cdot)$ is a scalar multiple of a real entire function. 
		We have $\lim\limits_{t\to T_-}K_{E(t,\cdot)}=0$. 
	\item It holds that 
		\begin{align}
			\label{eq:3}
			& \scrC(\cB(E))=\big\{\{0\}\big\}\cup\big\{\cB(E(t,\cdot))\mid t\in(T_-,0]\big\},
			\\
			\label{eq:994}
			& \big\{\cL\in\DB\mid \cL\subseteq_i\cB(E),\vartheta_\cL=\vartheta_E\big\}
			=\big\{\cB(E(t,\cdot))\mid t\in(T_-,0]\cap I_{\rm reg}\big\}.
		\end{align}
	\item Denote 
		\[
			\cH_t:=
			\begin{cases}
				\cB(E(t,\cdot)) &\text{if}\ t\in(T_-,0],
				\\
				\{0\} &\text{if}\ t\in[-\infty,T_-],
			\end{cases}
			\qquad
			\chi:\left\{
			\begin{array}{rcl}
				[T_-,0] & \to & \scrC(\cB(E)),
				\\[1mm]
				t & \mapsto & \cH_t.
			\end{array}
			\right.
		\]
		Then $\chi$ is an order isomorphism and homeomorphism. 
		In particular, $\scrC(\cB(E))$ is totally ordered w.r.t.\ $\sqsubseteq$ and compact as a subset of $\RK$. 
	\end{enumerate}
	The Hamiltonian $H$ is uniquely determined by the property that the solution of \cref{eq:4} satisfies \cref{eq:994}. 

	The following additional properties are satisfied.
	\begin{enumerate}[{\rm(i)}]
	\setcounter{enumi}{3}
	\item Let $(t_0,t_1)$ be an indivisible interval with $t_0\in[T_-,0]\cap I_{\rm reg}$. 
		Then $\cH_{t_0}=\cH_{t_1}^\flat$. For $G\in\cH_{t_1}\setminus\{0\}$ with $G\perp\cH_{t_0}$ the map 
		$t\mapsto \|G\|_{\cH_t}$ is a decreasing bijection from $(t_0,t_1)$ onto $(\|G\|_{\cH_{t_1}},\infty)$. 
	\item Set $\alpha:=-\arg E(0)$, then $\int_{-\infty}^0e_\alpha^*H(t)e_\alpha dt<\infty$. 
	\end{enumerate}
\end{theorem}

This theorem is only a slight extension of what was shown by de~Branges. 
In \cite{debranges:1968} all assertions of the theorem are shown with exception of the topological properties in (iii),
the inclusion ``$\subseteq$'' in \cref{eq:3}, and it was always assumed that $E$ has no real zeros.
While removing the restriction on real zeros is simple, showing equality in 
\cref{eq:3} requires an argument which we provide now. First note the following geometric fact.

\begin{lemma}
\label{lem:1}
	Let $(\cH,\langle\cdot,\cdot\rangle_{\cH})$ be a Hilbert space, and let $\cL_0$ and $\cL$ be linear subspaces of $\cH$ 
	with $\cL_0\subseteq\cL$. Assume that $\langle\cdot,\cdot\rangle_{\cL}$ is a scalar product on $\cL$ such that 
	\begin{equation}\label{eq:21}
	(\cL_0,\langle\cdot,\cdot\rangle_{\cL})\subseteq_i(\cH,\langle\cdot,\cdot\rangle_{\cH})
	\ \text{and}\ 
	(\cL,\langle\cdot,\cdot\rangle_{\cL})\subseteq_c(\cH,\langle\cdot,\cdot\rangle_{\cH})
	\end{equation}
	Then 
	\[
	\forall y\in\cL: y\perp_{\cH}\cL_0\Rightarrow y\perp_{\cL}\cL_0.
	\]
\end{lemma}
\begin{proof}
	Let $x\in\cL_0$. For each $\a\in\bbC$ we have
	\begin{align*}
	\langle x,x\rangle_\cL=\langle x,x\rangle_\cH\leq \langle \a y+x, \a y+x\rangle_\cH\leq \langle \a y+x, \a y+x\rangle_\cL.
	\end{align*}
	Since $\a$ was arbitrary, it follows that $\langle x,y\rangle_\cL=0$. 
\end{proof}

\begin{proof}[Proof of \Cref{thm:1}]
	We proceed in two steps. First we assume that $E\in\HB^*$ and prove the missing parts, 
	where we use that all other assertions in the theorem are readily known from \cite{debranges:1968}. 
	After having completed this step, we remove the assumption on real zeros. 
\begin{Elist}
\item
	To show the inclusion ``$\subseteq$'' in \cref{eq:3} let $E\in\HB^*$ and $\cL\in\DB^*$ with $\cL\sqsubseteq\cB(E)$. 
	In view of \cref{eq:994} we may assume furthermore 
	that $\cL\nsubseteq_i\cB(E)$, and hence that $\cL\neq\cL^\flat$. Since $\dim\cL/\cL^\flat=1$ 
	the set $\cL$ is a closed linear subspace of $\cB(E)$, and since $\cL$ is invariant under the operations in 
	\Cref{def:1}(ii),(iii) the space $(\cL,\langle\cdot,\cdot\rangle_{\cB(E)})$ is a 
	dB-space which is isometrically contained in $\cB(E)$. The same holds, clearly, for $\cL^\flat$, and \cref{eq:994}
	furnishes us with $t_0,t_1\in(T_-,0]$ such that 
	\[
		(\cL^\flat,\langle\cdot,\cdot\rangle_{\cB(E)})=\cH_{t_0},\quad
		(\cL,\langle\cdot,\cdot\rangle_{\cB(E)})=\cH_{t_1}.
	\]
	Since $\dim\cL/\cL^\flat=1$, we have $t_0<t_1$ and the interval $(t_0,t_1)$ is indivisible. 
	
	Choose $G\in\cL$ such that $(\cL,\langle\cdot,\cdot\rangle_{\cB(E)})=\cL^\flat\oplus_{\cB(E)}\spann\{G\}$. 
	By \Cref{lem:1}, also $(\cL,\langle\cdot,\cdot\rangle_{\cL})=\cL^\flat\oplus_{\cL}\spann\{G\}$. 
	Since $(\cL,\langle\cdot,\cdot\rangle_{\cL})$ is contractively contained in
	$(\cB(E),\langle\cdot,\cdot\rangle_{\cB(E)})$, we have $\|G\|_\cL\geq\|G\|_{\cB(E)}$. Item (iv) of the theorem
	provides us with $t\in(t_0,t_1]$ with $(\cL,\langle\cdot,\cdot\rangle_{\cL})=\cH_t$. 

	We come to item (iii).
	We know from \Cref{thm:1} that $\chi$ is order preserving and injective. Moreover, $t\mapsto E(t,\cdot)$ is 
	continuous for $t\in(-\infty,0]$ and $\lim\limits_{t\to T_-}K_{E(t,\cdot)}=0$, and this yields that $\chi$ 
	maps $[T_-,0]$ continuously into $\DB\cup\{\{0\}\}\subseteq\RK$. It remains to note that $[T_-,0]$ is compact. 
\item
	Assume now we have $E\in\HB$ with $\vartheta_E\neq 0$. Choose a real entire function $C$ with only real zeros such 
	that $\vartheta_E=\vartheta_C$. Set $E_0=\frac{E}{C}$, then $E_0\in\HB^*$ and we may define 
	\[
	H_E:=H_{E_0}.
	\]
	The solution of \cref{eq:4} is nothing but $E(t,\cdot)=CE_0(t,\cdot)$, and \Cref{lem:26} shows that all properties 
	of the family $E(t_0,\cdot)$ transfer to the corresponding properties of $E(t,\cdot)$. 
\end{Elist}
\end{proof}

Based on the above theorem we may introduce the following notation. 

\begin{definition}
	Let $E\in\HB$. The Hamiltonian whose existence and uniqueness is granted by \Cref{thm:1} is called the 
	\emph{structure Hamiltonian} of $E$. If we wish to emphasize the dependence on $E$, we denote the 
	structure Hamiltonian as $H_E$ and write $T_-(E)$ and $\cH_t(E)$. 
\end{definition}

\begin{example}
\label{ex:967}
	Assume we have $H\in\bbHe_{a,b}$ which is in lc at both endpoints, let $W(x,z)$, $x\in[a,b]$, be its fundamental 
	solution and $E(x,z)$ be the associated Hermite-Biehler functions \cref{eq:966}. 
	Then the structure Hamiltonian of $E(b,z)$ is given as 
	\[
		H_{E(b,\cdot)}(t)=
		\begin{cases}
			-JH(t-(b-a))J &\text{if}\ -(b-a)<t<0,
			\\
			\big(\begin{smallmatrix} 0 & 0 \\ 0 & 1 \end{smallmatrix}\big) &\text{if}\ t<-(b-a).
		\end{cases}
	\]
	Moreover, $T_-(E(b,\cdot))=-(b-a)$.
\end{example}

This example is universal in the following sense: If $E\in\HB$ with $E(0)=1$, then $T_-(E)>-\infty$ if and only if 
$E$ occurs from a limit circle Hamiltonian in the above way.

Let us state two immediate consequences of uniqueness in \Cref{thm:1}.
The first is obvious and the second relies on the transformation rule \Cref{lem:6}.

\begin{corollary}\label{cor:1}
	Let $E\in\HB$ and $H_E$ its structure Hamiltonian. Let $t_0\in(T_-,0)$ and $\tilde E:=E(t_0,\cdot)$, 
	where $E(t,\cdot)$ is the solution of \cref{eq:4}. Then we have 
	\[
		H_{\tilde E}(t)=H_E(t+t_0),\quad \tilde E(t,\cdot)=E(t+t_0,\cdot)
		.
	\]
\end{corollary}

\begin{corollary}\label{cor:5}
	Let $E\in\HB$ and $H_E$ its structure Hamiltonian. Let $M\in\SL(2,\bbR)$, then 
	\[
	H_{E\ltimes M}\sim \cT_{M^{-1}}H_E.
	\]
\end{corollary}
\begin{proof}
	We have
	\[
	(A(t,z),\ B(t,z))W_{H_E}(t,0,z)
	=
	(A(z),\ B(z)),
	\]
	and hence
	\[
	(A(t,z),\ B(t,z))MM^{-1}
	W_{H_E}(t,0,z)M
	=
	(
	A(z),\ B(z))
	M.
	\]
	By \Cref{lem:6} it holds that $M^{-1}W_{H_E}(t,0,z)M=W_{\cT_{M^{-1}}H_E}(t,0,z)$, and the assertion follows by 
	uniqueness of the structure Hamiltonian. 
\end{proof}

There is a close connection between structure Hamiltonians and Weyl coefficients. This is based on use of an 
involution on the set of Hamiltonians: for $H\in\bbH_{a,b}$, let $H^\dagger\in\bbH_{-b,-a}$ be the Hamiltonian 
defined on $(-b,-a)$ by 
\[
H^\dagger(t)=UH(-t)U
,
\]
where $U:=\big(\begin{smallmatrix} 1&0\\0&-1 \end{smallmatrix}\big)$.
The mentioned connection is established by the following result which can be found e.g.\ in \cite{kac:2007}.
For the convenience of the reader, we provide a direct deduction from \Cref{thm:1}. 

\begin{proposition}\label{lem:25}
	Let $E\in \HB$ and let $H_E$ be the corresponding structure Hamiltonian. Then we have 
	\[
		q_{H_E^\dagger}=\frac{B}{A}.
	\]
	Remember here the definition \cref{eq:36} of the Weyl coefficient.
\end{proposition}
\begin{proof}
	Let $W(t,0,z)$ denote the transfer matrix for $H_E$. Then we have
	\begin{align}\label{eq:46}
	(A(t,z),\ B(t,z))W(t,0,z)=(A(z),\ B(z)).
	\end{align}
	As short computation shows that the fundamental solution $W^\dagger(t,z)$ for $H^\dagger$ is given as
	\[
	W^\dagger(t,z)=UW(-t,0,z)^{-1}U. 
	\]
	Since $\det W(t,0,z)=1$, we have $J W(t,0,z)^\intercal J^{-1}=W(t,0,z)^{-1}$. 
	Transposing \cref{eq:46} and rewriting for $W^\dagger$, we get for $t\geq 0$
	\begin{align*}
	\begin{pmatrix}
	B(-t,z)\\
	A(-t,z)
	\end{pmatrix}
	=	W^\dagger (t,z)^{-1}
	\begin{pmatrix}
	B(z)\\
	A(z)
	\end{pmatrix}.
	\end{align*}
	For the parameter family $(\tau_t)_{t\in(-b,-a)}$ defined as $\tau_t(z)=\frac{B(-t,z)}{A(-t,z)}$, we have
	\[
	W^\dagger (t,z)\star \tau_t(z)=\frac{B(z)}{A(z)}.
	\]
	Since $E(-t,\cdot)\in\HB$ we have $\tau_t\in\cN_0$, and sending $t\to\infty$ proves the assertion.
\end{proof}

As a first consequence we obtain a kind of converse to \Cref{cor:1}.

\begin{corollary}
\label{cor:977}
	Let $E\in\HB$ and $H_E$ its structure Hamiltonian. Let $T\in(0,\infty)$ and $H\in\bbHe_{0,T}$, and set 
	\[
		\tilde E:=E\ltimes W_H(T,\cdot)
		.
	\]
	Then the structure Hamiltonian of $\tilde E$ is given (a.e.) as 
	\begin{equation}
	\label{eq:976}
		H_{\tilde E}(t)=
		\begin{cases}
			H(t+T) &\text{if}\ -T<t\leq 0,
			\\
			H_E(t+T) &\text{if}\ t\leq -T,
		\end{cases}
	\end{equation}
	and the corresponding chain of Hermite-Biehler functions as 
	\[
		\tilde E(t,\cdot)=
		\begin{cases}
			E\ltimes W_H(t+T) &\text{if}\ -T<t<0,
			\\
			E(t+T) &\text{if}\ t\leq -T.
		\end{cases}
	\]
\end{corollary}
\begin{proof}
	Write 
	\[
		W_H(0,T,\cdot)=\begin{pmatrix} w_{11} & w_{12} \\ w_{21} & w_{22} \end{pmatrix}
		,
	\]
	then we have 
	\[
		W_{H^\dagger}(-T,0,\cdot)=\begin{pmatrix} w_{22} & w_{12} \\ w_{21} & w_{11} \end{pmatrix}
		.
	\]
	Let $\tilde H$ be the Hamiltonian defined by the right side of \cref{eq:976}. Then 
	\[
		\tilde H^\dagger(t)=
		\begin{cases}
			H^\dagger(t-T) &\text{if}\ 0<t<T,
			\\
			H_E^\dagger(t-T) &\text{if}\ T<t.
		\end{cases}
	\]
	Further, note that 
	\[
		(\tilde A,\tilde B)=(A,B)\begin{pmatrix} w_{11} & w_{12}\\ w_{21} & w_{22}\end{pmatrix}
		=\big(w_{11}A+w_{21}B,w_{12}A+w_{22}B\big)
		.
	\]
	Multiplicativity of transfer matrices \cref{eq:948} implies 
	\[
		q_{\tilde H^\dagger}=W_{H^\dagger}(-T,0,\cdot)\star q_{H_E^\dagger}
		=\begin{pmatrix} w_{22} & w_{12} \\ w_{21} & w_{11}\end{pmatrix}\star \frac BA
		=\frac{w_{22}B+w_{12}A}{w_{21}B+w_{11}A}=\frac{\tilde B}{\tilde A}
		.
	\]
	Finally, it is clear that the written family $\tilde E(t,\cdot)$ is a solution of the correct equation.
\end{proof}

By our choice of a metric topologizing $\bbHe_{0,\infty}$ and $\bbHe_{-\infty,0}$, cf.\ \cref{eq:26}, 
the map $H\mapsto H^\dagger$ is an isometry between those two sets of Hamiltonians. 

We can now show a continuity result. 

\begin{proposition}
\label{prop:3}
	The maps
	\[
	\left\{
	\begin{array}{rcl}
	\HB & \to & \bbHe_{-\infty,0}
	\\
	E & \mapsto & H_E
	\end{array}
	\right.
	\qquad
	\left\{
	\begin{array}{rcl}
	\HB\times[-\infty,0] & \to & \DB\cup\{\{0\}\}
	\\
	(E,t) & \mapsto & \cH_t(E)
	\end{array}
	\right.
	\]
	are continuous. 
\end{proposition}
\begin{proof}
	Continuity of the first map comes from \Cref{thm:8}. 
	Let $(E_i)_{i\in I}$ be a net in $\HB$, $E\in\HB$, and assume that $\lim\limits_{i\in I}E_i=E$. 
	Then $\lim\limits_{i\in I}A_i=A$ and $\lim\limits_{i\in I}B_i=B$. \Cref{lem:25} yields 
	$\lim\limits_{i\in I}q_{H_{E_i}^\dagger}=q_{H_E^\dagger}$, and it follows that
	$\lim\limits_{i\in I}H_{E_i}=H_E$.
	
	Consider now the second map. Since all involved topologies are metrizable, it is enough to show sequential continuity. 
	Let $(E_j,t_j)\in\HB\times[-\infty,0]$ for $j\in\bbN_0$, $(E,t)\in\HB\times[-\infty,0]$, and assume that 
	$\lim_{j\to\infty}E_j=E$ and $\lim_{j\to\infty}t_j=t$. We distinguish two cases. 
	
	\medskip\noindent
	\textit{Case~1, $t>-\infty$:}
	We know that $\lim_{j\to\infty}H_{E_j}=H_E$, and this implies that 
	\[
	\lim_{j\to\infty}W_{H_{E_j}}(t,0,z)=W_{H_E}(t,0,z)
	.
	\]
	Since all $H_{E_j}$ are trace normalized, we have 
	$\lim_{j\to\infty}W_{H_{E_j}}(t_j,t,z)=I$, and together therefore 
	$\lim_{j\to\infty}W_{H_{E_j}}(t_j,0,z)=W_{H_E}(t,0,z)$. Since transfer matrices always have determinant $1$, also
	inverse matrices converge, and \cref{eq:46} implies that $\lim_{j\to\infty}E_j(t_j,\cdot)=E(t,\cdot)$. 
	In turn, it follows that 
	$\lim_{j\to\infty}K_{E_j(t_j,\cdot)}=K_{E(t,\cdot)}$, i.e., $\lim_{j\to\infty}\cH_{t_j}(E_j)=\cH_t(E)$. 
	
	\medskip\noindent
	\textit{Preparation for Case~2:}
	We show that 
	\begin{equation}\label{eq:988}
	\big\{(z,w)\mapsto K_{E_j(t,\cdot)}(z,\overline w)\mid j\in\bbN_0,t\in[-\infty,0]\big\}
	\end{equation}
	is a normal family in $\Hol(\bbC\times\bbC)$. 
	
	Since the sequence $(K_{E_j})_{j\in\bbN_0}$ converges in $\Hol(\bbC\times\bbC)$, it is bounded in the metric of 
	$\Hol(\bbC\times\bbC)$, i.e., locally bounded as a family of complex valued functions. We have 
	$\cH_t(E_j)\sqsubseteq\cB(E_j)$, and hence $K_{E_j(t,\cdot)}(w,w)\leq K_{E_j}(w,w)$ for all $t\in[-\infty,0]$. 
	The Cauchy-Schwarz inequality now yields that \cref{eq:988} is locally bounded and, by Montel's theorem, therefore a
	normal family in $\Hol(\bbC\times\bbC)$. 
	
	\medskip\noindent
	\textit{Case~2, $t=-\infty$:}
	To show that $\lim_{j\to\infty}\cH_{t_j}(E_j)=\{0\}$, it is sufficient to prove that every convergent subsequence has
	limit $\{0\}$. Assume we have $j_n\to\infty$, such that the limit $K:=\lim_{n\to\infty}K_{E_{j_n}(t_{j_n},\cdot)}$ exists.
	Let $s\in(-\infty,0]$, then $\cH_{t_{j_n}}(E_{j_n})\sqsubseteq\cH_s(E_{j_n})$, and hence 
	$K_{E_{j_n}(t_{j_n},\cdot)}(w,w)\leq K_{E_{j_n}(s,\cdot)}(w,w)$ for all sufficiently large $n$. 
	Using what we have shown in Case~1, we find 
	\[
	K(w,w)=\lim_{n\to\infty}K_{E_{j_n}(t_{j_n},\cdot)}(w,w)\leq\lim_{n\to\infty}K_{E_{j_n}(s,\cdot)}(w,w)=
	K_{E(s,\cdot)}(w,w)
	.
	\]
	Since $s$ was arbitrary, it follows that $K(w,w)=0$. This holds for all $w\in\bbC$, and therefore $K=0$. 
\end{proof}

The family $E(t,\cdot)$ of \Cref{thm:1} can be viewed as one possible parametrization of the family $\scrC(\cB(E))$. 
All possible parameterizations can be described. 

\begin{proposition}\label{cor:2}
	Let $E\in \HB$. For $M\in \SL(2,\bbR)$ denote by $T_{M,-}$ and $E_M(t,\cdot)$ the corresponding number and family of 
	functions given by \Cref{thm:1} when applied with the function $E\ltimes M\in\HB$. 
	Then (the dot indicates that the sets in the union are pairwise disjoint)
	\begin{equation}\label{eq:993}
	\{F\in \HB\mid \cB(F)\sqsubseteq\cB(E) , \vartheta_F=\vartheta_E\}=
	\overset{\cdot}{\bigcup_{M\in \SL(2,\bbR)}}\{E_M(t,\cdot)\mid t\in (T_{M,-},0]\}.
	\end{equation}
\end{proposition}
\begin{proof}
	Due to \Cref{cor:5} and its proof, $H_{E\ltimes M}\sim \cT_{M^{-1}}H_E$ and 
	\[
	\{E_M(t,\cdot)\mid t\in (T_{M,-},0]\}=\{E(t,\cdot)\ltimes M\mid t\in (T_{-},0]\}. 
	\]
	Note that $\cT_M$ preserves indivisible intervals. We first show that the union in \cref{eq:993} is disjoint. 
	Let $M,\tilde M\in \SL(2,\bbR)$ and $t,\tilde t\leq 0$ and assume that
	\begin{equation}\label{eq:992}
	E(t,\cdot)\ltimes M=E(\tilde t,\cdot)\ltimes \tilde M.
	\end{equation}
	Then 
	\[
	\cB(E(t,\cdot))=\cB(E(t,\cdot)\ltimes M)=\cB(E(\tilde t,\cdot)\ltimes \tilde M)=\cB(E(\tilde t,\cdot))
	,
	\]
	and it follows that $t=\tilde t$. Since the functions $A(t,\cdot)$ and $B(t,\cdot)$ are linearly independent, 
	\cref{eq:992} implies $M=\tilde M$. 
	
	The inclusion ``$\supseteq$'' in \cref{eq:993} is clear. To prove the reverse inclusion, let $F\in \HB$ with 
	$\cB(F)\sqsubseteq\cB(E)$ and $\vartheta_F=\vartheta_E$ be given. Then there exists $t\in (T_{-},0]$ such that 
	$\cB(F)=\cB(E(t,\cdot))$, and hence we find $M\in \SL(2,\bbR)$ with $F=E(t,\cdot)\ltimes M$. 
\end{proof}

\begin{corollary}
	\label{cor:975}
	Let $F\in\HB$ and $\cH\in\DB$ be such that $\cB(F)\sqsubseteq\cH$ and $\vartheta_F=\vartheta_\cH$. Then there exists a
	unique function $E\in\HB$ such that $\cH=\cB(E)$ and $F=E(t,\cdot)$ for some $t\leq 0$. 
	The number $t$ is uniquely determined by $F$ and $\cH$. 
\end{corollary}
\begin{proof}
	Choose $\tilde E$ with $\cH=\cB(\tilde E)$. Then there exists a unique matrix $M\in\SL(2,\bbR)$ with 
	\[
	F\in\{E_M(t,\cdot)\mid t\in (T_{M,-},0]\}
	.
	\]
	Set $E:=\tilde E\ltimes M$.
\end{proof}

%
%
%
\newpage
\section{Chains of de~Branges Spaces}
\label{sec:4}
%
%
%

We have seen in the previous section that a single Hermite-Biehler function gives rise to a whole chain of dB-spaces 
parameterized by a canonical system. Our aim in this section is to axiomatize the notion of a chain. This is a core concept, 
and in particular the convergence result \Cref{thm:980} is a key tool.

%
\subsection{Bounded and unbounded chains axiomatically}
\label{sec:4-1}
%

\begin{definition}\label{def:2}
	Let $\scrC\subseteq\DB\cup\{\{0\}\}$. We call $\scrC$ a \emph{chain} if it satisfies the following properties.
	\begin{enumerate}[{\rm(i)}]
		\item\label{it:3-1} $\{0\}\in\scrC$ and $\scrC\neq \{\{0\}\}$;
		\item\label{it:3-2} $\scrC$ is totally ordered with respect to $\sqsubseteq$;
		\item\label{it:3-3} $\scrC$ is closed (in the topology of $\RK$);
		\item\label{it:3-4} for each element $\cH\in\scrC\setminus\{\{0\}\}$ we have 
		\[
		\cH=\sup\{\cL\in\scrC\mid \cL\sqsubseteq\cH, \cL\neq \cH \};
		\]
		\item\label{it:3-5} for each two elements $\cH_1,\cH_2\in \scrC\setminus\{\{0\}\}$ we have 
		$\vartheta_{\cH_1}=\vartheta_{\cH_2}$.
	\end{enumerate}
	We call a chain \emph{bounded} if it contains a largest element and \emph{unbounded} otherwise. 
	We denote the set of chains, bounded chains, and unbounded chains as $\Ch, \bCh$ and $\ubCh$, respectively. 
	For a chain $\scrC$, we denote by $\vartheta_\scrC$ the	common real zero divisor of its nonzero elements, i.e., 
	$\vartheta_\scrC:=\vartheta_\cH$ for all $\cH\in\scrC\setminus\{\{0\}\}$.
\end{definition}

\begin{example}
\label{ex:986}
	For a space $\cH\in\DB$ we consider the set $\scrC(\cH)$ from \Cref{def:960}, and show that it is a bounded chain. 
	By its definition $\scrC(\cH)$ satisfies (i) and (v) of \Cref{def:2}, and $\cH$ is the largest element of $\scrC(\cH)$. 
	We know from \Cref{thm:1} that $\scrC(\cH)$ is compact and hence closed. Moreover, $\scrC(\cH)$ is order isomorphic to an
	interval $[T_-,0]$ with some $T_-\in[-\infty,0)$, and this implies that (ii) and (iv) of \Cref{def:2} hold.
\end{example}

\begin{example}
\label{923}
	Let $H\in\bbH_{a,b}$ be a Hamiltonian which is lc at $a$ and lp at $b$, and let $W_H(t,z)$ be its fundamental solution
	and $K_H(t,z,w)$ the corresponding kernel \cref{eq:59}. We show that
	\[
		\scrC(H):=\big\{\cH(K_H(t,.,.))\mid t\in[a,b)\big\}
	\]
	is an unbounded chain. By \Cref{thm:1} and \Cref{ex:967} we have 
	$\scrC(\cH(K_H(c,..)))=\{\cH(K_H(t,.,.))\in\scrC(H)\mid t\in[a,c]\}$. The properties (i), (ii), (iv), (v) readily follow. 
	To see (iii), it suffices to note that for a sequence $t_n\to b$ we have $K_H(t_n,w,w)\to\infty$ whenever $w$ is
	nonreal, and hence the limit $\lim_{n\to\infty}\cH(K_H(t_n,.,.))$ cannot exist. 
\end{example}

The next result explains a lot about the nature of chains: each beginning section is of the form described in \Cref{ex:986}.

\begin{proposition}\label{lem:991}
	Let $\scrC$ be a chain. If $\cH\in\scrC$, then
	\begin{equation}\label{eq:990}
		\big\{\cL\in\scrC\mid \cL\sqsubseteq\cH\big\}=\scrC(\cH).
	\end{equation}
	In particular, each beginning section of $\scrC$ is a bounded chain. 
\end{proposition}
\begin{proof}
	Choose $E\in\HB$ such that $\cH=\cB(E)$, and let $\chi:t\mapsto\cH_t$ be the map from \Cref{thm:1}\,(iii). 
	Then the set on the right side of \cref{eq:990} equals $\{\cH_t\mid t\in[T_-,0]\}$. 
	The inclusion ``$\subseteq$'' in \cref{eq:990} is clear, hence we have to show that $\chi^{-1}(\scrC)=[T_-,0]$. 
	
	The inverse image $\chi^{-1}(\scrC)$ is closed and contains the points $T_-$ and $0$. 
	Assume towards a contradiction that $\chi^{-1}(\scrC)\neq[T_-,0]$. Then we find $t_1,t_2\in\chi^{-1}(\scrC)$ 
	with $t_1<t_2$ and $(t_1,t_2)\cap\chi^{-1}(\scrC)=\emptyset$. It follows that 
	$\sup\{t\in\chi^{-1}(\scrC)\mid t<t_2\}=t_1$, and we obtain 
	\[
		\cH_{t_2}=\sup\{\cL\in\scrC\mid \cL\sqsubseteq\cH_{t_2},\cL\neq\cH_{t_2}\}=\cH_{t_1}.
	\]
	This contradicts injectivity of $\chi$.
\end{proof}

We have the obvious corollary that bounded chains can be seen as nothing but a different encoding of dB-spaces.

\begin{corollary}\label{thm:2}
	\phantom{}
	\begin{enumerate}[{\rm(i)}]
	\item The maps $\cH\mapsto\scrC(\cH)$ and $\scrC\mapsto\max\scrC$ establish mutually inverse bijections between $\DB$
		and $\bCh$.
	\item Let $\scrC\subseteq\DB\cup\{\{0\}\}$. Then $\scrC$ is a bounded chain if and only if there exist 
		$-\infty\leq a<b\leq\infty$ and $\chi:[a,b]\to\scrC$, such that $\chi$ is a homeomorphism and preserves
		order.
	\end{enumerate}
\end{corollary}

We also obtain a structural property of the set of chains. 

\begin{proposition}\label{cor:987}
	The following statements hold.
	\begin{enumerate}[{\rm(i)}]
	\item Let $C\subseteq\Ch$ with $|C|\geq 2$. Then either $\bigcap C=\{\{0\}\}$ or $\bigcap C\in\bCh$. 
	\item The set of maximal elements of $\Ch$ is equal to $\ubCh$.
	\end{enumerate}
\end{proposition}
\begin{proof}
	For the proof of (i) consider $C\subseteq\Ch$ with $\bigcap C\neq\{\{0\}\}$. 
	Then $\vartheta_{\cH_1}=\vartheta_{\cH_2}$ for each two elements $\cH_1,\cH_2\in(\bigcup C)\setminus\{\{0\}\}$. 
	Choose $\scrC_1\in C$ with $\scrC_1\nsubseteq\bigcap C$, and choose $\cH_1\in\scrC_1$ and $\scrC_2\in C$ with 
	$\cH_1\in\scrC_1\setminus\scrC_2$. Furthermore, choose $E_1\in\HB$ with $\cH_1=\cB(E_1)$. 
	
	Let $\cL\in\bigcap C$. Then $\cH_1\not\sqsubseteq\cL$ since $\cL\in\scrC_2$, and since $\cL\in\scrC_1$ it follows that
	$\cL\sqsubseteq\cH_1$. We obtain 
	\[
		\bigcap C\subseteq\scrC(\cH_1)\cong[T_-(E_1),0].
	\]
	The set $\bigcap C$ is closed, and therefore contains a largest element, say $\cH:=\max\bigcap C$. 
	\Cref{lem:991}, applied to each element of $C$, yields $\bigcap C=\scrC(\cH)$ and this is a bounded chain. 

	We come to the proof of (ii). Assume $\scrC\in\Ch$ is not maximal, and choose $\scrC_1\in\Ch$ with 
	$\scrC\subsetneq\scrC_1$. Applying the already proved statement (i) with $C:=\{\scrC,\scrC_1\}$ yields that 
	$\scrC$ is a bounded chain. Conversely, assume that $\scrC\in\bCh$. Choose $E\in\HB$ with $\cB(E)=\max\scrC$, and set 
	\[
		E_1:=E\ltimes\begin{pmatrix} 1 & z\\ 0 & 1\end{pmatrix}
		.
	\]
	Then $\cB(E)\sqsubsetneq\cB(E_1)$ and hence $\scrC\subsetneq\scrC(\cB(E_1))$.
\end{proof}

%
\subsection{Concrete realization of chains}
\label{sec:4-2}
%

We come to a description of chains (bounded or unbounded) which resembles \Cref{thm:1}. The idea is to pin one element of the
chain at ``$t=0$'' and describe the evolution to ``$t>0$'' by a canonical system similar as in \Cref{923}, and the part for 
``$t<0$'' by the structure Hamiltonian from \Cref{thm:1}. 
In particular, we will see that every unbounded chain is order isomorphic and homeomorphic to an
interval $[T_-,\infty)$ with some $T_-\in[-\infty,0)$. 
To formulate this description in a unified manner for bounded and unbounded chains, we introduce the following notation. 

\begin{definition}\label{def:985}
	We denote by $\bbHen$ the set of all locally integrable functions $H:(0,\infty)\to\bbR^{2\times 2}$ 
	such that $H(t)\geq 0$ for a.a.\ $t\in(0,\infty)$, and that there exists $T_+(H)\in[0,\infty]$ with 
	\[
		\tr H(t)=
		\begin{cases}
		1 &\text{if}\ t\in(0,T_+(H))\text{ a.e.},
		\\
		0 &\text{if}\ t\in[T_+(H),\infty)\text{ a.e.},
		\end{cases}
	\]
	We call the interval $(T_+(H),\infty)$ indivisible (of course without assigning a type to it), and let 
	$I_{\rm reg}$ be the set of all points $t\in(0,\infty)$ which are not inner point of an indivisible interval.
\end{definition}

\begin{lemma}
	\label{lem:972}
	The map 
	\[
		T_+:\left\{
		\begin{array}{rcl}
			\bbHen & \to & [0,\infty]
			\\[1mm]
			H & \mapsto & T_+(H)
		\end{array}
		\right.
	\]
	is continuous, and $\bbHen$ is compact.
\end{lemma}
\begin{proof}
	We show that $\bbHen$ is closed in $\bbHle_{0,\infty}$. 
	To this end, let $H\in\bbHle_{0,\infty}$ and $(H_n)_{n\in\bbN}$ a sequence in $\bbHen$ with 
	$\lim_{n\in\bbN}H_n=H$. Then, in particular, for all $0\leq x<y<\infty$ we have 
	\[
		\lim_{n\to\infty}\int_x^y\tr H_n(t)dt=\int_x^y\tr H(t)dt
		.
	\]
	Choose a subsequence $(T_+(H_{n_k}))_{k\in\bbN}$ which converges to some number $\tau\in[0,\infty]$. Then 
	\[
		\int_0^y\tr H(t)dt=y,\ y<\tau,\qquad \int_x^y\tr H(t)dt=0,\ \tau<x<y
		.
	\]
	Since $\tr H(t)\in[0,1]$ a.e., it follows that $\tr H=\mathds{1}_{(0,\tau)}$ a.e. Thus $H\in\bbHen$ and 
	$T_+(H)=\tau$. We also see that the number $\tau$ is independent of the chosen subsequence, and therefore 
	$\lim_{n\to\infty}T_+(H_n)=T_+(H)$. This shows that $T_+$ is continuous.
\end{proof}

\begin{definition}\label{def:984}
	We define a map\footnote{Here $\cP(\cdot)$ denotes the power set.} 
	$\Phi:\HB\times\bbHen\to\cP\big(\DB\cup\{\{0\}\}\big)$ as 
	\[
		\Phi(E,H):=\scrC(\cB(E))\cup\big\{\cB(E\ltimes W_H(t,\cdot))\mid t>0\big\}
		.
	\]
\end{definition}

The description of chains announced above now reads as follows.

\begin{theorem}\label{thm:983}
	The following statements hold. 
	\begin{enumerate}[{\rm(i)}]
	\item Let $(E,H)\in\HB\times\bbHen$. Then $\Phi(E,H)\in\Ch$, and $\Phi(E,H)$ is bounded if and only if $T_+(H)<\infty$. 
	\item If $E\in\HB$ and $H_1,H_2\in\bbHen$ are such that $\Phi(E,H_1)\subseteq\Phi(E,H_2)$, then 
		$T_+(H_1)\leq T_+(H_2)$ and $H_1=H_2\cdot\mathds{1}_{(0,T_+(H_1))}$ a.e.
		In particular, if $\Phi(E,H_1)=\Phi(E,H_2)$, then $H_1=H_2$ a.e.
	\item If $\scrC\in\Ch$ and $E\in\HB$ with $\cB(E)\in\scrC$, then there exists $H\in\bbHen$ such that $\scrC=\Phi(E,H)$.
	\item Assume that $\scrC\in\bCh$ and $E\in\HB$ with $\cB(E)\in\scrC$. Then $\cB(E)=\max\scrC$ if and only if 
		$\scrC=\Phi(E,0)$.
	\end{enumerate}
\end{theorem}

For later reference we state the essence for the proof of item (i) as a separate lemma.

\begin{lemma}
\label{lem:974}
	Let $(E,H)\in\HB\times\bbHen$. Define 
	\[
		\tilde H(t):=
		\begin{cases}
			H(t) &\text{if}\ 0>t,
			\\
			H_E(t) &\text{if}\ t<0,
		\end{cases}
	\]
	and let $\tilde E(t,z)$ be the solution of 
	\begin{align*}
		& \partial_t(\tilde A(t,z),\ \tilde B(t,z))J=z(\tilde A(t,z),\ \tilde B(t,z))\tilde H(t),
		\quad t\in\bbR\text{ a.e.},
		\\
		& (\tilde A(0,z),\ \tilde B(0,z))=(A(z),\ B(z)).
	\end{align*}
	Moreover, set 
	\[
		\tilde\cH_t:=
		\begin{cases}
			\cB(\tilde E(t,\cdot)) &\text{if}\ t>T_-(E),
			\\
			\{0\} &\text{if}\ t\leq T_-(E),
		\end{cases}
	\]
	and let $\chi$ be the map 
	\[
		\chi:\left\{
		\begin{array}{rcl}
			[T_-(E),T_+(H)]\,\setminus\,\{\infty\} & \to & \RK
			\\[1mm]
			t & \mapsto & \tilde\cH_t
		\end{array}
		\right.
	\]
	Then $\ran\chi=\Phi(E,H)$, and $\chi$ is an order isomorphism and homeomorphism onto its range.
	Moreover, $\Phi(E,H)$ is closed in $\RK$.
\end{lemma}
\begin{proof}
	By \Cref{thm:1} and the definition of $\Phi$ we have $\ran\chi=\Phi(E,H)$.
	Consider $T\in(T_-(E),T_+(H)]\setminus\{\infty\}$, then the restriction $\chi|_{[T_-(E),T]}$ is an order isomorphism and
	homeomorphism onto its image by \Cref{thm:1}. If $T_+(H)<\infty$, we can use $T:=T_+(H)$ and are done.

	Assume that $T_+(H)=\infty$. Then the above shows that $\chi$ is an order isomorphism and continuous. 
	It remains to show that $\chi^{-1}$ is continuous and that $\Phi(E,H)$ is closed. 
	Let $(t_n)_{n\in\bbN_0}$ be a sequence in $[T_-(E),\infty)$ such that the limit $\cH:=\lim_{n\to\infty}\tilde\cH_{t_n}$ 
	exists in $\RK$. Since $T_+(H)=\infty$, we have $\lim_{t\to\infty}\Delta_{\tilde\cH_t}(z)=\infty$ for all 
	$z\in\bbC\setminus\bbR$, and therefore $(t_n)_{n\in\bbN}$ must be bounded. Now the already settled case applies, and we
	obtain that the limit $t:=\lim_{n\to\infty}t_n$	exists and $\cH=\tilde\cH_t$. 
	We see that $\chi^{-1}$ is continuous and that $\ran\chi$ is closed.  
\end{proof}

\begin{proof}[Proof of \Cref{thm:983}]
\phantom{}
\begin{Elist}
\item 
	Item (i) of the theorem is immediate from \Cref{lem:974}.
	Since $\Phi(E,H)$ is order isomorphic to the interval $[T_-(E),T_+(H)]\setminus\{\infty\}$, the properties 
	(ii), (iv) in \Cref{def:2} hold, and $\Phi(E,H)$ has a maximal element if and only if $T_+(H)<\infty$. 
	Property (iii) in \Cref{def:2} is directly from the lemma, and property (v) from the definition of $\Phi$. 
	To see \Cref{def:2}(i), note that $\chi(T_-(E))=\{0\}$ and $\chi(0)=\cB(E)\neq\{0\}$. 

	Also item (iv) of the theorem is easy to see. First note that the definition of $\Phi$ ensures that 
	$\Phi(E,0)$ has a largest element, namely $\cB(E)$. On the other hand, if $\cB(E)$ is the largest element of 
	$\scrC$, then $\scrC=\Phi(E,0)$ by \Cref{lem:991}.
\item 
	In this step we establish the uniqueness statement (ii). Assume we are given 
	$E\in\HB$ and $H_1,H_2\in\bbHen$ such that $\Phi(E,H_1)\subseteq\Phi(E,H_2)$. 
	Let $t_1\in[0,T_+(H_1)]\setminus\{\infty\}$, $t_2\in[0,T_+(H_2)]\setminus\{\infty\}$, and denote 
	$E_j:=E\ltimes W_{H_j}(t_j,.)$ for $j\in\{1,2\}$. Assume that $\cB(E_1)=\cB(E_2)$, then 
	\Cref{cor:977} implies that, for $j\in\{1,2\}$, 
	\begin{align*}
		H_{E_j}(t)= &\, 
		\begin{cases}
			H_j(t+t_j) &\text{if}\ -t_j\leq t\leq 0\text{ a.e.},
			\\
			H_E(t+t_j) &\text{if}\ t<-t_j\text{ a.e.},
		\end{cases}
		\\
		E_j(t,\cdot)= &\, 
		\begin{cases}
			E\ltimes W_{H_j}(t+t_j,\cdot) &\text{if}\ -t_j\leq t\leq 0,
			\\
			E(t+t_j,\cdot) &\text{if}\ t<-t_j.
		\end{cases}
	\end{align*}
	The uniqueness part of \Cref{cor:975} implies that $E_1=E_2$ and $t_1=t_2$.
	Uniqueness of the structure Hamiltonian now implies that $H_1|_{(0,t_1)}=H_2|_{(0,t_1)}$ a.e. 

	We choose an increasing sequence $t_{1,n}\in[0,T_+(H_1)]\setminus\{\infty\}$ with $\lim_{n\to\infty}t_{1,n}=T_+(H_1)$,
	and apply what we showed above. This yields $H_1|_{[0,T_+(H_1)]}=H_2|_{[0,T_+(H_1)]}$ a.e., and in turn 
	$T_+(H_1)\leq T_+(H_2)$ and $H_1=H_2\cdot\mathds{1}_{[0,T_+(H_1)]}$.
\item 
	The last step is to prove the existence result (iii).
	Assume we are given $\scrC\in\Ch$ and $E\in\HB$ with $\cB(E)\in\scrC$. 

	Let $\cH\in\scrC$ with $\cB(E)\sqsubseteq\cH$. \Cref{cor:975} provides us with 
	$E_\cH\in\HB$ and $t_\cH\geq 0$ such that 
	\[
		\cH=\cB(E_\cH),\quad E=E_\cH(-t_\cH,\cdot)
		.
	\]
	Define (a.e.)
	\[
		H_{\cH}(t):=
		\begin{cases}
			H_{E_\cH}(t-t_\cH) &\text{if}\ 0<t\leq t_\cH,
			\\
			0 &\text{if}\ t>t_\cH,
		\end{cases}
	\]
	then $H_\cH\in\bbHen$ and $T_+(H_\cH)=t_\cH$. 

	For each $s\in[-t_\cH,0]$ we have 
	\[
		E_\cH(s,\cdot)=E\ltimes W_{H_{E_\cH}}(-t_\cH,s,\cdot)=E\ltimes W_{H_\cH}(s+t_\cH,\cdot)
		.
	\]
	This relation, together with \Cref{lem:991} applied with $\cB(E)$ and $\cH$ and \Cref{thm:1} with $E_\cH$, yields that 
	\begin{align*}
		\Phi(E,H_\cH) = &\, 
		\scrC(\cB(E))\cup\big\{\cB(E\ltimes W_{H_\cH}(t,\cdot))\mid t>0\big\}
		\\
		= &\, 
		\scrC(\cB(E))\cup\big\{\cB(E_\cH(s,\cdot))\mid -t_\cH<s\leq 0\big\}
		\\
		= &\, 
		\big\{\cL\in\scrC\mid \cL\sqsubseteq\cB(E)\big\}
		\cup\big\{\cL\in\scrC\mid \cB(E)\sqsubseteq\cL\sqsubseteq\cH\big\}
		=\big\{\cL\in\scrC\mid \cL\sqsubseteq\cH\big\}
		.
	\end{align*}
	If $\scrC$ is bounded, we can use $\cH=\max\scrC$ and are done. If $\scrC$ is unbounded, we have to make a
	limit construction. 

	We start with observing a monotonicity property. Assume that $\cH,\cH'\in\scrC$ with 
	$\cB(E)\sqsubseteq\cH\sqsubsetneq\cH'$. Then $\Phi(E,H_\cH)\subsetneq\Phi(E,H_{\cH'})$, and by the already established
	item (ii) thus
	\begin{equation}
	\label{eq:962}
		T_+(H_\cH)<T_+(H_{\cH'}),\quad H_\cH=H_{\cH'}\cdot\mathds{1}_{[0,T_+(H_\cH)]}
		.
	\end{equation}
	Set 
	\[
		T:=\sup\big\{T_+(H_\cH)\mid \cH\in\scrC,\cB(E)\sqsubseteq\cH\big\}\in[0,\infty]
		.
	\]
	This supremum is not attained since $\scrC$ has no largest element. Choose a sequence $(\cH_n)_{n\in\bbN}$ of spaces 
	$\cH_n\in\scrC$ such that 
	\[
		\cB(E)\sqsubseteq\cH_1\sqsubsetneq\cH_2\sqsubsetneq\cdots,\qquad \lim_{n\to\infty}T_+(\cH_n)=T
		.
	\]
	Due to \cref{eq:962} an element $H$ in $\bbHen$ is (a.e.) well-defined by
	\[
		H(t):=
		\begin{cases}
			H_{\cH_n}(t) &\text{if}\ 0\leq t\leq T_+(\cH_n),\ n\in\bbN,
			\\
			0 &\text{if}\ t\geq T.
		\end{cases}
	\]
	Let $\cH\in\scrC$ with $\cB(E)\sqsubseteq\cH$, and choose $n\in\bbN$ with $T_+(H_\cH)\leq T_+(H_n)$. Then 
	$\cH\sqsubseteq\cH_n$, and we find $s\in[0,T_+(\cH_n)]$ such that 
	\[
		\cH=\cB(E\ltimes W_{H_{\cH_n}}(s,\cdot))=\cB(E\ltimes W_H(s,\cdot))\in\Phi(E,H)
		.
	\]
	We see that $\scrC\subseteq\Phi(E,H)$. \Cref{cor:987} implies that equality holds. 
\end{Elist}
\end{proof}

Based on the above theorem we may introduce the following notation. 

\begin{definition}
	\label{def:971}
	Let $\scrC\in\Ch$ and $E\in\HB$ such that $\cB(E)\in\scrC$. Let $H\in\bbHen$ be the (a.e.) unique element 
	such that $\scrC=\Phi(E,H)$. Then we denote 
	\[
	H_{E,\scrC}(t):=
	\begin{cases}
	H_E(t) &\text{if}\ t<0,
	\\
	H(t) &\text{if}\ t>0.
	\end{cases}
	\]
\end{definition}

We have the analogue to \Cref{thm:2}(ii).

\begin{corollary}
\label{cor:947}
	Let $\scrC\subseteq\DB\cup\{\{0\}\}$. Then $\scrC$ is an unbounded chain if and only if there exist 
	$-\infty\leq a<b\leq\infty$ and $\chi:[a,b)\to\scrC$, such that $\chi$ is a homeomorphism, preserves order, and the
	limit $\lim_{t\to b}\chi(t)$ does not exist in $\RK$. 
\end{corollary}

%
\subsection{Convergence of chains}
\label{sec:4-3}
%

We introduce a notion of convergence of chains. 

\begin{definition}\label{def:982}
	Let $(\scrC_j)_{j\in J}$ be a net in $\Ch$ and $\scrC\in\Ch$. Then we say that $(\scrC_j)_{j\in J}$ \emph{converges to} 
	$\scrC$, and write $\scrC_j\rightsquigarrow\scrC$, if 
	\begin{enumerate}[{\rm(i)}]
		\item for every $\cH\in\scrC$ there exists a net $(\cH_j)_{j\in J}$ of spaces $\cH_j\in\scrC_j$, such that 
		$\lim_{j\in J}\cH_j=\cH$;
		\item for every subnet $h:K\to J$, $(\cH_{h(k)})_{k\in K}$ of spaces $\cH_{h(k)}\in\scrC_{h(k)}$ which 
		converges in $\RK$, the limit belongs to $\scrC$. 
	\end{enumerate}
\end{definition}

We intentionally do not use the notation ``$\lim$'' since we do not know if this notion of convergence comes from a topology. 
However, convergence does transfer to subnets and limits are unique.

\begin{lemma}\label{lem:965}
	Let $(\scrC_j)_{j\in J}$ be a net in $\Ch$. 
	\begin{enumerate}[{\rm(i)}]
	\item Let $\scrC\in\Ch$. If $\scrC_j\rightsquigarrow\scrC$ and $h:K\to J$ is a subnet, then also 
		$\scrC_{h(k)}\rightsquigarrow\scrC$.
	\item Let $\scrC_1,\scrC_2\in\Ch$. If $\scrC_j\rightsquigarrow\scrC_1$ and $\scrC_j\rightsquigarrow\scrC_2$, 
		then $\scrC_1=\scrC_2$.
	\end{enumerate}
\end{lemma}
\begin{proof}
	The assertion in (i) is clear. We come to the proof of (ii).
	Let $\cH\in\scrC_1$. Since $\scrC_j\rightsquigarrow\scrC_1$ we find a net $(\cH_j)_{j\in J}$ with $\cH_j\in\scrC_j$ and 
	$\lim_{j\in J}\cH_j=\cH$ in $\RK$. Since $\scrC_j\rightsquigarrow\scrC_2$, it follows that $\cH\in\scrC_2$. This shows
	that $\scrC_1\subseteq\scrC_2$. Exchanging the roles of $\scrC_1$ and $\scrC_2$ yields the reverse inclusion. 
\end{proof}

\begin{example}
\label{ex:961}
	Let $\scrC\in\Ch$. Then $\scrC$ endowed with $\sqsubseteq$ is (in particular) a directed set. The net 
	$(\scrC(\cH))_{\cH\in\scrC}$ converges to $\scrC$. This follows easily: in item (i) of \Cref{def:982} we can take a net
	which is constant from some index, and (ii) holds because $\scrC$ is closed. 
\end{example}

In the following theorem we make the connection between the abstract notion of convergence introduced above, and the concrete
realization of chains from \Cref{thm:983}.

\begin{theorem}\label{thm:980}
	Let $(\scrC_j)_{j\in J}$ be a net in $\Ch$ and $\scrC\in\Ch$. Then the following statements are equivalent.
	\begin{enumerate}[{\rm(i)}]
	\item $\scrC_j\rightsquigarrow\scrC$
	\item There exist $E\in\HB$ and $E_j\in\HB$, $j\in J$, such that 
		\begin{equation}
		\label{eq:964}
			\cB(E)\in\scrC,\ \cB(E_j)\in\scrC_j,j\in J,\qquad
			\lim_{j\in J}E_j=E,
		\end{equation}
		\begin{equation}
		\label{eq:963}
			\lim_{j\in J}H_{E_j,\scrC_j}=H_{E,\scrC}
			.
		\end{equation}
	\item There exist $E\in\HB$ and $E_j\in\HB$, $j\in J$, with \cref{eq:964}. For every choice of $E$ and $E_j$ with 
		\cref{eq:964} the limit relation \cref{eq:963} holds.
	\end{enumerate}
\end{theorem}
\begin{proof}
	The implication ``(iii)$\Rightarrow$(ii)'' is trivial.
\begin{Elist}
\item 
	We show that ``(ii)$\Rightarrow$(i)''.
	Assume that we have $E,E_j\in\HB$ with \cref{eq:964} and \cref{eq:963}. 
	Let notation $\tilde H_j,\tilde H$ and $\tilde\cH_{j,t},\tilde\cH_t$ and $\chi_j,\chi$ be as in \Cref{lem:974} 
	for $(E_j,H_{E_j,\scrC_j})$ and $(E,H_{E,\scrC})$, respectively. 
	By \Cref{prop:3} the relation \cref{eq:964} implies that also $\lim_{j\in J}H_{E_j}=H_E$, and it follows that 
	\[
		\lim_{j\in J}W_{\tilde H_j}(t,0,\cdot)=W_{\tilde H}(t,0,\cdot),t<0,\quad
		\lim_{j\in J}W_{\tilde H_j}(0,t,\cdot)=W_{\tilde H}(0,t,\cdot),t>0
		.
	\]
	Combining this with $\lim_{j\in J}E_j=E$ yields
	\[
		\forall t\in\bbR:\ \lim_{j\in J}\tilde\cH_{j,t}=\tilde\cH_t
		.
	\]
	In particular, \Cref{def:982}\,(i) is satisfied. 

	Now assume we have a subnet $h:K\to J$ and $\cH_{h(k)}\in\scrC_{h(k)}$, $\cH\in\RK$, such that 
	$\cH=\lim_{k\in K}\cH_{h(k)}$ in $\RK$. Our aim is to show that $\cH\in\scrC$. Let 
	\[
		t(k)\in\big[T_-(E_{h(k)}),T_+(H_{E_{h(k)},\scrC_{h(k)}})\big]\,\setminus\{\infty\}
	\]
	be such that $\cH_{h(k)}=\tilde\cH_{h(k),t(k)}$. 
	By passing to a further subnet if necessary, we may assume that the limit 
	$t:=\lim_{k\in K}t(k)$ exists in $[-\infty,\infty]$. First consider the case that $t\in(-\infty,\infty)$.
	Using the convergence given by \cref{eq:964} and \cref{eq:963} in the same way as in the previous paragraph leads to 
	\[
		\cH=\lim_{k\in K}\tilde\cH_{h(k),t(k)}=\tilde\cH_t\in\scrC
		.
	\]
	Second, assume that $t=-\infty$. For each $s\in\bbR$ we find $k_0\in K$ such that
	$\tilde\cH_{h(k),t(k)}\sqsubseteq\tilde\cH_{h(k),s}$ for all $k\geq k_0$ . 
	Passing to the limit yields $\cH\subseteq_c\tilde\cH_s$, and hence 
	$\Delta_\cH(w)\leq\Delta_{\tilde\cH_s}(w)$ for all $w\in\bbC$. We have $\lim_{s\to-\infty}\Delta_{\tilde\cH_s}(w)=0$,
	and obtain 
	\[
		\cH=\{0\}\in\scrC
		.
	\]
	Finally, we are going to rule out the case that $t=\infty$. If we had $t=\infty$, then we find for each $s\in\bbR$ 
	an index $k_0\in K$ such that $\tilde\cH_{h(k),s}\sqsubseteq\tilde\cH_{h(k),t(k)}$ for all $k\geq k_0$. 
	Passing to the limit yields $\tilde\cH_s\subseteq_c\cH$, and hence $\Delta_{\tilde\cH_s}(w)\leq\Delta_\cH(w)$ 
	for all $w\in\bbC$. By continuity of $T_+$, we have $T_+(H_{E,\scrC})=\infty$, i.e., $\scrC$ is an unbounded chain. 
	This implies that $\lim_{s\to\infty}\Delta_{\tilde\cH_s}(w)=\infty$ whenever $w$ is nonreal, 
	and we have reached a contradiction. 
\item 
	We show that ``(i)$\Rightarrow$(iii)''.
	Assume that $\scrC_j\rightsquigarrow\scrC$ and pick $\cH\in\scrC$. By \Cref{def:982}\,(i) we find $\cH_j\in\scrC_j$ so that $\lim_{j\in J}\cH_j=\cH$. By \Cref{prop:1}\,(iv), $\cB$ has a continuous right inverse and we find $E_j, E$ with $\cB(E_j)=\cH_j, \cB(E)=\cH$ and $\lim_{j\in J}E_j=E$. Our aim is to show
	that, for each such choice of $E,E_j$, it holds that
	$\lim_{j\in j}H_{E_j,\scrC_j}=H_{E,\scrC}$ in $\bbHen$. Since $\bbHen$ is compact, it suffices to
	evaluate limits of convergent subnets. Hence, assume we have $h:K\to J$ and $H\in\bbHen$ such that 
	$\lim_{k\in K}H_{E_{h(k)},\scrC_{h(k)}}=H$. Set $\scrC':=\Phi(E,H)$, then $\scrC_{h(k)}\rightsquigarrow\scrC'$ by what
	we aready proved in the first step. Remembering \Cref{lem:965}\,(i), we find $\scrC_{h(k)}\rightsquigarrow\scrC$, and 
	now \Cref{lem:965}\,(ii) implies 
	\[
		\Phi(E,H)=\scrC'=\scrC=\Phi(E,H_{E,\scrC})
		.
	\]
	\Cref{thm:983}{\rm(ii)} yields $H=H_{E,\scrC}$. 
\end{Elist}
\end{proof}

As a corollary we obtain that on the set of bounded chains convergence can be characterized in a simple (in particular
metrizable) way. 

\begin{corollary}\label{lem:981}
	Let $(\scrC_j)_{j\in J}$ be a net in $\bCh$, and $\scrC\in\bCh$. Then $\scrC_j\rightsquigarrow\scrC$ if and only if
	$\lim_{j\in J}(\max\scrC_j)=\max\scrC$.
\end{corollary}
\begin{proof}
	Assume first that $\lim_{j\in J}(\max\scrC_j)=\max\scrC$. Choose $E_j,E\in\HB$ such that 
	\[
		\max\scrC_j=\cB(E_j),\ \max\scrC=\cB(E),\quad \lim_{j\in J}E_j=E
		,
	\]
	Since $H_{E_j,\scrC_j}=H_{E,\scrC}=0$, it follows that $\scrC_j\rightsquigarrow\scrC$.
	
	Conversely, assume that $\scrC_j\rightsquigarrow\scrC$. Choose $E\in\HB$ with $\cB(E)=\max\scrC$ and 
	$E_j\in\HB$ such that $\cB(E_j)\in\scrC_j$ and $\lim_{j\in J}E_j=E$. Then 
	$\lim_{j\in J}H_{E_j,\scrC_j}=H_{E,\scrC}=0$. 
	It follows that $T_+(H_{E_j,\scrC_j})\to 0$, and hence 
	\[
		\lim_{j\in J}\big(E_j\ltimes W_{H_{E_j,\scrC_j}}(T_+(H_{E_j,\scrC_j}),\cdot)\big)
		=E\ltimes I=E
		.
	\]
	It remains to note that $\cB(E_j\ltimes W_{H_{E_j,\scrC_j}}(T_+(H_{E_j,\scrC_j}),\cdot))=\max\scrC_j$.
\end{proof}

The next result allows us to conclude convergence of arbitrary (also unbounded) chains when a candidate for the limit is
guessed. We state a variant which is sufficient for our later needs.

\begin{proposition}
\label{927}
	Let $(\scrC_j)_{j\in J}$ be a net in $\Ch$, and let $\chi:[0,\infty)\to\RK$. Assume that $\chi$ is continuous with 
	$\chi(0)=\{0\}$, that $a\mapsto K_{\chi(a)}(0,0)$ is strictly increasing with 
	$\lim_{a\to\infty}K_{\chi(a)}(0,0)=\infty$, and that there exists $\cH_{a,j}\in\scrC_j$ for $a>0$ and $j\in J$, such that 
	\[
	\forall a>0: \lim_{j\in J}\cH_{a,j}=\chi(a)
	.
	\]
	Then $\scrC:=\chi([0,\infty))\in\ubCh$ and $\scrC_j\rightsquigarrow\scrC$. 
\end{proposition}
\begin{proof}
	First of all note that all spaces $\chi(a)$ with $a>0$ are $\neq\{0\}$ and, being limits of dB-spaces, 
	belong to $\DB$. 
	
	Let $b\in(0,\infty)$. For $a\in(0,b)$ we have 
	\[
		\lim_{j\in J}K_{\cH_{a,j}}(0,0)=K_{\chi(a)}(0,0)<K_{\chi(b)}(0,0)=\lim_{j\in J}K_{\cH_{b,j}}(0,0)
		,
	\]
	and hence find $j_0\in J$ such that 
	\[
		\forall j\in J,j\geq j_0: K_{\cH_{a,j}}(0,0)<K_{\cH_{b,j}}(0,0)
		.
	\]
	The function 
	\[
		\left\{
		\begin{array}{rcl}
			\scrC_j & \to & [0,\infty)
			\\
			\cH & \mapsto & K_\cH(0,0)
		\end{array}
		\right.
	\]
	is nondecreasing, and we conclude that $\cH_{a,j}\sqsubseteq\cH_{b,j}$ for all $j\geq j_0$. In other words, it holds
	that $\cH_{a,j}\in\scrC(\cH_{b,j})$ for such $j$. 
	By \Cref{lem:981} we have 
	\[
		\scrC(\cH_{b,j})\rightsquigarrow\scrC(\chi(b))
		.
	\]
	Since $\chi(a)$ is the limit of the subnet $(\cH_{a,j})_{\substack{j\in J\\ j\geq j_0}}$, 
	it follows that $\chi(a)\in\scrC(\chi(b))$. We see that 
	\[
		\big\{\chi(a)\mid a\in[0,b]\big\}\subseteq\scrC(\chi(b))
		.
	\]
	Let $\cH\in\scrC(\chi(b))$. Then $K_\cH(0,0)\leq K_{\chi(b)}(0,0)$, and thus we find $a\in[0,b]$ such that 
	$K_\cH(0,0)=K_{\chi(a)}(0,0)$. If $a'\in[0,a)$, then $K_{\chi(a')}(0,0)<K_\cH(0,0)$. Since $\chi(a')\in\scrC(\chi(b))$
	it follows that $\chi(a')\sqsubseteq\cH$. Similarly, we obtain that $\cH\sqsubseteq\chi(a')$ for all $a'\in(a,b]$. 
	
	Consider the case that $0<K_\cH(0,0)<K_{\chi(b)}(0,0)$. Then $a\in(0,b)$, and continuity of $\chi$ yields 
	\[
		\chi(a)=\lim_{a'\uparrow a}\chi(a')\subseteq_c\cH\subseteq_c\lim_{a'\downarrow a}\chi(a')=\chi(a)
		,
	\]
	i.e., $\cH=\chi(a)$. If $K_\cH(0,0)=0$, we have 
	\[
		\{0\}\subseteq_c\cH\subseteq_c\lim_{a'\downarrow 0}\chi(a')=\{0\}
		,
	\]
	and if $K_\cH(0,0)=K_{\chi(b)}(0,0)$, then 
	\[
		\chi(b)=\lim_{a'\uparrow b}\chi(a')\subseteq_c\cH\subseteq_c\chi(b)
		.
	\]
	Thus, in every case, $\cH=\chi(a)$. We conclude that 
	\begin{equation}
	\label{941}
		\forall b\in(0,\infty): \scrC(\chi(b))=\big\{\chi(a)\mid a\in[0,b]\big\}
		.
	\end{equation}
	We can now check that $\scrC\in\ubCh$. The properties {\rm(i)}, {\rm(ii)}, {\rm(iv)}, {\rm(v)} of \Cref{def:2} are clear
	from \cref{941}. To show that $\scrC$ is closed, it is enough to note that for any convergent net $(\chi(a_i))_{i\in I}$
	the net $(a_i)_{i\in I}$ is eventually bounded since $\lim_{a\to\infty}K_{\chi(a)}(0,0)=\infty$, and that $\chi$ is
	continuous. 
	
	It remains to show that $\scrC_j\rightsquigarrow\scrC$. Property {\rm(i)} of \Cref{def:982} holds directly 
	by the present assumption. Assume we have a convergent subnet as in {\rm(ii)} of this definition, say, 
	$\lim_{k\in K}\cH_{h(k)}=:\cH$ where $\cH_{h(k)}\in\scrC_{h(k)}$. We argue in the same way as above. 
	Let $a\in[0,\infty)$ be such that $K_{\chi(a)}(0,0)=K_{\cH}(0,0)$. For $a'\in[0,a)$ we have 
	\[
		\lim_{j\in J}K_{\cH_{a'j}}(0,0)=K_{\chi(a')}(0,0)<K_{\cH}(0,0)=\lim_{k\in K}K_{\cH_{h(k)}}(0,0)
		,
	\]
	and hence there exists $k_0\in K$ such that $\cH_{a',h(k)}\sqsubseteq\cH_{h(k)}$ for all $k\geq k_0$. We obtain 
	\[
		\chi(a')=\lim_{\substack{k\in K\\ k\geq k_0}}\cH_{a',h(k)}\subseteq_c
		\lim_{\substack{k\in K\\ k\geq k_0}}\cH_{h(k)}=\cH.
	\]
	Similarly, $\cH\subseteq_c\chi(a')$ for all $a'\in(a,\infty)$. Continuity of $\chi$ yields $\cH=\chi(a)$. 
\end{proof}

%
%
%
\newpage
\section{Measures associated to unbounded chains}
\label{sec:5}
%
%
%

In the previous section we saw that bounded chains correspond to de~Branges spaces: by \Cref{thm:2} and \Cref{lem:981} the maps 
\[
	\left\{
	\begin{array}{rcl}
		\bCh & \to & \DB
		\\
		\scrC & \mapsto & \max\scrC
	\end{array}
	\right.
	,\qquad
	\left\{
	\begin{array}{rcl}
		\DB & \to & \bCh
		\\
		\cH & \mapsto & \scrC(\cH)
	\end{array}
	\right.
\]
are mutually inverse bijections and both preserve convergence. 

For unbounded chains the situation is much more complex. The substitute for the set $\DB$ above is 
\[
	\bbM:=\big\{\mu\mid \mu\text{ positive Borel measure on }\bbR\big\}
	,
\]
and a map $\ubCh\to\bbM$ can be constructed, cf.\ \Cref{thm:900} below. This map is surjective and preserves convergence, 
but it is not anymore injective.

In order to simplify the presentation we restrict all considerations to chains with $\vartheta_{\scrC}=0$; 
treating the general case is not necessary for our purposes, and would involve some technical complications. 

%
\subsection{The direct problem}
\label{sec:5-1}
%

Inclusions of a space of entire functions in a space $L^2(\mu)$ are understood via the restriction map 
$F\mapsto F|_{\bbR}$ $\mu$-a.e. (which is often, but not always, injective). 

\begin{definition}\label{def:901}
	Let $\scrC\in\ubCh$ with $\vartheta_\scrC=0$. A measure $\mu\in\bbM$ is called a \emph{spectral measure for $\scrC$}, if 
	\[
		\forall\cH\in\scrC: \cH\sqsubseteq L^2(\mu).
	\]
\end{definition}

The following theorem is again a slight addition to the results shown by de~Branges.

\begin{theorem}
\label{thm:900}
	\phantom{}
	\begin{enumerate}[{\rm(i)}]
	\item
	Let $\scrC\in\ubCh$ with $\vartheta_\scrC=0$. Then there exists a unique spectral measure for $\scrC$, and we 
	denote this measure as $\mu_\scrC$. 
	\item Let $(E,H)\in\HB^*\times\bbHe_{0,\infty}$. Then
	\[
		q_{E,H}:=J\begin{pmatrix} A & B\\ -B & A\end{pmatrix}\star q_H\in\cN_0
	\]
	and the measure in the integral representation of $q_{E,H}$ is $|E(x)|^2d\mu_{\Phi(E,H)}(x)$. 
	\end{enumerate}
\end{theorem}
\begin{proof}
	Let $\scrC\in\ubCh$ and choose $(E,H)\in\HB^*\times\bbHe_{0,\infty}$ with $\scrC=\Phi(E,H)$. Let $H_{E,\scrC}$ be the
	Hamiltonian from \Cref{def:971}, let $I_{\rm reg}$ refer to $H_{E,\scrC}$, and denote by $E_t=A_t-iB_t$ the solution of 
	\begin{align*}
		& \partial_t(A_t(z),\ B_t(z))J=z(A_t(z),\ B_t(z))H_{E,\scrC}(t),\quad t\in\bbR,
		\\
		& (A_0(z),\ B_0(z))=(A(z),\ B(z)).
	\end{align*}
	For $t>T_-(E)$ let $H_t\in\bbHe_{0,\infty}$ be the Hamiltonian 
	\[
		H_t(s):=H_{E,\scrC}(s+t),\quad s\in(0,\infty)
		.
	\]
	Then $\scrC=\Phi(E_t,H_t)$ for all $t>T_-(E)$. By \cite[Problem~158]{debranges:1968} there exists a measure $\nu$ such
	that for all $t>T_-(E)$ the measure in the integral representation of $q_{E_t,H_t}$ is $|E_t|^2d\nu$. 

	At this point we split the argument
	distinguishing the cases whether $I_{\rm reg}$ is bounded from above or not. If $\sup I_{\rm reg}<\infty$ we
	use $t:=\sup I_{\rm reg}$ to show that $\nu$ is the unique spectral measure for $\scrC$, and if 
	$\sup I_{\rm reg}=\infty$ we refer to \cite[Problem~163]{debranges:1968} to obtain a unique spectral measure and then
	show that this measure equals $\nu$. 
\begin{Elist}
\item 
	Assume that $t:=\sup I_{\rm reg}<\infty$. Let $\alpha\in\bbR$ be such that $H_t=e_\alpha e_\alpha^*$ a.e., set
	\[
		G:=(A_t,\ B_t)e_\alpha=A_t\cos\alpha+B_t\sin\alpha,\quad 
		L:=(A_t,\ B_t)e_{\alpha+\frac\pi2}=-A_t\sin\alpha+B_t\cos\alpha
		.
	\]
	Since $\vartheta_E=0$, the functions $G$ and $L$ have no common zeros, and since $\frac BA\in\cN_0$ also 
	$\frac LG\in\cN_0$. In particular, this implies that both functions $L$ and $G$ have only real and simple zeros. Write 
	\[
		\cH:=
		\begin{cases}
			\cB(E_t) &\text{if}\ t>T_-(E),
			\\
			\{0\} &\text{if}\ t=T_-(E),
		\end{cases}
	\]
	then we have for all $s>t$
	\[
		\cB(E_s)=\cH\oplus\spann\{G\}\ \text{with}\ \|G\|_{\cB(E_s)}^2=\frac 1{s-t},\qquad \cB(E_s)^\flat=\cH
		,
	\]
	cf.\ \cref{eq:50}. A direct computation shows that 
	we have $q_{E_t,H_t}=\frac{L}G$. Hence, $\nu$ is discrete and supported on the zero set of $G$. If $t>T_-(E)$ we have $\cH\subseteq_i L^2(\nu)$ by 
	\cite[Theorem~22]{debranges:1968}, if $t=T_-(E)$ the same relation holds trivially. Moreover, obviously, 
	$\int_\bbR|G|^2d\nu=0$. Thus $\nu$ is a spectral measure for $\scrC$. Conversely, if $\tilde\nu$ is a spectral measure
	for $\scrC$, then we must have 
	\[
		\int_{\bbR}|G|^2d\tilde\nu\leq\|G\|_{\cB(E_s)}^2,\quad s>t
		,
	\]
	and hence $\int_\bbR|G|^2d\tilde\nu=0$. Thus $\tilde\nu$ is discrete with support contained in the zero set of $G$. 
	If $x\in\bbR$ with $G(x)=0$, we have $\frac{G(z)}{z-x}\in\cH$, and hence may evaluate 
	\begin{align*}
		\tilde\nu(\{x\})= &\, \frac 1{G'(x)^2}\int_\bbR\Big|\frac{G(y)}{y-x}\Big|^2d\tilde\nu(y)
		\\
		= &\, \frac 1{G'(x)^2}\Big\|\frac{G(z)}{z-x}\Big\|^2=
		\frac 1{G'(x)^2}\int_\bbR\Big|\frac{G(y)}{y-x}\Big|^2d\nu(y)=\nu(\{x\})
		.
	\end{align*}
\item 
	We invoke \cite[Problem~163]{debranges:1968} which tells us that there exists an unique measure $\mu$ such that 
	\[
		\forall t\in I_{\rm reg}: \cB(E_t)\subseteq_i L^2(\mu)
		.
	\]
	We observe that $\mu$ is the unique spectral measure for $\scrC$. 
	Given $s\in\bbR$, we can choose $t\in I_{\rm reg}$ with $t\geq s$, and it follows that 
	\[
		\cB(E_s)\sqsubseteq\cB(E_t)\subseteq_i L^2(\mu)
		.
	\]
	On the other hand, if $\tilde\mu$ is a spectral measure for $\scrC$, 
	and $t\in I_{\rm reg}$, choose $s\in I_{\rm reg}$ with $s>t$. Then 
	\[
		\cB(E_t)\subseteq_i\cB(E_s)^\flat\subseteq_i L^2(\tilde\mu)
		,
	\]
	and it follows that $\tilde\mu=\mu$. 

	In order to identify $\mu$, we provide an auxiliary argument. 
	Let $t\geq 0$. Since $\cB(E)\subseteq\cB(E_t)$, \cite[Theorem~27]{debranges:1968}
	provides us with a matrix function $M_t(z)$ such that 
	\begin{enumerate}[{\rm(i)}]
	\item $M_t$ has real and entire entries, 
	\item $(1,0)M_t=(A_t,B_t)$ and the kernel $\frac{M_t(z)JM_t(w)^*-E(z)JE(\overline w)}{z-\overline w}$ is positive
		definite,
	\item $\lim_{y\to\infty}\frac 1yJM_t(iy)\star i=0$. 
	\end{enumerate}
	Set 
	\[
		\tilde M_t(z):=\begin{pmatrix} A & B\\ -B & A \end{pmatrix}W_{H_{E,\scrC}}(t,.)
		,
	\]
	then $\tilde M_t$ also has the properties {\rm(i)}, {\rm(ii)}. The functions $JM_t\star i$ and $J\tilde M_t\star i$
	belong to $\cN_0$, are continuous along $\bbR$, and a computation shows that 
	\[
		\Im(JM_t\star i)(x)=\Im(J\tilde M_t\star i)(x)=\Big|\frac{E(x)}{E(t,x)}\Big|^2,\quad x\in\bbR
		.
	\]
	Thus, we find $\alpha_t,\beta_t\in\bbR$ such that $(J\tilde M_t\star i)(z)=\alpha_t+\beta_tz+(JM_t\star i)(z)$, and 
	by \cite[Theorem 4.22]{remling:2018}
	\[
		\tilde M_t(z)=\begin{pmatrix} 1 & 0\\ -(\alpha_t+\beta_tz) & 1\end{pmatrix}M_t(z)
		.
	\]
	It follows that for all $q\in\cN_0$ 
	\[
		(J\tilde M_t\star q)(z)=\alpha_t+\beta_t z+(JM_t\star q)(z)
		,
	\]
	and hence that the measures in the integral representations of $J\tilde M_t\star q$ and $JM_t\star q$ coincide.

	Now fix $t\in I_{\rm reg}\cap(0,\infty)$. Then \cite[Problem~90]{debranges:1968} provides us with a function 
	$q_t\in\cN_0\cup\{\infty\}$ such that the measure in the Herglotz integral representation of 
	\[
		J\begin{pmatrix} A(t,.) & B(t,.)\\ -B(t,.) & A(t,.)\end{pmatrix}\star q_t
	\]
	is $|E(t,x)|^2d\mu(x)$. We obtain from \cite[Theorem~32]{debranges:1968} that the measure in the integral
	representation of $JM_t\star q_t$, and hence in the one of $J\tilde M_t\star q_t$, is $|E(x)|^2d\mu(x)$. 

	We have $\lim_{t\to\infty}[J\tilde M_t\star q_t]=q_{E,H}$, and since we can let $t\to\infty$ inside $I_{\rm reg}$ it
	follows that the measure in the integral representation of $q_{E,H}$ is $|E(x)|^2d\mu(x)$. 
\end{Elist}
\end{proof}

By means of \Cref{thm:900} we have a map 
\[
\ubCh\to\bbM,\quad \scrC\mapsto\mu_{\scrC}.
\]
We show that this map is surjective but not injective, and that it preserves convergence. 

\begin{proposition}
\label{pro:907}
	The following statements hold.
	\begin{enumerate}[{\rm(i)}]
	\item For each $\mu\in\bbM$ the set 
		\[
			\ubCh_\mu:=\big\{\scrC\in\ubCh\mid \vartheta_\scrC=0,\mu_\scrC=\mu\big\}
		\]
		has infinitely many elements.
	\item Let $(\scrC_j)_{j\in J}$ be a net in $\ubCh$, and $\scrC\in\ubCh$. If $\scrC_j\rightsquigarrow\scrC$, then 
		$\lim_{j\in J}\mu_{\scrC_j}=\mu_{\scrC}$ in the $w^*$-topology of $C_c(\bbR)^*$.
	\end{enumerate}
\end{proposition}
\begin{proof}
	For the proof of item (i) we observe that the construction from \Cref{lem:26} lifts to chains. 
	If $\scrC\in\Ch$, $\vartheta_\scrC=0$, and $C$ is a real and zerofree entire function, then 
	\[
		\frac 1C\cdot\scrC:=\big\{\big\{\tfrac FC\mid F\in\cH\big\}\mid \cH\in\scrC\big\}\in\Ch
		.
	\]
	The chains $\scrC$ and $\frac 1C\cdot\scrC$ are together bounded or unbounded. If $\scrC=\Phi(E,H)$ for some $E\in\HB^*$,
	$H\in\bbHen$, then $\frac 1C\cdot\scrC=\Phi(\frac EC,H)$. 
\begin{Elist}
\item We show that $\ubCh_\mu\neq\emptyset$ whenever $\int_{\bbR}e^{|t|}d\mu(t)<\infty$. Assuming this decay of $\mu$ ensures
	that $L^2(\mu)$ contains the set $\bbC[z]$ of all polynomials with complex coefficients as a dense linear
	subspace (e.g.\ \cite[Satz~5.2]{freud:1969}).

	Let $m\in\bbN$ and assume that $\supp\mu$ contains at least $m$ points. Then the space 
	\[
		\cH_m:=\big\{F\in\bbC[z]\mid \deg F<m\big\}
	\]
	becomes a dB-space when endowed with the $L^2(\mu)$-scalar product. We are going to fill up this sequence of spaces in
	order to obtain an unbounded chain. To this end denote by $p_n$, $n\in\bbN_0$, the orthonormal polynomials in
	$L^2(\mu)$, i.e., $p_n\in\bbC[z]$ with $\deg p_n=n$ and 
	\[
		(p_n,p_{n'})_{L^2(\mu)}=
		\begin{cases}
			1 &\text{if}\ n=n',
			\\
			0 &\text{if}\ n\neq n'.
		\end{cases}
	\]
	Then (here we set $\cH_0:=\{0\}$)
	\[
		\cH_m=\spann\{p_n\mid n\in\bbN_0,n<m\},\quad 
		\cH_m^\flat=\spann\{p_n\mid n\in\bbN_0,n<m-1\}
		.
	\]
	For $t\in(m-1,m)$ set 
	\[
		\cH_t:=\spann\{p_n\mid n\in\bbN_0,n<m\}
		,
	\]
	and, for $F\in\cH_m^\flat$ and $\alpha\in\bbC$, 
	\[
		\|F+\alpha p_{m-1}\|_{\cH_t}^2:=\|F\|_{\cH_m^\flat}^2+\frac{|\alpha|^2}{t-(m-1)}
		.
	\]
	Then $\cH_t$ is a dB-space, and $\cH_{m-1}\subseteq_i\cH_t\subseteq_c\cH_m$. Clearly, for all $m\in\bbN$, 
	\[
		\scrC(\cH_m)=\big\{\cH_t\mid t\in[0,m]\big\}
		.
	\]
	If $|\supp\mu|=\infty$, then $\scrC:=\{\cH_t\mid t\in[0,\infty)\}$ is an unbounded chain and 
	\[
		\forall t\in[0,\infty): \cH_t\sqsubseteq L^2(\mu)
		.
	\]
	Assume that $N:=|\supp\mu|<\infty$. Then we choose a polynomial $p$ with degree $N$ such that $p(t)=0$ for all 
	$t\in\supp\mu$, and define for $t>N$ the space $\cH_t$ as $\cH_t:=\spann\big(\cH_N\cup\{p\}\big)$ endowed with the norm 
	($F\in\cH_N$, $\alpha\in\bbC$)
	\[
		\|F+\alpha P\|_{\cH_t}^2:=\|F\|_{\cH_N}^2+\frac{|\alpha|^2}{t-N}
		.
	\]
	Then $\scrC:=\{\cH_t\mid t\in[0,\infty)\}$ is an unbounded chain and again $\mu_\scrC=\mu$. 

\item Let $\mu\in\bbM$ be given. Choose a continuous function $\omega:[0,\infty)\to[1,\infty)$ such that 
	\[
		\int_{\bbR}\frac{d\mu(t)}{\omega(|t|)}<\infty
		,
	\]
	and choose an entire function $f$ such that 
	\[
		\forall r\geq 0: \max_{|z|=r}|f(z)|\geq\omega(\sqrt r)
		. 
	\]
	This is possible, e.g., by \cite[Theorem~10.3]{rubel:1996}. We may assume w.l.o.g.\ that all power series coefficients
	of $f$ are nonnegative, so that 
	\[
		\forall r\geq 0: \max_{|z|=r}|f(z)|=f(r)
		.
	\]
	For $m\in\bbN$ let $\nu_m$ be the measure 
	\[
		d\nu_m(t):=\exp\big[-t^{2m+2}-f(t^2)\big]d\mu(t)
		.
	\]
	Then 
	\[
		\int_{\bbR\setminus[-1,1]} e^{t^2}d\nu_m(t)\leq
		\int_{\bbR\setminus[-1,1]} e^{-f(t^2)}d\mu(t)\leq\int_\bbR e^{-\omega(|t|)}d\mu(t)\leq
		\int_\bbR\frac{d\mu(t)}{\omega(|t|)}<\infty
		.
	\]
	By the first part of the proof there exists $\scrC_m\in\ubCh$ with $\vartheta_{\scrC_m}=0$ and $\mu_{\scrC_m}=\nu_m$.
	Let $G_m$ be the entire function 
	\[
		G_m(z):=\exp\big[-\frac 12\big(z^{2m+2}+f(z^2)\big)\big]
		.
	\]
	This function is real and zerofree. We have $G_m\cdot\scrC_m\in\ubCh$, and for each $\cH\in\scrC_m$, $F\in\cH$, it holds
	that 
	\begin{align*}
		\|G_mF\|_{G_m\cdot\cH}^2= &\, \|F\|_{\cH}^2\geq\int_\bbR |F(t)|^2d\nu_m(t)
		\\
		= &\, \int_\bbR |F(t)|^2\cdot|G_m(t)|^2d\mu(t)=\int_\bbR |(G_mF)(t)|^2d\mu(t)
		,
	\end{align*}
	where equality holds when $G_mF\in(G_m\cdot\cH)^\flat$; recall again \Cref{lem:26}. 
	We see that $\mu_{G_m\cdot\scrC_m}=\mu$. 

	Each chain $\scrC_m$ contains the space $\spann\{1\}$, and hence the chain $G_m\cdot\scrC_m$ contains 
	$\spann\{G_m\}$. For $m\neq m'$ we have 
	\[
		G_{m'}\notin\spann\{G_m\},\quad G_m\notin\spann\{G_{m'}\}
		,
	\]
	and therefore $G_m\cdot\scrC_m\neq G_{m'}\cdot\scrC_{m'}$. 

\item We come to the proof of (ii). Assume we have $\scrC_j\rightsquigarrow\scrC$. According to \Cref{thm:980} we find 
	$E_j,E\in\HB$ and $H_j,H\in\bbHe_{0,\infty}$ such that 
	\[
		\scrC_j=\Phi(E_j,H_j),\ \scrC=\Phi(E,H),\qquad
		\lim_{j\in J}E_j=E,\ \lim_{j\in J}H_j=H.
	\]
	It follows that 
	\[
		\lim_{j\in J}\Big[J\begin{psmallmatrix} A_j & B_j\\ -B_j & A_j\end{psmallmatrix}\star q_{H_j}\Big]
		=J\begin{psmallmatrix} A & B\\ -B & A\end{psmallmatrix}\star q_H
		,
	\]
	and, remembering \cref{eq:968}, therefore $\lim_{j\in J}\mu_{\scrC_j}=\mu_{\scrC}$.
\end{Elist}
\end{proof}

\begin{remark}
\label{rem:902}
	Let $\mu\in\bbM$. By \cite[Theorem~40]{debranges:1968} we have 
	\[
		\bigcup\ubCh_\mu=\big\{\cH\in\DB^*\mid \cH\sqsubseteq L^2(\mu)\big\}\cup\big\{\{0\}\big\}
		.
	\]
	By the ordering theorem \cite[Theorem~35]{debranges:1968} two elements $\scrC_1,\scrC_2$ of $\ubCh_\mu$ are either equal
	or $\scrC_1\cap\scrC_2=\{\{0\}\}$. We have $\scrC_1=\scrC_2$ if and only if there exist functions 
	$F_1\in(\bigcup\scrC_1)\setminus\{0\}$ and $F_2\in(\bigcup\scrC_2)\setminus\{0\}$, such that their quotient $F_1/F_2$ 
	is a meromorphic function of bounded characteristic in $\bbC_+$ and $\bbC_-$ (and further equivalent that this holds for
	all such $F_1,F_2$). 
\end{remark}

%
\subsection{The inverse problem}
\label{sec:5-2}
%

We do not know any natural way to construct a right-inverse of the surjective map 
\[
	\ubCh\to\bbM,\quad\scrC\to\mu_\scrC
	.
\]
For the subclass $\bbM_{<\infty}$ of power bounded measures, an inverse construction can be made. 
This is based on the fact that for such measures an element of $\ubCh_\mu$ with a particular additional function theoretic 
property can be singled out.

\begin{theorem}
\label{thm:904}
	\phantom{}
	\begin{enumerate}[{\rm(i)}]
	\item For each $\mu\in\bbM_{<\infty}$ there exists a unique element $\scrC\in\ubCh_\mu$ such that 
		\[
			\forall F\in\bigcup\scrC: F\text{ is of bounded type in $\bbC_+$ and $\bbC_-$}
		\]
		We denote this chain as $\scrC(\mu)$. 
	\item Let $(\mu_j)_{j\in J}$ be a net in $\bbM_{<\infty}$ and assume that 
		\[
			\exists \k\in\bbN_0: \sup_{j\in J}\|\mu_j\|_\k<\infty
			.
		\]
		If $\lim_{j\in J}\mu_j=\mu_\infty$ in the $w^*$-topology of $C_c(\bbR)^*$, then $\mu_\infty\in\bbM_{\leq \kappa}$ and
		$\scrC(\mu_j)\rightsquigarrow\scrC(\mu_\infty)$.
	\end{enumerate}
\end{theorem}

The proof of this theorem proceeds via a detour through the sign-indefinite world: it relies on the results recalled in 
the preliminaries, in particular on \Cref{thm:950} and \Cref{thm:906}. 

\begin{proof}[Proof of \Cref{thm:904}(i), existence]
	Let $\mu\in\bbM_{<\infty}$ be given. The case that $\mu=0$ is trivial. In fact, set 
	\[
		E(t,z):=1-itz,\quad t>0
		.
	\]
	Then 
	\[
		\cB(E(t,.))=\spann\{1\},\quad \|1\|_{\cB(E(t,.))}^2=\frac 1t
		,
	\]
	and we see that 
	\[
		\scrC:=\big\{\{0\}\big\}\cup\big\{\cB(E(t,.))\mid t>0\big\}\in\ubCh
	\]
	and $\cH\sqsubseteq L^2(\mu)$ for all $\cH\in\scrC$. 

	Assume throughout the following that $\mu$ is not the zero measure. Choose $\k\in\bbN$ such that 
	$\|\mu\|_\k<\infty$, and a polynomial $p$ with real coefficients of degree $2\k+1$ whose leading coefficient is not
	smaller than $\|\mu\|_\k$. Then the function $q:=\mathrm C_\k[\mu,p]$ belongs to the class $\cN_{<\infty}^{(\infty)}$. 
	Let $W:(a,b)\times\bbC\to\bbC^{2\times 2}$ be a matrix family for $q$ with Hamiltonian $H:(a,b)\to\bbR^{2\times 2}$
	(recall \Cref{thm:906}). Our candidate for the required unbounded chain is 
	\[
		\scrC:=\big\{\{0\}\big\}\cup\big\{\cB(E(t,.))\mid t\in(a,b)\big\}
		,
	\]
	where $E(t,.)$ is as in item (ii) of \Cref{thm:906}. Note that $E(t,.)$ is of bounded type in $\bbC_+$ and $\bbC_-$. 
	Fix $c\in(a,b)$. The map $\chi:[a,c]\to\DB\cup\{\{0\}\}$ defined as 
	\[
		\chi(t):=
		\begin{cases}
			\cB(E(t,.)) &\text{if}\ t\in(a,c],
			\\
			\{0\} &\text{if}\ t=a,
		\end{cases}
	\]
	is continuous, injective, and preserves order. Thus, by \Cref{thm:2}(ii), we have 
	\[
		\big\{\{0\}\big\}\cup\big\{\cB(E(t,.))\mid t\in(a,c]\big\}=\scrC\big(\cB(E(c,.))\big)
		.
	\]
	Let $H_c\in\bbHe_{0,\infty}$ be a reparameterization of $-JH|_{(c,b)}J$, i.e., $H_c(s)=-JH(\tau(s))J$ with an
	appropriate increasing bijection $\tau:(0,\infty)\to(c,b)$. Then 
	\[
		\big\{\cB(E(t,.))\mid t\in[c,b)\big\}=\big\{\cB(E(c,.))\ltimes W_{H_c}(s,.)\mid s\geq 0\big\}
		,
	\]
	and hence 
	\begin{equation}
	\label{eq:909}
		\scrC=\Phi\big(E(c,.),H_c\big)
		.
	\end{equation}
	We use the construction from \Cref{thm:900} to compute $\mu_\scrC$. It holds that 
	\[
		W_{H_c}(s,.)=-JW(c,.)^{-1}W(\tau(s),.)J
		,
	\]
	and hence 
	\begin{align*}
		q_{H_c}= &\, \lim_{s\to\infty}\big[W_{H_c}(s,.)\star 0\big]
		\\
		= &\, -JW(c,.)^{-1}\star\Big(\lim_{s\to\infty}\big[W(\tau(s),.)\star\infty\big]\Big)
		=-JW(c,.)^{-1}\star q
		.
	\end{align*}
	We denote $(1,0)W(c,z)=:(D(z),-C(z))$, then 
	\[
		W(c,z)=\begin{pmatrix} D(z) & -C(z) \\ -B(c,z) & A(c,z)\end{pmatrix},\quad 
		W(c,z)^{-1}=\begin{pmatrix} A(c,z) & C(z) \\ B(c,z) & D(z)\end{pmatrix}
		.
	\]
	This leads to 
	\begin{align*}
		J \mkern120mu&\,\mkern-120mu \begin{pmatrix} A(c,z) & B(c,z) \\ -B(c,z) & A(c,z)\end{pmatrix}\star q_{H_c}(z)
		\\
		= &\, -J\begin{pmatrix} A(c,z) & B(c,z) \\ -B(c,z) & A(c,z)\end{pmatrix}J
		\begin{pmatrix} A(c,z) & C(z) \\ B(c,z) & D(z)\end{pmatrix}\star q(z)
		\\
		= &\, \begin{pmatrix} A(c,z)^2+B(c,z)^2 & A(c,z)C(z)+B(c,z)D(z) \\ 0 & 1\end{pmatrix}\star q(z)
		\\
		= &\, \big(A(c,z)^2+B(c,z)^2\big)q(z)+\big(A(c,z)C(z)+B(c,z)D(z)\big)
		\\
		= &\, \big(A(c,z)^2+B(c,z)^2\big)\cdot(1+z^2)^{\kappa+1}\int_\bbR\frac 1{t-z}\frac{d\mu(t)}{(1+t^2)^{\k+1}}
		\\
		&\, +\Big[\big(A(c,z)^2+B(c,z)^2\big)p(z)+\big(A(c,z)C(z)+B(c,z)D(z)\big)\Big]
		.
	\end{align*}
	Since $A(c,.),B(c,.),C,D,p$ are all real entire functions, and 
	\[
		A(c,t)^2+B(c,t)^2=|E(c,t)|^2,\quad t\in\bbR
		,
	\]
	the Stieltjes-Livshits inversion \cref{eq:916} formula yields that for $\alpha<\beta$ with 
	$\mu_\scrC(\{\alpha\})=\mu_\scrC(\{\beta\})=0$
	\begin{align*}
		\int_\alpha^\beta|E(c,t)|^2 &\, d\mu_\scrC(t)=\lim_{\varepsilon\to 0}\frac 1\pi\int_\alpha^\beta\Im\bigg[J
		\begin{psmallmatrix} 
			A(c,t+i\varepsilon) & B(c,t+i\varepsilon) 
			\\ 
			-B(c,t+i\varepsilon) & A(c,t+i\varepsilon)
		\end{psmallmatrix}
		\star q_{H_c}(t+i\varepsilon)\bigg]dt
		\\
		= &\, \lim_{\varepsilon\to 0}\frac 1\pi\int_\alpha^\beta\Im\bigg[
		\big(A(c,t+i\varepsilon)^2+B(c,t+i\varepsilon)^2\big)\big(1+(t+i\varepsilon)^2\big)^{\k+1}
		\\
		&\, \mkern245mu\cdot\int_\bbR\frac 1{x-(t+i\varepsilon)}\frac{d\mu(x)}{(1+x^2)^{\k+1}}\bigg]dt
		\\
		= &\, \int_\alpha^\beta|E(c,t)|^2(1+t^2)^{\k+1}\cdot\frac{d\mu(t)}{(1+t^2)^{\k+1}}
		=\int_\alpha^\beta|E(c,t)|^2d\mu(t)
		.
	\end{align*}
	Since $|E(c,t)|>0$ for all $t\in\bbR$, we conclude that $\mu_\scrC=\mu$. 
\end{proof}

\begin{proof}[Proof of \Cref{thm:904}(i), uniqueness]
	By \Cref{rem:902} each set $\ubCh_\mu$ can contain at most one element $\scrC$ such that all elements of $\bigcup\scrC$
	are of bounded type in $\bbC_+$ and $\bbC_-$. 
\end{proof}

Our next aim is to prove a continuity property on the level of functions $q\in\cN_{<\infty}^{(\infty)}$. The assertion stated in 
\Cref{thm:904}(ii) will then follow easily. 

We use the following notation: 
Let $q\in\cN_{<\infty}^{(\infty)}$ and $W:(a,b)\times\bbC\to\bbC^{2\times2}$ be a matrix family for $q$ with Hamiltonian 
$H\in\bbH_{a,b}$. For $c\in(a,b)$ denote by $H_c\in\bbHe_{0,\infty}$ the trace-normalized reparameterization of $H|_{(c,b)}$.
Note here that $H|_{(c,b)}$ is in limit circle case at $c$ and limit point case at $b$. 

\begin{proposition}
\label{pro:910}
	Let $\kappa\in\bbN_0$, $(q_j)_{j\in J}$ a net in $\cN_{\leq\kappa}^{(\infty)}$, and
	$q_\infty\in\cN_{\leq\kappa}^{(\infty)}$. Then the following statements are equivalent. 
	\begin{enumerate}[{\rm(i)}]
	\item $\lim_{j\in J}q_j=q_\infty$ locally uniformly on $\bbC_+$. 
	\item There exist matrix families 
		\[
			W_j:(a_j,b_j)\times\bbC\to\bbC^{2\times 2},\quad j\in J\cup\{\infty\},
		\]
		for $q_j$ with respective Hamiltonians $H_j\in\bbH_{a_j,b_j}$, and points $c_j\in(a_j,b_j)$, such that 
		$\lim_{j\in J}W_j(c_j,.)=W_\infty(c_\infty,.)$ locally uniformly on $\bbC$ and 
		$\lim_{j\in J}H_{j,c_j}=H_{\infty,c_\infty}$ in $\bbHe_{0,\infty}$.
	\end{enumerate}
\end{proposition}
\begin{proof}
	\phantom{}
\begin{Elist}
\item
	The implication ``(ii)$\Rightarrow$(i)'' is the easy one, and we settle it first. Since $W_j(t,z)$ and $W(t,z)$ are
	solutions of the respective equations \cref{eq:9}, we have 
	\begin{align*}
		q_\infty= &\, W_\infty(c_\infty,.)\star q_{H_{\infty,c_\infty}}
		=\Big[\lim_{j\in J}W_j(c_j,.)\Big]\star\Big[\lim_{j\in J}q_{H_{j,c_j}}\Big]
		\\
		= &\, \lim_{j\in J}\Big[W_j(c_j,.)\star q_{H_{j,c_j}}\Big]=\lim_{j\in J}q_j
		.
	\end{align*}
\item
	The essence in the proof of ``(i)$\Rightarrow$(ii)'' is the following statement, which we are going to prove in this
	step. 

	\medskip\emph{
		Let $\kappa\in\bbN_0$, $\mu_j\in\bbM_{\leq\kappa}$, $\pi_j\in\bbR$ for $j\in J\cup\{\infty\}$, assume that 
		$\pi_j\geq\|\mu_j\|_\kappa$, $\lim_{j\in J}\mu_j=\mu_\infty$ in the $w^*$-topology in $C_c(\bbR)^*$, and 
		$\lim_{j\in J}\pi_j=\pi_\infty$. Then there exist matrix families 
		\[
			W_j:(a_j,b_j)\times\bbC\to\bbC^{2\times 2},\quad j\in J\cup\{\infty\}
			,
		\]
		for the functions 
		\[
			q_j:=\mathrm C_\kappa\big[\mu_j,\pi_jz(1+z^2)^\kappa\big]
			,
		\]
		respectively, and $c_j\in(a_j,b_j)$, $j\in J\cup\{\infty\}$, such that 
		\[
			\lim_{j\in J}W_j(c_j,.)=W_\infty(c_\infty,.)
			.
		\]
	}
	
	\noindent
	Note here that the assumptions on $\mu_j,\pi_j$ ensure that $\mu_\infty\in\bbM_{\leq\kappa}$ and 
	$\pi_\infty\geq\|\mu_\infty\|_\kappa$, and hence that $q_\infty$ is well-defined. 

	We use induction on $\kappa$. Consider the case that $\kappa=0$. For $j\in J\cup\{\infty\}$ let $H_j\in\bbHe_{0,\infty}$
	be the Hamiltonian with $q_j=q_{H_j}$. Then the fundamental solution $W_{H_j}(t,.)$, $t>0$, of $H_j$ is a 
	matrix family for $q_j$ with Hamiltonian $H_j$. By the Grommer-Hamburger theorem we have $\lim_{j\in J}q_j=q_\infty$,
	and hence $\lim_{j\in J}H_j=H_\infty$. In particular, $\lim_{j\in J}W_{H_j}(1,.)=W_{H_\infty}(1,.)$. 

	Assume now that the assertion holds for some $\kappa$, and let $\mu_j,\pi_j$ be given as in the assertion for
	$\kappa+1$. Then we define $\tilde\mu_j,\tilde\pi_j$ as 
	\[
		d\tilde\mu_j:=\frac{d\mu_j}{1+t^2},\quad \tilde\pi_j:=\pi_j
		,
	\]
	and this is data to which our inductive hypothesis applies. We thus obtain matrix families 
	$\tilde W_j:(\tilde a_j,\tilde b_j)\times\bbC\to\bbC^{2\times 2}$ for 
	$\tilde q_j:=\mathrm C_\kappa[\tilde\mu_j,\tilde\pi_jz(1+z^2)^\kappa]$, and $\tilde c_j\in(\tilde a_j,\tilde b_j)$, such
	that 
	\begin{equation}
	\label{eq:911}
		\lim_{j\in J}\tilde W_j(\tilde c_j,.)=\tilde W_\infty(\tilde c_\infty,.)
		.
	\end{equation}
	W.l.o.g.\ we may thereby assume that $\tilde b_j<\infty$ (this can always be achieved by a reparameterization). 

	We have 
	\begin{align*}
		q_j(z)= &\, \mathrm C_{\kappa+1}\big[\mu_j,\pi_jz(1+z^2)^{\kappa+1}\big](z)
		\\
		= &\, (1+z^2)\mathrm C_\kappa\Big[\frac{d\mu_j}{1+t^2},\pi_jz(1+z^2)^\kappa\Big](z)
		=(1+z^2)\tilde q_j(z)
		,
	\end{align*}
	and invoke \cite[Lemma~4.16]{langer.woracek:gpinf}. This provides us with a matrix family
	$W_j:(a_j,b_j)\times\bbC\to\bbC^{2\times 2}$ for $q_j$. The formulae in the proof of this lemma are explicit, and say
	that $a_j=\tilde a_j$, $b_j\geq\tilde b_j$, and 
	\[
		W_j(t,z)=\begin{pmatrix} 1 & 0 \\ 0 & \frac 1{1+z^2}\end{pmatrix}\tilde W_j(t,z)
		\begin{pmatrix}
			1+z\frac{\tilde R_j(t)}{\tilde J_j(t)} & -z\frac 1{\tilde J_j(t)}
			\\[2mm]
			z\tilde J_j(t)\big(1+\frac{\tilde R_j(t)^2}{\tilde J_j(t)^2}\big) & 1-z\frac{\tilde R_j(t)}{\tilde J_j(t)}
		\end{pmatrix}
		,\quad t\in(a_j,\tilde b_j)
		,
	\]
	where 
	\[
		\tilde R_j(t)=-\Re\frac{\tilde w_{j,21}(t,i)}{\tilde w_{j,22}(t,i)},\quad 
		\tilde J_j(t)=-\Im\frac{\tilde w_{j,21}(t,i)}{\tilde w_{j,22}(t,i)}
		.
	\]
	Note here that $\tilde J_j(t)>0$ for all $t\in(\tilde a_j,\tilde b_j)$ since 
	$\tilde w_{j,21}(t,.)+i\tilde w_{j,21}(t,.)\in\HB$. The values of $W_j(t,.)$ on the possible remainder 
	$(\tilde b_j,b_j)$ of the domain of definition of $W_j$ are also determined in \cite{langer.woracek:gpinf} but are 
	irrelevant for our purposes. 

	We see that \cref{eq:911} implies that 
	\[
		\lim_{j\in J}W_j(c_j,.)=W_\infty(c_\infty,.)
	\]
	where $c_j:=\tilde c_j$, $j\in J\cup\{\infty\}$. 
\item
	We deduce the implication ``(i)$\Rightarrow$(ii)''. Assume that $\lim_{j\in J}q_j=q_\infty$. Write 
	\[
		q_j=\mathrm C_\kappa[\mu_j,p_j],\quad \pi_j^0:=p_j(0),\ \pi_j:=\frac{p_j^{2\kappa+1}(0)}{(2\kappa+1)!}
		,
	\]
	then 
	\[
		\lim_{j\in J}\mu_j=\mu_\infty
		\quad\text{and}\quad 
		\lim_{j\in J}p_j=p_\infty,\ \lim_{j\in J}\pi_j^0=\pi_\infty^0,\ \lim_{j\in J}\pi_j=\pi_\infty
		.
	\]
	By what we showed in the previous step, there exist matrix families $W_j$ for 
	$\mathrm C_\kappa[\mu_j,\pi_jz(1+z^2)^\kappa]$ defined on certain intervals $(a_j,b_j)$, and corresponding 
	points $c_j\in(a_j,b_j)$, such that 
	\[
		\lim_{j\in J}W_j(c_j,.)=W_\infty(c_\infty,.)
		.
	\]
	We have 
	\[
		q_j(z)=\mathrm C_\kappa\big[\mu_j,\pi_jz(1+z^2)^\kappa\big]
		+\pi_j^0+\big(p_j(z)-\pi_j^0-\pi_jz(1+z^2)^\kappa\big)
		,
	\]
	and by \cite[Lemma~10.2]{kaltenbaeck.woracek:p2db} and the computation in the proof of 
	\cite[Corollary~5.9]{langer.woracek:gpinf} matrix families for $q_j$ can be obtained as 
	\[
		\tilde W_j(t,z):=
		\begin{pmatrix} 
			1 & p_j(z)-\pi_jz(1+z^2)^\kappa
			\\
			0 & 1
		\end{pmatrix}
		W_j(t,z)\begin{pmatrix} 1 & -\pi_j^0 \\ 0 & 1 \end{pmatrix},\quad t\in(a_j,b_j)
		.
	\]
	We see that 
	\[
		\lim_{j\in J}\tilde W_j(c_j,.)=\tilde W_\infty(c_\infty,.)
		.
	\]
	By \Cref{lem:913} we have 
	\[
		q_{H_{j,c_j}}=\tilde W_j(c_j,.)^{-1}\star q_j
		,
	\]
	and hence $\lim_{j\in J}q_{H_{j,c_j}}=q_{H_{\infty,c_\infty}}$. This implies that 
	$\lim_{j\in J}H_{j,c_j}=H_{\infty,c_\infty}$. 
\end{Elist}
\end{proof}

\begin{proof}[Proof of \Cref{thm:904}(ii)]
	We have $\mu_j\in\bbM_{\leq\kappa}$ and 
	\[
		\pi:=\sup_{j\in J}\|\mu_j\|_\kappa<\infty
		.
	\]
	Since $\lim_{j\in J}\mu_j=\mu_\infty$, it follows that $\mu_\infty\in\bbM_{\leq\kappa}$ and 
	$\|\mu_\infty\|_\kappa\leq\pi$. Set 
	\[
		q_j(z)=\mathrm C_\kappa\big[\mu_j,\pi z(1+z^2)^\kappa\big],\quad j\in J\cup\{\infty\}
		,
	\]
	then $q_j\in\cN_{<\infty}^{(\infty)}$ and $\lim_{j\in J}q_j=q_\infty$. According to \Cref{pro:910} we find matrix
	families $W_j:(a_j,b_j)\times\bbC\to\bbC^{2\times 2}$ for $q_j$ with Hamiltonians $H_j$ and corresponding points 
	$c_j\in(a_j,b_j)$, such that 
	\[
		\lim_{j\in J}W_j(c_j,.)=W_\infty(c_\infty,.),\quad \lim_{j\in J}H_{j,c_j}=H_{\infty,c_\infty}
		.
	\]
	By the construction of $\scrC(\mu_j)$, cf.\ \cref{eq:909}, we have 
	\[
		\scrC(\mu_j)=\Phi(E_j(c_j,.),H_{j,c_j})
	\]
	where $E_j(t,.):=w_{j,22}(t,.)+iw_{j,21}(t,.)$. Clearly, $\lim_{j\in J}E_j(c_j,.)=E_\infty(c_\infty,.)$, and 
	\Cref{thm:980} implies that $\scrC(\mu_j)\rightsquigarrow\scrC(\mu_\infty)$. 
\end{proof}

There are only few cases where the chain $\scrC(\mu)$ corresponding to a measure $\mu$ can be determined explicitly. 
One of them are measures with power density. For such measures $\scrC(\mu)$ can be described in terms of confluent hypergeometric
functions. Recall:
\[
	M(\alpha,\beta,z):=\sum_{n=0}^\infty\frac{(\alpha)_n}{(\beta)_n}\cdot\frac{z^n}{n!}
	,\quad
	\prescript{}{0}{F}_1(\beta,z):=\sum_{n=0}^\infty\frac 1{(\beta)_n}\cdot\frac{z^n}{n!}
	,
\]
where $\alpha,z\in\bbC$ and $\beta\in\bbC\setminus(-\bbN_0)$. The symbol $(\Dummy)_n$ denotes the rising factorial, i.e., 
\[
	(\alpha)_0=1,\quad (\alpha)_{n+1}=(\alpha)_n(\alpha+n)\quad\text{for }n\in\bbN_0
	.
\]
The following fact is shown in \cite[Corollary~7.6]{eichinger.woracek:homo-arXiv}. 

\begin{example}
\label{940}
	Let $\beta>0$, $(\sigma_+,\sigma_-)\in[0,\infty)^2\setminus\{(0,0)\}$, and let $\nu$ be the measure which is absolutely
	continuous w.r.t.\ the Lebesgue measure and has derivative 
	\begin{equation}\label{eqnHomMeasure}
		\frac{d\nu(\xi)}{d\xi}=
		\begin{cases}
			\sigma_+\beta\cdot\xi^{\beta-1} &\text{if}\ \xi>0,
			\\
			\sigma_-\beta\cdot|\xi|^{\beta-1} &\text{if}\ \xi < 0.
		\end{cases}
	\end{equation}
	We define functions $A,B$ by distinguishing two cases.
	\begin{enumerate}[{\rm(i)}]
	\item Assume that $\sigma_+,\sigma_->0$. Define 
		\begin{align*}
			& \alpha:=\frac i{2\pi}\log\frac{\sigma_-}{\sigma_+}+\frac{\beta-1}2
			,\quad 
			\kappa:=\frac 12
			\Big(\frac{2\Gamma(\beta+1)^2 \pi \sqrt{\sigma_+\sigma_-}}{|\Gamma(\alpha+1)|^2}\Big)^{\frac 1\beta}
			,
			\\
			& A(z):=e^{i\kappa z}\frac{M(\alpha,\beta,-2i\kappa z)+M(\alpha+1,\beta,-2i\kappa z)}2
			,
			\\
			& B(z):=z\cdot e^{i\kappa z}M(\alpha+1,\beta+1,-2i\kappa z)
			.
		\end{align*}
	\item Assume that $\sigma_+=0$ or $\sigma_-=0$. Define 
		\begin{align*}
			& \sigma:=
			\begin{cases}
				\big(\sigma_+\Gamma(\beta+1)^2\big)^{\frac 1\beta} &\text{if}\ \sigma_+>0
				,
				\\
				-\big(\sigma_-\Gamma(\beta+1)^2\big)^{\frac 1\beta} &\text{if}\ \sigma_->0
				,
			\end{cases}
			\\
			& A(z):=\prescript{}{0}{F}_1(\beta,-\sigma z)
			,\quad 
			B(z):=z\cdot\prescript{}{0}{F}_1(\beta+1,-\sigma z)
			.
		\end{align*}
	\end{enumerate}
	Now set 
	\[
		K_{\sigma_-,\sigma_+,\beta}(z,w):=\frac{B(z)A(\overline w)-A(z)B(\overline w)}{z-\overline w}
		,
	\]
	and $K_{\sigma_-,\sigma_+,\beta}(a,z,w):=a^\beta K_{\sigma_-,\sigma_+,\beta}(az,aw)$ for $a\geq 0$. Then 
	\[
		\scrC(\nu)=\big\{\cH(K_{\sigma_-,\sigma_+,\beta}(a,.,.))\mid a\geq 0\big\}
		.
	\]
	Note also that $K_{\sigma_-,\sigma_+,\beta}(0,0)=1$.
\end{example}

\begin{remark}
\label{rem:905}
	It is an open problem to characterize those measures $\mu\in\bbM$ for which there exists a chain
	$\scrC\in\ubCh_\mu$, such that all elements of $\bigcup\scrC$ are entire functions of bounded type in $\bbC_+$ and
	$\bbC_-$. We do not expect an easy answer.
\end{remark}

%
%
%
\newpage
\section{Rescaling limits for measures with regularly varying distribution function}
\label{sec:7}
%
%
%

\noindent
We apply the theory developed in the previous sections to investigate rescaling limits of reproducing kernels. 
Recall the notion of regular variation from \cite{bingham.goldie.teugels:1989}. 
We will say that $\scrh$ is locally $\asymp 1$ if for every compact $M\subset (0,\infty)$  it holds that $\inf_{t\in M}\scrh(t)>0$ and 
$\sup_{t\in M}\scrh(t)<\infty$.

\begin{definition}
\label{920}
	Let $H\in\bbH_{0,\infty}$ be in lc at $0$ and in lp at $\infty$. Recall the kernel $K_H(t,z,w)$ from \cref{eq:59}, and
	denote 
	\[
		\kappa(t):=K_H(t,0,0),\quad t\geq 0
		.
	\]
	Let $\scrh:(0,\infty)\to(0,\infty)$ be a regularly varying function with positive index, and assume that $\scrh$ is
	locally $\asymp 1$. We say that $H$ has a \emph{rescaling limit with rate $\scrh$}, if the limit 
	\begin{equation}
	\label{921}
		K(z,w):=\lim_{t\to\infty}\frac 1{\kappa(t)}K_H\Big(t,\frac z{\scrh(\kappa(t))},\frac w{\scrh(\kappa(t))}\Big)
	\end{equation}
	exists locally uniformly for $(z,w)\in\bbC\times\bbC$ and is not constant. 
\end{definition}

Note that the factor $\frac 1{\kappa(t)}$ is chosen such that $K(0,0)=1$. 

A Hamiltonian $H\in\bbH_{0,\infty}$ which is in lc at $0$ and lp at $\infty$ gives rise to a spectral measure $\mu_H$.
In the below theorem we relate existence of a rescaling limit of $H$ with the local behaviour of $\mu_H$ at zero. 

\begin{theorem}
\label{922}
	Let $H\in\bbH_{0,\infty}$ be in limit circle case at $0$ and in limit point case at $\infty$, and let $\mu_H$ be its
	spectral measure. Then the following statements are equivalent. 
	\begin{enumerate}[{\rm(i)}]
	\item There exists a regularly varying function $\scrh$ with positive index which is locally $\asymp 1$, such that $H$
		has a rescaling limit with rate $\scrh$. 
	\item There exists a regularly varying function $\scrg$ with positive index and numbers $\sigma_+,\sigma_-\geq 0$ with 
		$(\sigma_+,\sigma_-)\neq(0,0)$, such that 
		\begin{equation}
		\label{929}
			\lim_{r\to\infty}\scrg(r)\mu_H\big((0,\tfrac 1r)\big)=\sigma_+,\quad
			\lim_{r\to\infty}\scrg(r)\mu_H\big((-\tfrac 1r,0)\big)=\sigma_-,
		\end{equation}
		and $\mu_H(\{0\})=0$. 
	\end{enumerate}
	Assume that {\rm(i)} and {\rm(ii)} hold. Then the functions $\scrh$ and $\scrg$ are asymptotic inverses of each other,
	and the limit kernel in {\rm(i)} is equal to the kernel $K_{\sigma_-,\sigma_+,\beta}(z,w)$ from Example~\ref{940} built with the data from
	{\rm(ii)}, namely $\sigma_+,\sigma_-$ and $\beta$, where $\beta$ is the index of $\scrg$. 
\end{theorem}

To prove this theorem we have to relate the reproducing kernels $K_H(t,z,w)$ for large $t$ and small $z,w$ with the concentration
of mass of the spectral measure $\mu_H$ around zero. This is achieved by relating both with a third object, and this man in the
middle is a family of transforms of the Hamiltonian $H$. 

\begin{definition}
\label{924}
	Let $H\in\bbH_{0,\infty}$ be lc at $0$ and lp at $\infty$, and let $\scrg:(0,\infty)\to(0,\infty)$ be a function.
	The we define, for each $r>0$, a weighted rescaling $\cA_rH:(0,\infty)\to\bbR^{2\times 2}$ of $H$ as follows:
	write $H=\big(\begin{smallmatrix} h_1 & h_3\\ h_3 & h_2\end{smallmatrix}\big)$ and set 
	\[
		(\cA_rH)(t):=
		\begin{pmatrix}
			\frac{\scrg(r)}rh_1(rt) & h_3(rt)
			\\[2mm]
			h_3(rt) & \frac r{\scrg(r)}h_2(rt)
		\end{pmatrix}
		,\quad t>0.
	\]
\end{definition}

In the next lemma we provide the properties of this transform which we will use in the sequel. 

\begin{lemma}
\label{925}
	Let $H\in\bbH_{0,\infty}$ be lc at $0$ and lp at $\infty$, and let $\scrg:(0,\infty)\to(0,\infty)$.
	Moreover, let $r>0$. Then $\cA_rH\in\bbH_{0,\infty}$, and is lc at $0$ and lp at $\infty$. 
	We have 
	\[
		K_{\cA_rH}(t,z,w)=\frac 1{\scrg(r)}K_H\big(rt,\tfrac zr,\tfrac wr\big),\quad t\in[0,\infty),
	\]
	\[
		\mu_{\cA_rH}=\scrg(r)\Sigma^r_*\mu_H
	\]
	where $\Sigma^r:\bbR\to\bbR$ is the map $\xi\mapsto r\xi$ and $\Sigma^r_*\mu_H$ is the pushforward of the measure $\mu_H$ under $\Sigma^r$.
\end{lemma}
\begin{proof}
	It is clear that $\cA_rH\in\bbH_{0,\infty}$ for each $r>0$ and, since 
	\[
		\tr\cA_rH(t)\geq\min\Big\{\frac{\scrg(r)}r,\frac r{\scrg(r)}\Big\}\cdot\tr H(rt)
		,
	\]
	that $\cA_rH$ is lc at $0$ and lp at $\infty$. 
	Plugging in the differential equation shows that the fundamental solution of $\cA_rH$ is 
	\[
		W_{\cA_rH}(t,z)=
		\Big(\begin{smallmatrix} \sqrt{\frac{\scrg(r)}r} & 0 \\ 0 & \sqrt{\frac r{\scrg(r)}}\end{smallmatrix}\Big)
		W_H\big(rt,\tfrac zr\big)
		\Big(\begin{smallmatrix} \sqrt{\frac r{\scrg(r)}} & 0 \\ 0 & \sqrt{\frac{\scrg(r)}r}\end{smallmatrix}\Big)
		,\quad t\geq 0.
	\]
	A computation yields the asserted formula for the kernel $K_{\cA_rH}(t,z,w)$. Moreover, we see that 
	\[
		q_{\cA_rH}(z)=\tfrac{\scrg(r)}r\cdot q_H\big(\tfrac zr\big)
		,
	\]
	and the Stieltjes inversion formula implies the assertion about the spectral measure.
\end{proof}

Combining the above lemma with \Cref{923} we have the following immediate corollary.

\begin{corollary}
\label{926}
	Assume we are in the situation of \Cref{925}. Then
	\[
		\forall r>0: \scrC(\mu_{\cA_rH})=
		\Big\{\cH\Big(\tfrac 1{\scrg(r)}K_H\big(rt,\tfrac zr,\tfrac wr\big)\Big)\Big|\ t\in[0,\infty)\Big\}
		.
	\]
\end{corollary}

We can already establish one implication from \Cref{922}.

\begin{proof}[Proof of \Cref{922}, ``{\rm(i)}$\Rightarrow${\rm(ii)}'']
	Assume that $H$ has the rescaling limit $K(z,w)$ with rate $\scrh$. We proceed in four steps. 
\begin{Elist}
\item The first step is to show that 
	\[
		\lim_{t\to\infty}\kappa(t)=\infty
		.
	\]
	Assume towards a contradiction that $C:=\sup_{t\geq 0}\kappa(t)<\infty$; note here that $\kappa(t)$ is nondecreasing. 
	Since $\scrh$ is locally $\asymp 1$, and we have $t_0$ such that $\kappa(t_0)>0$ 
	\[
		c_-:=\inf_{t\geq t_0}\scrh(\kappa(t))>0,\quad c_+:=\sup_{t\geq t_0}\scrh(\kappa(t))<\infty
		.
	\]
	Choose a sequence $(t_n)_{n\in\bbN}$ with $t_n\to\infty$, such that the limit 
	$\alpha:=\lim_{n\to\infty}\scrh(\kappa(t_n))$ exists. Let $w\in\bbC_+$, then by \cref{918} we have 
	\[
		\lim_{t\to\infty}K_H(t,z,z)=\infty
	\]
	uniformly for $z\in w\cdot[\frac 1{c_+},\frac 1{c_-}]$. We obtain that 
	\begin{align*}
		\infty= &\, 
		\lim_{n\to\infty}\frac 1C K_H\Big(t_n,\frac w{\scrh(\kappa(t_n))},\frac w{\scrh(\kappa(t_n))}\Big)
		\\
		\leq &\, 
		\lim_{n\to\infty}\frac 1{\kappa(t_n)}
		K_H\Big(t_n,\frac w{\scrh(\kappa(t_n))},\frac w{\scrh(\kappa(t_n))}\Big)
		=K(\tfrac w\alpha,\tfrac w\alpha)<\infty
		,
	\end{align*}
	and have reached a contradiction.

	Since $\kappa(t)$ tends to infinity, it holds for every function $\tilde{\scrh}:(0,\infty)\to(0,\infty)$ with 
	$\tilde{\scrh}\sim\scrh$ that 
	\[
		\lim_{t\to\infty}\frac 1{\kappa(t)}
		K_H\Big(t,\frac z{\tilde{\scrh}(\kappa(t))},\frac w{\tilde{\scrh}(\kappa(t))}\Big)=K(z,w)
		.
	\]
	By the smooth variation theorem \cite[Theorem~1.8.2]{bingham.goldie.teugels:1989} we may thus switch to a rate which is
	possibly better behaved than $\scrh$. We use this freedom and assume for the rest of the proof that $\scrh$ is
	continuous and has a finite positive limit at $0$. 
		
\item In this step we pass to the man in the middle. Choose an asymptotic inverse $\scrg$ of $\scrh$, and use $\scrg$ to build
	the transforms $\cA_rH$. Moreover, denote by $\beta$ the index of $\scrg$. We consider the chains 
	$\scrC(\mu_{\cA_rH})$ and the function 
	\[
		\chi:\left\{
		\begin{array}{rcl}
			[0,\infty) & \to & \RK
			\\
			a & \mapsto & \cH\big(a^\beta K(az,aw)\big)
		\end{array}
		\right.
	\]
	Note here that $K(z,w)$ is a positive kernel as limit of positive kernels, and thus $a^{\beta}K(az,aw)$ also is a
	positive kernel. 

	Our aim is to apply \Cref{927}. It is clear that $\chi(0)=\{0\}$ and that $\chi$ is continuous. We have 
	$a^\beta K(a\cdot 0,a\cdot 0)=a^\beta$, and since $\beta>0$ this function is strictly increasing and tends to $\infty$. 
	We have to produce elements $\cH_{a,r}\in\scrC(\mu_{\cA_rH})$, $a,r>0$, such that 
	$\lim_{r\to\infty}\cH_{a,r}=\chi(a)$ for all $a>0$. Set 
	\[
		r(a,t):=\scrh\Big(\frac{\kappa(t)}{a^\beta}\Big),\quad T(a,t):=\frac t{r(a,t)},
	\]
	then 
	\begin{align*}
		K_{\cA_{r(a,t)}H}\big(T(a,t),z,w\big)= &\, 
		\frac 1{\scrg(r(a,t))}K_H\Big(t\,,\,\frac z{r(a,t)},\frac w{r(a,t)}\Big)
		\\
		\sim &\, \frac{a^\beta}{\kappa(t)}K_H\bigg(t,
		\frac 1{\scrh(\kappa(t))}\cdot 
		\underbrace{z\frac{\scrh(\kappa(t))}{\scrh\big(\frac{\kappa(t)}{a^\beta}\big)}}_{\to za}
		\,,\,
		\frac 1{\scrh(\kappa(t))}\cdot 
		\underbrace{w\frac{\scrh(\kappa(t))}{\scrh\big(\frac{\kappa(t)}{a^\beta}\big)}}_{\to wa}
		\bigg)
	\end{align*}
	and we see that 
	\[
		\lim_{t\to\infty}K_{\cA_{r(a,t)}H}\big(T(a,t),z,w\big)=a^\beta K(az,aw)
		.
	\]
	It remains to note that $t\mapsto r(a,t)$ is continuous and $\lim_{t\to\infty}r(a,t)=\infty$. 

	Now \Cref{927} implies that 
	\[
		\scrC:=\big\{\cH(a^\beta K(az,aw))\mid a\geq 0\big\}\in\ubCh,\quad 
		\scrC(\mu_{\cA_rH})\rightsquigarrow\scrC
		.
	\]
	We show that $\vartheta_{\scrC}=0$. Let $\xi\in\bbR$, and assume towards a contradiction that 
	$K(\xi,\xi)=0$. Since $\chi(a)\subseteq_c\chi(1)$ for all $a\in(0,1)$, it follows that 
	\[
		\forall a\in(0,1]: a^\beta K(a\xi,a\xi)=K_{\chi(a)}(\xi,\xi)=0
		.
	\]
	Passing to the limit $a\downarrow 0$ yields that $K(0,0)=0$, and this is a contradiction.
\item Let $\sigma$ be the spectral measure of $\scrC$. In this step we show that $\sigma$ is absolutely continuous w.r.t.\ the 
	Lebesgue measure and that its derivative has the form 
	\begin{equation}
	\label{928}
		\frac{d\sigma(\xi)}{d\xi}=
		\begin{cases}
			\sigma_+\beta\cdot\xi^{\beta-1} &\text{if}\ \xi>0,
			\\
			\sigma_-\beta\cdot|\xi|^{\beta-1} &\text{if}\ \xi<0,
		\end{cases}
	\end{equation}
	with some $\sigma_+,\sigma_-\geq 0$ and $(\sigma_+,\sigma_-)\neq(0,0)$. 
	The argument relies on the theory of homogeneous de~Branges spaces developed in 
	\cite{deBranges62,eichinger.woracek:homo-arXiv}. 

	To start with, we show that $\dim\chi(1)>1$. Assume the contrary, then 
	$\chi(a)=\chi(1)=\spann\{K(.,0)\}$ for all $a\in(0,1]$. In particular, the function $K(\frac z2,0)$ is a scalar multiple
	of $K(z,0)$, say $K(\frac z2,0)=\alpha\cdot K(z,0)$. Writing the power series expansion of $K(z,0)$ as 
	\[
		K(z,0)=1+\sum_{n=1}^\infty \gamma_nz^n
		,
	\]
	and comparing coefficients of $K(z,0)$ and $K(\frac z2,0)$ yields that $\alpha=1$ and $\gamma_n=0$ for all $n$. 
	Thus $K(z,0)$ is constant equal to $1$, and hence all
	elements of the space $\chi(1)$ are constant. We obtain that for all $z,w\in\bbC$ 
	\[
		K(z,w)=K(0,w)=\overline{K(w,0)}=1
		,
	\]
	i.e., $K(z,w)$ is constant. This is a contradiction. 

	Since $\dim\chi(1)/\chi(1)^\flat\leq 1$, we have $\chi(1)^\flat\neq\{0\}$, and hence there exists $a\in(0,1]$ such that 
	$\chi(1)^\flat=\chi(a)$. We obtain 
	\[
		\forall b\geq 1: \chi(a)\subseteq_i\chi(b)
		.
	\]
	By \cite[Lemma~2.3]{eichinger.woracek:homo-arXiv}, for each $b\geq 1$, the map 
	\[
		F(z)\mapsto \Big(\frac 1b\Big)^\frac{\beta}2 F\Big(\frac zb\Big)
	\]
	restricts to an isometric isomorphism of $\chi(b)$ onto $\chi(1)$ and to one of $\chi(a)$ onto $\chi(\frac ab)$. 
	It follows that 
	\[
		\forall b\geq 1: \chi(\tfrac ab)\subseteq_i\chi(1)
		.
	\]
	Now \cite[Proposition~5.4]{eichinger.woracek:homo-arXiv} implies that $\chi(1)$ 
	is homogeneous of order $\frac\beta2-1$, and by 
	\cite[Theorem~7.2]{eichinger.woracek:homo-arXiv} (together with \cite[Theorem~6.2]{eichinger.woracek:homo-arXiv}) the 
	measure $\sigma$ is of the form \cref{928}. 
\item It is easy to pass on to the measure $\mu_H$. By \Cref{pro:907}{\rm(ii)}, we have 
	$\lim_{r\to\infty}\mu_{\cA_rH}=\sigma$ in the $w^*$-topology of $C_c(\bbR)^*$. By \Cref{lem:950b} this means that 
	\[
		\forall a,b\in\bbR,a<b: \lim_{r\to\infty}\mu_{\cA_rH}((a,b))=\sigma((a,b))
		.
	\]
	Note here that $\sigma$ has no point masses. By \Cref{925} we have 
	\[
		\mu_{\cA_rH}((a,b))=\scrg(r)\mu_H\big(\big(\tfrac ar,\tfrac br\big)\big)
		,
	\]
	and it follows that 
	\[
		\lim_{r\to\infty}\scrg(r)\mu_H\big(\big(0,\tfrac 1r\big)\big)=\sigma\big((0,1)\big),\quad 
		\lim_{r\to\infty}\scrg(r)\mu_H\big(\big(-\tfrac 1r,0\big)\big)=\sigma\big((-1,0)\big).
	\]
	Moreover, 
	\[
		\limsup_{r\to\infty}\scrg(r)\mu_H(\{0\})\leq
		\lim_{r\to\infty}\scrg(r)\mu_H\big(\big(-\tfrac 1r,\tfrac 1r\big)\big)=\sigma\big((-1,1)\big)<\infty
		,
	\]
	and since $\scrg(r)\to\infty$ this implies that $\mu_H(\{0\})=0$. 
\end{Elist}
\end{proof}

The proof of the converse implication ``{\rm(ii)}$\Rightarrow${\rm(i)}'' in \Cref{922} works in essence by reversing the steps
in the above argument. 

The first step is to exploit the conditions on $\mu_H$ stated in {\rm(ii)}. We do this by means of 
the following lemma (recall here the notation \cref{908}). 

\begin{lemma}
\label{946}
	Let $\kappa_0\in\bbN_0$, $\mu\in\bbM_{\leq\k_0}$ with $\mu(\{0\})=0$, let $\scrg$ be regularly varying with positive index 
	$\beta$, and assume the limits \cref{929} exist with $(\sigma_+,\sigma_-)\neq(0,0)$. Denote by $\sigma$ the measure 
	which is absolutely continuous w.r.t.\ the Lebesgue measure and has derivative \cref{928}. Moreover, denote 
	$\mu_r:=\scrg(r)\Sigma^r_*\mu$ for $r>0$. Then 
	\begin{align}
	\label{945}
		& \forall a,b\in\bbR,a<b: \lim_{r\to\infty}\mu_r\big((a,b)\big)=\sigma\big((a,b)\big),
		\\
	\label{944}
	& \forall \kappa\in\bbN_0,\kappa+1>\max\big\{\frac\beta2,\kappa_0\big\}: \lim_{r\to\infty}\|\mu_r\|_\kappa=\|\sigma\|_\kappa.
	\end{align}
\end{lemma}
\begin{proof}
	The relation \cref{945} is easy to see. Let $b>0$ and compute 
	\begin{align*}
		\lim_{r\to\infty}\mu_r\big((0,b)\big)= &\, 
		\lim_{r\to\infty}\scrg(r)\mu\big(\big(0,\tfrac br\big)\big)
		\\
		= &\, \lim_{r\to\infty}
		\bigg[\scrg\big(\tfrac rb\big)\mu\big(\big(0,\tfrac br\big)\big)\cdot\frac{\scrg(r)}{\scrg(\tfrac rb)}\bigg]
		=\sigma_+b^\beta=\sigma\big((0,b)\big)
		.
	\end{align*}
	In the same way we obtain 
	\[
		\lim_{r\to\infty}\mu_r\big((-b,0)\big)=\sigma\big((-b,0)\big)
		.
	\]
	Since $\mu_r(\{0\})=0$, the relation \cref{945} follows. 

	The proof of \cref{944} is more involved; it relies on Karamata's theorems about asymptotics of integrals and Stieltjes
	transforms of regularly varying functions. First, we rewrite the norms $\|\mu_r\|_\kappa$ to a more convenient form. To
	this end let $\Theta:\bbR\setminus\{0\}\to(0,\infty)$ be the function $\Theta(\xi):=\frac 1{\xi^2}$, and let $\nu$ be the
	pushforward $\nu:=\Theta_*\mu$. Then
	\begin{align*}
		\|\mu_r\|_\kappa =&\, \int_{\bbR}\frac{d\mu_r(\xi)}{(1+\xi^2)^{\kappa+1}}
		=\scrg(r)\int_{\bbR\setminus\{0\}}\frac{d\mu(\xi)}{(1+r^2\xi^2)^{\kappa+1}}
		\\
		=&\, \scrg(r)\int_{(0,\infty)}\frac{d\nu(\xi)}{(1+\frac{r^2}\xi)^{\kappa+1}}
		=\scrg(r)\int_{(0,\infty)}\frac{\xi^{\kappa+1}d\nu(\xi)}{(\xi+r^2)^{\kappa+1}}
		.
	\end{align*}
	In order to understand the behaviour of $\|\mu_r\|_\kappa$, we thus have to analyze the measure $\xi^{\kappa+1}d\nu(\xi)$.
	Note first that 
	\[
		\int_{(0,1)}\xi^{\kappa+1}d\nu(\xi)=\int_{(1,\infty)}\frac{d\mu(\xi)}{\xi^{2(\kappa+1)}}<\infty,\quad
		\nu\big((1,\infty)\big)=\mu\big((0,1)\big)<\infty.
	\]
	Now consider the function $V:[1,\infty)\to\bbR$ defined as $V(t):=\nu((t,\infty))$. This function is nonincreasing and
	nonnegative, in particular of bounded variation. We have 
	\[
		V(t)=\mu\Big(\Big(-\frac 1{\sqrt t},\frac 1{\sqrt t}\Big)\Big)\sim\frac{\sigma_++\sigma_-}{\scrg(\sqrt t)},
	\]
	and see that $V$ is regularly varying with index $-\frac\beta2$. We apply 
	\cite[Theorem~1.6.4]{bingham.goldie.teugels:1989} to obtain (this integral is understood in Riemann-Stieltjes sense)
	\[
		\int_1^t s^{\kappa+1}dV(s)\sim\frac{-\frac\beta2}{\kappa+1-\frac\beta2}\cdot t^{\kappa+1}V(t)\sim
		\frac{-\frac\beta2(\sigma_++\sigma_-)}{\kappa+1-\frac\beta2}\cdot \frac{t^{\kappa+1}}{\scrg(\sqrt t)}.
	\]
	Consider the function $U:[0,\infty)\to\bbR$ defined as 
	\[
		U(t):=\int_{(0,t)}\xi^{\kappa+1}d\nu(\xi).
	\]
	Then 
	\[
		U(t)=\int_{(0,1)}\xi^{\kappa+1}d\nu(\xi)-\int_1^t s^{\kappa+1}dV(s)\sim
		\frac{\frac\beta2(\sigma_++\sigma_-)}{\kappa+1-\frac\beta2}\cdot\frac{t^{\kappa+1}}{\scrg(\sqrt t)}.
	\]
	We apply \cite[Theorem~1.7.4]{bingham.goldie.teugels:1989} to obtain
	\begin{align*}
		\int_{(0,\infty)}\frac{\xi^{\kappa+1}d\nu(\xi)}{(\xi+r^2)^{\kappa+1}}\sim &\, 
		\frac{\Gamma(\frac\beta2)\Gamma(\kappa+2-\sigma)}{\Gamma(\kappa+1)}\cdot\frac 1{(r^2)^{\kappa+1}}U(r^2)
		\\
		\sim &\, \frac{\Gamma(\frac\beta2)\Gamma(\kappa+2-\sigma)}{\Gamma(\kappa+1)}
		\frac{\frac\beta2(\sigma_++\sigma_-)}{\kappa+1-\frac\beta2}\cdot\frac 1{\scrg(r)},
	\end{align*}
	and thus (here $B$ is Euler's Beta-function)
	\[
		\lim_{r\to\infty}\|\mu_r\|_\kappa=
		B\big(\tfrac\beta2,\kappa+1-\tfrac\beta2\big)\cdot\frac\beta2 (\sigma_++\sigma_-)
		.
	\]
	Making a change of variable, we evaluate 
	\[
		\int_0^\infty \frac{\xi^{\beta-1}d\xi}{(1+\xi^2)^{\kappa+1}}
		=\frac 12B\big(\tfrac\beta2,\kappa+1-\tfrac\beta2\big),
	\]
	and this establishes \cref{944}. 
\end{proof}

\begin{proof}[Proof of \Cref{922}, ``{\rm(ii)}$\Rightarrow${\rm(i)}'']
	Assume that $\mu_H(\{0\})=0$ and that the limits \cref{929} exist where $\scrg$ is regularly varying with positive 
	index.
	We proceed in three steps. 
\begin{Elist}
\item The above lemma applied with $\scrg$ and $\mu_H$ and \Cref{lem:950b} justify an application of \Cref{thm:904}{\rm(ii)}, from which we obtain 
	that $\scrC(\mu_r)\rightsquigarrow\scrC(\nu)$, with the measure $\nu$ known from Example~\ref{940}; we
	use the notation $K_{\sigma_-,\sigma_+,\beta}(a,z,w)$ from this example. Moreover, recall that the chains $\scrC(\mu_{\cA_rH})$ are known
	from \Cref{926}. 

	The definition of convergence of chains yields that there exists $T:(0,\infty)\times[1,\infty)\to(0,\infty)$ such that
	\begin{equation}
	\label{943}
		\forall a>0: \lim_{r\to\infty}\frac 1{\scrg(r)}K_H\big(T(a,r)r,\tfrac zr,\tfrac wr\big)=K_{\sigma_-,\sigma_+,\beta}(a,z,w)
		.
	\end{equation}
\item In this step we show that 
	\[
		\lim_{t\to\infty}\kappa(t)=\infty
		.
	\]
	Assume towards a contradiction that $C:=\sup_{t\geq 0}\kappa(t)<\infty$. Since $\scrg(r)\to\infty$, we obtain 
	\[
		1=K_{\sigma_-,\sigma_+,\beta}(1,0,0)=\lim_{r\to\infty}\frac 1{\scrg(r)}K_H(T(1,r)r,0,0)=0
		,
	\]
	a contradiction.
\item Let $\scrh$ be an asymptotic inverse of $\scrg$. Our aim is to show that 
	\begin{equation}
	\label{942}
		\lim_{t\to\infty}\frac 1{\kappa(t)}K_H\Big(t,\frac z{\scrh(\kappa(t))},\frac w{\scrh(\kappa(t))}\Big)
		=K_{\sigma_-,\sigma_+,\beta}(1,z,w)
		.
	\end{equation}
	Set $r(t):=\scrh(\kappa(t))$. Then $\lim_{t\to\infty}r(t)=\infty$, and hence $\scrg(r(t))\sim\kappa(t)$ and 
	\[
		\forall a>0: 
		\lim_{t\to\infty}\frac 1{\kappa(t)}K_H\Big(T(a,r(t))r(t),\tfrac z{r(t)},\tfrac w{r(t)}\Big)=K_{\sigma_-,\sigma_+,\beta}(a,z,w)
		.
	\]
	The function $\kappa(t)$ is nondecreasing and the function $a\mapsto K_{\sigma_-,\sigma_+,\beta}(a,0,0)$ is strictly increasing. Hence, the
	above limit relation implies that 
	\begin{align*}
		& \forall a_+>1: \liminf_{t\to\infty}T(a_+,r(t))r(t)>t,
		\\
		& \forall a_-\in(0,1): \limsup_{t\to\infty}T(a_-,r(t))r(t)<t.
	\end{align*}
	From the first relation it follows that (here ``$\leq$'' refers to the order of positive kernels, cf.\ \cref{eq:47})
	\[
		\frac 1{\kappa(t)}K_H\big(t,\frac z{r(t)},\frac w{r(t)}\big)
		\leq\frac 1{\kappa(t)}K_H\big(T(2,r(t))r(t),\tfrac z{r(t)},\tfrac w{r(t)}\big)
	\]
	for all sufficiently large $t$, and this shows that 
	\[
		\Big\{\frac 1{\kappa(t)}K_H\Big(t,\frac z{r(t)},\frac w{r(t)}\Big)\Big|\, t\geq 1\Big\}
	\]
	is a normal family. In order to show \cref{942} it is thus enough to evaluate the limit of convergent subsequences.

	Assume that $(t_n)_{n\in\bbN}$ is a sequence with $t_n\to\infty$ such that the limit 
	\[
		K(z,w):=\lim_{n\to\infty}\frac 1{\kappa(t_n)}K_H\Big(t_n,\frac z{r(t_n)},\frac w{r(t_n)}\Big)
	\]
	exists. Let $a_+>1$, $a_-\in(0,1)$, then for all sufficiently large $n$ we have 
	\begin{multline*}
		\frac 1{\kappa(t_n)}K_H\Big(T(a_-,r(t_n))r(t_n),\frac z{r(t_n)},\frac w{r(t_n)}\Big)
		\leq\frac 1{\kappa(t_n)}K_H\Big(t_n,\frac z{r(t_n)},\frac w{r(t_n)}\Big)
		\\ 
		\leq \frac 1{\kappa(t_n)}K_H\Big(T(a_+,r(t_n))r(t_n),\frac z{r(t_n)},\frac w{r(t_n)}\Big)
	\end{multline*}
	Passing to the limit $n\to\infty$ yields 
	\[
		K_{\sigma_-,\sigma_+,\beta}(a_-,z,w)\leq K(z,w)\leq K_{\sigma_-,\sigma_+,\beta}(a_+,z,w)
		,
	\]
	and letting $a_-\uparrow 1$ and $a_+\downarrow 1$ in this relation gives $K(z,w)=K_{\sigma_-,\sigma_+,\beta}(1,z,w)$.  \qedhere
\end{Elist}
\end{proof}

\begin{proof}[Proof of Theorem~\ref{theorem:main}]
Since $W(t,\xi) \in \SL(2,\bbR)$,  the family $W(t,\xi + z)W(t,\xi)^{-1}$  is also a $J$-decreasing family and it corresponds to shifted kernels $K_1(t,z,w) = K(t,\xi + z, \xi + w)$. Thus, we can assume without loss of generality that $\xi = 0$ and that $W(t,\xi) = I$ for all $t$.

Since $W$ is $J$-decreasing, the corresponding family of de Branges spaces $\cH(t) = \cB(w_{22}(t, \cdot) + i w_{21}(t,\cdot))$ is contained in a chain $\scrC$, and the map $\cH : [a,b) \to \scrC$ is monotone increasing. Denote by $H$ the corresponding trace-normalized Hamiltonian so that
\[
W(t,z) = W_H( \gamma(t),0, z).
\]
If $W(t_1,\cdot) = W(t_2,\cdot)$ for some $t_1 < t_2$, then $\gamma(t_1) = \gamma(t_2)$. Since $W$ is continuous, so is $\gamma$. Since $W$ is limit point at $\infty$, $H \in \bbH_{0,\infty}$ and $\gamma(t) \to \infty$ as $t\to \infty$. Thus, the rescaling limit 
\[
 \lim_{t\to b}  \frac 1{K_H(\gamma(t),0,0)} K_H\left( \gamma(t), \frac{z}{\scrh(K_H(\gamma(t),0,0))},  \frac{w}{\scrh(K_H(\gamma(t),0,0))} \right)
\]
exists if and only if the rescaling limit \eqref{921} exists, and in this case their values are equal. Now the result follows directly from Theorem~\ref{922}.
\end{proof}

\newpage
\section{Two conventions and Schr\"odinger operators} \label{section:TwoConventions}

\subsection{Two conventions}
In this text, we used the convention prevalent in canonical systems \cite{debranges:1968}, that for a $J$-decreasing family of transfer matrices $W(x,z)$ in the limit point case, its Weyl function is described by
\[
q(z)=\lim\limits_{x\to\infty}W(x,z)\star \tau
\]
(independent of $\tau \in \bbC_+$). One way to obtain such transfer matrices is as solutions of
\begin{equation}\label{eqn:CanonicalConvention1}
\partial_xW(x,z)J=zW(x,z)H(x),
\end{equation}
with $H$ a Hamiltonian. Another convention, more common in mathematical physics, is to work with a $J$-increasing family of transfer matrices $T(x,z)$ and associate with it a Weyl function by
\[
m(z)=\lim\limits_{x\to\infty}T(x,z)^{-1}\star \tau.
\]
To switch between the two conventions while preserving the Weyl function and the measure, we will use
\begin{equation}\label{eqn:TwoConventions}
W(x,z) = T(x,z)^{-1}.
\end{equation}

\begin{remark}
One way to obtain a $J$-increasing family is as a solution of
\begin{equation}\label{eqn:CanonicalConvention2}
J\partial_xT(x,z) = -zH(x) T(x,z)
\end{equation}
with a Hamiltonian $H$, but we warn the reader that this is not compatible with \eqref{eqn:CanonicalConvention1}, \eqref{eqn:TwoConventions}. Instead, if $T$ is defined by \eqref{eqn:CanonicalConvention1}, \eqref{eqn:TwoConventions}, we have $T^{-1} = JT^\top J^{-1}$ since $\det T = 1$; therefore $T$ satisfies
\[
J\partial_xT(x,z) = -zJH(x)J^\top T(x,z).
\]
Thus, to switch from one convention to the other while preserving the Weyl function, the Hamiltonian $H$ should be replaced by $JHJ^\top$. This also explains differences between Hamiltonians written in this paper and those in \cite{EichLukSimanek}.
\end{remark}

\subsection{Schr\"odinger operators}
For Schr\"odinger operators $-d^2 /dx^2 + V(x)$ with a potential $V \in H^{-1}([a,b))$, we use the decomposition $V = s' + q$,  where $s \in L^2_\loc([a,b))$, $q \in L^1_\loc([a,b))$ \cite{HrynivMykytyuk2001,HrynivMykytyuk2012,LukicSukhtaievWang}. The special case $s=0$ recovers the setting $V \in L^1_\loc([a,b))$. Note that we impose the local integrability assumptions also at the endpoint $a$, i.e., it is a regular endpoint. The corresponding transfer matrices are solutions of the initial value problem
\[
\partial_x T(x,z) = R_\beta \begin{pmatrix} - s(x) & q(x) - s(x)^2 - z \\ 1 & s(x) \end{pmatrix} R_\beta^* T(x,z), \qquad T(a,z) = I,
\]
where
\[
R_\beta = \begin{pmatrix}
\cos \beta & - \sin \beta \\
\sin \beta & \cos \beta
\end{pmatrix}.
\]
We set $W(x,z) = T(x,z)^{-1}$. It is standard that this gives a $J$-decreasing family of entire matrices and that the definition of limit-point in this work is consistent with the definition in the Sturm--Liouville setting; moreover, the choice of kernels \eqref{eqn:SchrKernel} corresponds to $W$ in a manner consistent with the notations of this paper; these facts can be found in \cite{LukicSukhtaievWang,EichLukSimanek}.

\begin{proof}[Proof of Theorem~\ref{theorem:schr}]
This is a direct consequence of Theorem~\ref{theorem:main} applied to the family $W(x,z) = T(x,z)^{-1}$.
\end{proof}

\newpage
\section{Orthogonal polynomials and subexponential growth} \label{section:OP}

In this section, we explain the specializations of our work to orthogonal polynomials on the real line and on the unit circle. We recall how the study of these systems is reduced to the canonical system setting, and how the sequence of CD kernels associated to orthogonal polynomials is embedded in a continuous family of kernels. Beyond this, we study the distinction between the scaling limit of the sequence of kernels and the scaling limit of the continuous family, and the role of the subexponential growth of orthogonal polynomials. 

\subsection{Orthogonal polynomials on the real line}
Let $\mu$ be a measure on $\bbR$ such that $\supp \mu$ has infinite cardinality, and $\mu$ has finite moments corresponding to a determinate moment problem. Since shifting $\mu$ by $\xi$ merely shifts the CD kernels by $\xi$, there is no loss of generality in discussing the canonical system correspondence in its usual notation, normalized at $\xi=0$.  The Weyl function
\begin{equation}\label{eqn:Cauchy0}
m(z) = \int \frac 1{\lambda - z}\,d\mu(\lambda)
\end{equation}
of the measure $\mu$ corresponds to the canonical system with Hamiltonian
\[
H(L) = \begin{pmatrix}
q_n(0)^2 & - p_n(0) q_n(0) \\
- p_n(0) q_n(0) & p_n(0)^2
\end{pmatrix}, \qquad n \le L < n+1.
\]
Since $(jH)^2 = 0$ and $H$ is constant on $[n,n+1)$, the family of kernels corresponding to this canonical system is known to be piecewise linear \eqref{eqnPiecewiseLinear1},  and $[n,n+1]$ are indivisible intervals. Thus, most of Theorem~\ref{theorem3} will be an immediate consequence of Theorem~\ref{theorem:main}, and it remains to explain how the scaling limit of the continuous family of kernels is related to the scaling limit of the sequence of CD kernels. We start with a preliminary lemma:

\begin{lemma}\label{lemmaRegVar0615}
	Let $\scrh$ be regularly varying at $\infty$ with index $\rho > 0$ and assume that $a_n,b_n$ are sequences tending to $\infty$ with $\lim\frac{a_n}{b_n}=c\in (0,\infty)$. Then 
	\[
	\lim_{n\to\infty}\frac{\scrh(a_n^\kappa b_n^{1-\kappa})}{\scrh(b_n)}=c^{\rho \kappa}
	\]
	uniformly in $\kappa \in [0,1]$.
\end{lemma}
\begin{proof}
	Define $c_n=a_n/b_n$ and note that $c_n\to c$. For $n$ sufficiently large, $c_n\in [\alpha,\beta]$ for some compact $[\alpha,\beta] \subset (0,\infty)$. Fix $\e>0$. Then by the uniform convergence theorem \cite[Theorem 1.5.2]{bingham.goldie.teugels:1989}, there exists $n_0$ so that for $n\geq n_0$
	\[
	\left|\frac{\scrh(b_n c_n^\kappa)}{\scrh(b_n)}-c^{\rho\kappa}\right|\leq \left|\frac{\scrh(b_nc_n^\kappa)}{\scrh(b_n)}-c_n^{\rho\kappa}\right|+|c_n^{\rho\kappa}-c^{\rho\kappa}|<2\e.  \qedhere
	\]	
\end{proof}

\begin{lemma}\label{lemmaOPRLsubexponential}
Let $K(n,z,w)$ be the CD kernels corresponding to a sequence of orthogonal polynomials on the real line and $K(t,z,w)$ the corresponding linear interpolation \eqref{eqnPiecewiseLinear1}. Assume that $\lim_{n\to\infty} K(n,\xi,\xi) = \infty$, \eqref{eqn:Kcontlimit1} holds, 
and $\scrh$ is regularly varying of index $\rho > 0$ and $K_\infty \not\equiv 1$. Then \eqref{eqnNevaiConditionRelated} holds.
\end{lemma}

\begin{proof}
Without loss of generality we assume $\xi = 0$. The proof is by contradiction. If \eqref{eqnNevaiConditionRelated} fails, then $\liminf_{n\to\infty} \frac{K(n,0,0)}{K(n+1,0,0)} < 1$. By continuity, there exist $c \in (0,1)$, a subsequence $n_l \in \bbN$ with $n_l \to \infty$ as $l\to\infty$, and a sequence $\delta_l \in (0,1)$, such that for all $l$,
\[
\frac{K(n_l, 0,0)}{K(n_l + \delta_l,0,0)} = c.
\]
By linear interpolation, for $0\leq s\leq 1$,
\begin{equation}\label{abc}
K(n_l+s\delta_l,z,w)=sK(n_l+\delta_l,z,w)+(1-s)K(n_l,z,w).
\end{equation}
Let $0<s < 1$ and denote $x = ( s+c(1-s))^\rho$. For the sequence $t_l=n_l+s\delta_l$, we get
\[
\frac{K(n_l + \delta_l, 0,0)}{K(t_l, 0,0)}  = \frac 1{s+c(1-s)} = x^{-1/\rho},
\]
so since $\scrh$ is regularly varying with index $\rho$, the constants
\[
\alpha_l = \frac{ \scrh( K(n_l + \delta_l, 0,0))}{\scrh(K(t_l, 0,0))} 
\]
satisfy $\alpha_l \to x^{-1}$, $l\to\infty$. Similarly, 
\[
\frac{K(n_l, 0,0)}{K(t_l, 0,0)}  = \frac c{s+c(1-s)} = cx^{-1/\rho}
\]
implies that the constants
\[
\beta_l = \frac{\scrh( K(n_l, 0,0))}{ \scrh( K(t_l, 0,0))} 
\]
satisfy $\beta_l \to c^\rho x^{-1}$, $l \to \infty$. We write the linear interpolation \eqref{abc} as
\begin{align*}
&\frac{K(t_l,\frac{z}{\scrh(K(t_l,0,0))},\frac{w}{\scrh(K(t_l,0,0))})}{K(t_l,0,0)} \\
& \quad =\frac{K\left(n_l + \delta_l,\frac{ \alpha_l  z}{\scrh(K(n_l+\delta_l,0,0)) } ,\frac{ \alpha_l  w}{\scrh(K(n_l+\delta_l,0,0) )} \right)}{ K(n_l+\delta_l,0,0) } x^{-1/\rho} s \\
& \qquad +
\frac{K\left(n_l,\frac{ \beta_l z}{\scrh( K(n_l,0,0)  )} ,\frac{\beta_l w}{\scrh(  K(n_l,0,0)  } \right)}{ K(n_l,0,0) } c x^{-1/\rho}(1-s).
\end{align*}
By Lemma~\ref{lemmaRegVar0615} and the assumption \eqref{eqn:Kcontlimit1},  taking $l \to \infty$ gives
\[
K_\infty(z,w)=K_\infty\left(z x^{-1},w x^{-1}\right) x^{-1/\rho} s+K_\infty\left(z c^\rho x^{-1} , w c^\rho x^{-1} \right)x^{-1/\rho} c(1-s).
\]
To see the consequences of such a relation, we first rescale $(z,w)$ by a factor of $x$ to rewrite as
\[
K_\infty(x z,x w)=K_\infty\left(z,w\right)x^{-1/\rho} s+K_\infty\left(c^\rho z,c^\rho w\right) x^{-1/\rho} c(1-s).
\]
Expressing $s$ in terms of $x \in (c^\rho, 1)$, we rewrite this relation as
\[
K_\infty(xz, xw) = \frac{1 - c x^{-1/\rho}}{1-c} K_\infty(z,w) +  \frac{c x^{-1/\rho}-c}{1-c} K_\infty(c^\rho z,c^\rho w).
\]
Viewing this as a function of $x$ with fixed $z,w \in\bbC$, the function $K_\infty(xz, xw)$ is of the form $A + B x^{-1/\rho}$ on the part $\{(xz,xw)\mid x\in (c^\rho,1)\}$ of the ray $\{(rz,rw)\mid r>0\}$. Replacing $(z,w)$ by $(x_0z,x_0w)$ with arbitrary $x_0>0$ we cover the whole ray. On overlapping intervals the constants must match. Since $K_\infty(0,0)=1$ and $\rho>0$, by taking $x \to 0$ we see that $K_\infty(xz,xw)\equiv 1$ on any ray, so $K_\infty \equiv 1$, which is a contradiction.
\end{proof}

\begin{proof}[Proof of Theorem~\ref{theorem3}]
(i)$\iff$(ii)$\iff$(iii) follows immediately from Theorem~\ref{theorem:main}.

(i),(iii)$\implies$(iv): it follows from (i) that $\mu(\{\xi\}) = 0$ so $K(n,\xi,\xi) \to \infty$ as $n\to\infty$, see \cite[Lemma 5.2]{EichLukSimanek}. Now (iv) follows from (iii) and Lemma~\ref{lemmaOPRLsubexponential}.

(iv)$\implies$(iii): The assumption \eqref{eqnNevaiConditionRelated} implies by Lemma~\ref{lemmaRegVar0615} that
\begin{equation}\label{eqn:15aug1}
\lim_{n\to\infty} \frac{ \scrh(K(n+s,\xi,\xi)) }{  \scrh( K(n,\xi,\xi)) } = 1,
\end{equation}
uniformly in $s\in [0,1]$.

The condition \eqref{eqnTheorem3n} with index shifted by $1$ reads
\begin{equation}\label{eqn14aug1}
\lim_{n\to\infty}\frac{K\left(n+1, \xi+\frac{z}{\scrh(K(n+1,\xi,\xi))},\xi+\frac{w}{\scrh(K(n+1,\xi,\xi))}\right)}{K(n+1,\xi,\xi)}=K_\infty(z,w).
\end{equation}
Since convergence in \eqref{eqn14aug1} is uniform on compacts, we can combine it inside the parentheses with $\scrh(K(n+1,\xi,\xi)) / \scrh( K(n,\xi,\xi)) \to 1$ and finally multiply by $K(n+1,\xi,\xi) / K(n,\xi,\xi) \to 1$ to conclude
\begin{equation}\label{eqnTheorem3nshifted}
\lim_{n \to \infty} \frac{ K\left(n+1, \xi + \frac {z}{\scrh( K(n,\xi,\xi))} ,  \xi + \frac {w}{\scrh( K(n,\xi,\xi))} \right) }{K(n,\xi,\xi)} = K_\infty(z,w).
\end{equation}
For $t=n+s$ with $0 \le s \le 1$, using \eqref{eqnPiecewiseLinear1} and computing a convex combination of limits \eqref{eqnTheorem3n} and \eqref{eqnTheorem3nshifted}, it follows that
\[
\lim_{n \to \infty} \frac{ K\left(n+s, \xi + \frac {z}{\scrh( K(n,\xi,\xi))} ,  \xi + \frac {w}{\scrh( K(n,\xi,\xi))} \right) }{K(n,\xi,\xi)} = K_\infty(z,w)
\]
uniformly in $s\in [0,1]$. Combining this with \eqref{eqn:15aug1} inside the parentheses and multiplying by $K(n+s,\xi,\xi) / K(n,\xi,\xi) \to 1$ shows
\[
\lim_{n \to \infty} \frac{ K\left(n+s, \xi + \frac {z}{\scrh( K(n+s,\xi,\xi))} ,  \xi + \frac {w}{\scrh( K(n+s,\xi,\xi))} \right) }{K(n+s,\xi,\xi)} = K_\infty(z,w)
\]
uniformly in $s\in [0,1]$, which is equivalent to \eqref{eqn:Kcontlimit1}.
\end{proof}

\begin{proof}[Proof of Theorem~\ref{theorem2}]
This is merely the specialization of Theorem~\ref{theorem3} to the special case $\sigma_- = \sigma_+ = \beta = 1$.
\end{proof}

\begin{proof}[Proof of Theorem~\ref{theorem1}]
This is the specialization of Theorem~\ref{theorem2} to the special case $g(r)=r / \eta$.
\end{proof}

\subsection{Orthogonal polynomials on the unit circle}
Let $\nu$ be a probability measure on $\partial\bbD$ such that $\supp \nu$ is not a finite set, $\varphi_n$ its orthogonal polynomials, and $k_n$ the CD kernels \eqref{eqn:OPUCCD}.  In terms of reflected polynomials $\varphi_n^*(\zeta) = \zeta^n \overline{\varphi_n(1/\overline \zeta)}$, they satisfy the CD formula
\[
k_n(\zeta,\omega) = \frac{ \varphi_n^*(\zeta) \overline{ \varphi_n^*(\omega) } - \varphi_n(\zeta) \overline{ \varphi_n(\omega) }} {1 - \zeta \overline \omega}.
\]
A way to relate orthogonal polynomials on the unit circle to an energy-periodic canonical system was described in \cite[Section 6]{EichLukSimanek}.  At the level of functions, it relates OPUC with Carath\'eodory function
\begin{equation}\label{eqnCaratheodory}
F(\zeta) = \int_{\partial\bbD} \frac{ e^{i\theta} + \zeta}{ e^{i\theta} - \zeta} \,d\nu(e^{i\theta}), \qquad \zeta \in \bbD
\end{equation}
to the canonical system with the $2\pi$-periodic Weyl $m$-function
\begin{equation}\label{eqnOPUCtoWeyl}
m(z) = i F(e^{iz}), \qquad z \in \bbC_+.
\end{equation}
We provide further information about this correspondence in the following lemma. Since rotating the measure by $\xi$ as
\[
\int_{\partial\bbD} f(\zeta) \, d\tilde\nu(\zeta) = \int_{\partial\bbD} f(e^{-i\xi}\zeta) \,d\nu(\zeta), \qquad \forall f \in C(\partial\bbD)
\]
replaces orthogonal polynomials $\varphi_n$ by $e^{-in\xi} \varphi_n(e^{i\xi} \zeta)$, it replaces kernels $k_n(\zeta,\omega)$ by $k_n(e^{i\xi} \zeta, e^{i\xi} \omega)$; thus writing down the canonical system normalized at $\xi=0$ suffices to study scaling limits at any point $e^{i\xi} \in \partial\bbD$.

\begin{lemma}
Let $\nu$ be a probability measure on $\partial\bbD$ such that $\supp \nu$ is not a finite set, and let $F$ be its Carath\'eodory function \eqref{eqnCaratheodory}.  Then, the canonical system with Weyl function \eqref{eqnOPUCtoWeyl} has the following properties:
\begin{enumerate}[{\rm(i)}]
\item It corresponds to the measure $\mu$ on $\bbR$ which is $2\pi$-periodic in the sense that $\mu(B) = \mu(2\pi + B)$ for every Borel set $B \subset \bbR$, and 
\[
\mu(B) =  2 \nu(\{ e^{ix} \mid x\in B \})
\]
for all Borel sets $B \subset [0,2\pi)$.
\item It corresponds to the piecewise constant Hamiltonian
\[
H(x) = \begin{pmatrix}
\lvert \psi_n(1) \rvert^2 & \Im( \psi_n(1) \ol{\varphi_n(1)} ) \\
\Im ( \psi_n(1) \ol{\varphi_n(1)} ) & \lvert \varphi_n(1)\rvert^2
\end{pmatrix}, n \le x < n+1,
\]
where $\psi_n$ denote the second kind polynomials. Note that $(HJ)^2 \neq 0$.
\item It has the family of reproducing kernels given for $n\in\bbN$, $s\in [0,1]$ by
\begin{equation}\label{eqnOPUCcanonicals}
K(n+s,z,w) =  \frac { e^{-in(z - \ol w)/2} }{2i (z -\overline{w}) } \left[  e^{i s(z - \ol w)/2} \varphi_n(e^{iz}) \overline{ \varphi_n(e^{iw}) } - e^{-i s(z - \ol w)/2} \varphi_n^*(e^{iz}) \overline{ \varphi_n^*(e^{iw}) } \right].
\end{equation}
\end{enumerate}
\end{lemma}

\begin{proof}
(i) follows from  \eqref{eqnOPUCtoWeyl} by Stieltjes inversion
\[
d\mu(x) = \frac 1\pi \wlim\limits_{\epsilon \downarrow 0} \Im m(x + i \epsilon) \,dx
\]
\[
d\nu(e^{ix}) = \frac 1{2\pi} \wlim\limits_{r \uparrow 1} \Re F(r e^{ix}) \,dx
\]

(ii) Orthogonal polynomials satisfy the Szeg\H o recursion, expressed by the Szeg\H o transfer matrices
\[
S_{n+1}(\zeta) = A(\alpha_n,\zeta) S_n(\zeta), \quad S_0(\zeta)=I, \quad A(\alpha,\zeta) = \frac 1{\sqrt{1-\lvert \alpha \rvert^2}} \begin{pmatrix}
\zeta & -\ol\alpha \\
-\alpha \zeta & 1
\end{pmatrix}.
\]
The derivation in \cite[Section 6]{EichLukSimanek}, expressed in the conventions of this paper, shows that this corresponds to the monotonic family of transfer matrices
\[
M(n,z) = e^{inz/2} \cC^{-1}  j_1 S_n(e^{iz})^{-1} j_1 \cC
\]
where $\cC = \frac 1{1+i} (\begin{smallmatrix} 1 & -i \\ 1 & i \end{smallmatrix})$ and $j_1 =  ( \begin{smallmatrix}  0 & 1 \\ 1 & 0 \end{smallmatrix})$, and that after the gauge change $W(n,z) = M(n,z) M(n,0)^{-1}$, it obeys
\[
W(n+1,z) = W(n,z) M(n,0) e^{-zJ/2} M(n,0)^{-1}  
\]
and therefore, since $M(n,0)^* J M(n,0) = J$,  is the transfer matrix associated with the Hamiltonian
\[
H(x)  = \frac 12 M(n,0) M(n,0)^*, \quad n \le x < n+1.
\]
Since
\[
M(n,z) = e^{-inz/2} \begin{pmatrix} \frac{\psi_n(e^{iz}) + \psi_n^*(e^{iz})}2 &  \frac{ \psi_n(e^{iz}) - \psi_n^*(e^{iz})}{2i} \\  \frac{ \varphi^*_n(e^{iz}) - \varphi_n(e^{iz})}{2i} &  \frac{ \varphi_n(e^{iz}) + \varphi_n^*(e^{iz})}{2} \end{pmatrix}
\]
this gives the value of the piecewise constant Hamiltonian.

(iii) Let $s\in [0,1]$. Since $H$ is constant on $[n,n+s]$, solving \eqref{eqn:CanonicalConvention1} gives
\[
W(n+s,z) = W(n,z) e^{-szHJ} = W(n,z) M(n,0) e^{-szJ/2} M(n,0)^{-1}  
\]
from which a direct calculation gives \eqref{eqnOPUCcanonicals}.
\end{proof}

Note that the formula \eqref{eqnOPUCcanonicals} can be used in two ways to evaluate the kernel with an integer index, by using $s=1$ or by using $s=0$ with $n$ shifted by $1$; compatibility of the two answers can be verified by the property of Szeg\H o transfer matrices
\[
A(\alpha,\omega)^* \begin{pmatrix} -1 & 0 \\ 0 & 1 \end{pmatrix} A(\alpha,\zeta) =  \begin{pmatrix} \omega & 0 \\ 0 & 1 \end{pmatrix}^* \begin{pmatrix} -1 & 0 \\ 0 & 1 \end{pmatrix}  \begin{pmatrix} \zeta & 0 \\ 0 & 1 \end{pmatrix}.
\]

The intervals $[n,n+1]$ have constant Hamiltonians, but they are not indivisible intervals; this is a qualitative difference compared to OPRL, and it affects the next step. When the kernel at $n+s$ is expressed as a linear combination the kernels at $n, n+1$, the formula is different from the OPRL case:

\begin{corollary}
The canonical system kernels \eqref{eqnOPUCcanonicals} associated to OPUC satisfy
\begin{equation}\label{eqn:OPUCkernelsLinComb}
K(n+s,z,w) = \frac{\sin((1-s)(z-\ol w)/2)}{\sin((z-\ol w)/2)} K(n,z,w) +  \frac{\sin(s(z-\ol w)/2)}{\sin((z-\ol w)/2)} K(n+1,z,w)
\end{equation}
for $s\in [0,1]$.
\end{corollary}

\begin{proof}
The equation  \eqref{eqnOPUCcanonicals} expresses the three kernels at $n+s, n, n+1$ as linear combinations of the two functions $\varphi_n(e^{iz}) \ol{\varphi_n(e^{iw})}$,  $\varphi_n^*(e^{iz}) \ol{\varphi_n^*(e^{iw})}$, so the proof of \eqref{eqn:OPUCkernelsLinComb} is a linear algebra calculation.
\end{proof}

\begin{lemma}\label{lemmaOPUCsubexponential}
Assume that the canonical system kernels \eqref{eqnOPUCcanonicals} associated to OPUC satisfy $\lim_{n\to\infty} K(n,\xi,\xi) = \infty$ and
\begin{equation}\label{eqn:contlimitOPUC}
\lim_{t\to\infty}\frac{K\left(t,\xi+\frac{z}{\scrh(K(t,\xi,\xi))},\xi+\frac{w}{\scrh(K(t,\xi,\xi))}\right)}{K(t,\xi,\xi)}=K_\infty(z,w),
\end{equation}
for some $\scrh$  regularly varying of index $\rho > 0$ and $K_\infty \not\equiv 1$. Then
\[
\lim_{n\to\infty} \frac{ K(n+1,\xi,\xi) }{ K(n,\xi,\xi)} = 1.
\]
\end{lemma}

\begin{proof}
Without loss of generality we take $\xi=0$. Imitating the proof of Lemma~\ref{lemmaOPRLsubexponential}, we get to
\begin{align*}
&\frac{K(t_l,\frac{z}{\scrh(K(t_l,0,0))},\frac{w}{\scrh(K(t_l,0,0))})}{K(t_l,0,0)} \\
& \quad =\frac{K\left(n_l+\delta_l,\frac{\alpha_l z}{\scrh(K({n_l+\delta_l},0,0) )} ,\frac{ \alpha_l w}{\scrh(K(n_l+\delta_l,0,0) )} \right) }{K(n_l + \delta_l, 0,0) }  x^{-1/\rho}
\frac{\sin\left(\frac{ s(z-\ol w)}{2 \scrh(K(t_l,0,0))}\right)}{\sin\left(\frac{ z-\ol w}{2 \scrh(K(t_l,0,0))}\right)} \\
& \qquad +
\frac{K\left(n_l,\frac{ \beta_l  z}{\scrh(K(n_l,0,0))} ,\frac{ \beta_l  w}{\scrh(K(n_l,0,0) )}  \right) }{K(n_l,0,0) }
c x^{-1/\rho} \frac{\sin\left(\frac{ (1-s)(z-\ol w)}{2 \scrh(K(t_l,0,0))}\right)}{\sin\left(\frac{ z-\ol w}{2 \scrh(K(t_l,0,0))}\right)}.
\end{align*}
Since $\scrh(x) \to \infty$ as $x\to \infty$ (see \cite[Prop. 1.5.1]{bingham.goldie.teugels:1989}),  taking $l \to \infty$ gives
\[
K_\infty(z,w)=K_\infty\left(z x^{-1},w x^{-1}\right) x^{-1/\rho} s +K_\infty\left(z c^\rho x^{-1} , w c^\rho x^{-1} \right)x^{-1/\rho} c (1-s).
\]
Remarkably, the different interpolation of kernels leads to the same functional equation for the limit kernel as in OPRL, so the rest of the proof follows as in the proof of Lemma~\ref{lemmaOPRLsubexponential}.
\end{proof}

\begin{proof}[Proof of Theorem~\ref{theorem:OPUC}]
Let $\scrg = u/2$ and $\scrh(x) = f(2x)$. After rotating the measure so that $\xi=0$ and passing to the canonical system, (i) reformulates as \eqref{eqn:muLocalBehavior}. By the CD formula for OPUC,
\[
k_n(\zeta,\omega) = \frac{ \varphi^*_n(\zeta) \ol{\varphi^*_n(\omega)} - \varphi_n(\zeta) \ol{\varphi_n(\omega)} }{1 - \zeta \ol\omega}
\]
so
\[
K(n,z,w) = \frac{e^{in(z - \ol w)/2} }{2i (z-\ol w)} (e^{i (z-\ol w)} - 1) k_n(e^{iz}, e^{iw})
\]
We note that $K(n,0,0) = k_n(1,1)/2$ so $\scrh(K(n,0,0)) = s(k_n(1,1))$. Moreover, for any $\tau_n \to \infty$,
\[
\frac{ e^{i (z - \ol w) / \tau_n} - 1 } { i (z - \ol w) / \tau_n } \to 1, \quad n \to\infty,
\]
uniformly on compacts, so these factors will vanish in the scaling limits.

Theorem~\ref{theorem:main} implies the equivalence of (i), (ii), and \eqref{eqn:contlimitOPUC} with $K_\infty\not\equiv 1$ and $\scrh$ regularly varying with index $\rho>0$. Since (i) implies $\lim_{n\to\infty} K(n,\xi,\xi) = \infty$, by Lemma~\ref{lemmaOPUCsubexponential}, (i) and \eqref{eqn:contlimitOPUC} implies (iii).

Conversely, if (iii) holds, it implies $K(n+1,\xi,\xi) / K(n,\xi,\xi) \to 1$ as $n\to\infty$ and
\[
\lim_{n\to\infty}\frac{K\left(n, \xi+\frac{z}{\scrh(K(n,\xi,\xi))},\xi+\frac{w}{\scrh(K(n,\xi,\xi))}\right)}{K(n,\xi,\xi)}=K_\infty(z,w).
\]
As in the proof of Theorem~\ref{theorem3}, this implies
\[
\lim_{n\to\infty}\frac{K\left(n+1, \xi+\frac{z}{\scrh(K(n,\xi,\xi))},\xi+\frac{w}{\scrh(K(n,\xi,\xi))}\right)}{K(n,\xi,\xi)}=K_\infty(z,w).
\]
Using the interpolation formula \eqref{eqn:OPUCkernelsLinComb} and 
\[
\lim_{n\to\infty} \frac{\sin \left(s \frac{ z-\ol w}{2 \scrh(K(n,\xi,\xi))} \right) }{\sin \left(\frac{ z-\ol w}{2 \scrh(K(n,\xi,\xi))} \right) } = s,
\]
this shows
\[
\lim_{n \to \infty} \frac{ K\left(n+s, \xi + \frac {z}{\scrh( K(n+s,\xi,\xi))} ,  \xi + \frac {w}{\scrh( K(n+s,\xi,\xi))} \right) }{K(n+s,\xi,\xi)} = K_\infty(z,w)
\]
uniformly in $s\in [0,1]$, which is equivalent to \eqref{eqn:contlimitOPUC}.
\end{proof}

\newpage

\section{Bulk universality and spectral type} \label{sec:SpectralType}

In this section, we explore the interplay of bulk universality and spectral type of $\mu$ through a few brief remarks.

Historically, bulk universality (sine kernel asymptotics) was proved under conditions which included a continuity or Lebesgue point condition for the Radon--Nikodym derivative $d\mu(\xi) / d\xi$ at the point, with a positive value at $\xi$. For this reason, bulk universality was closely associated with the absolutely continuous part of $\mu$. Our local condition \eqref{theorem1condition} on $\mu$ makes it apparent that bulk universality at a single point can occur even for a pure point measure:

\begin{lemma}\label{lemmaBulkUniversalityPurePointMeasure}
Let
\[
\mu = \sum_{j=1}^\infty \frac 1{j(j+1)} (\delta_{1/j} + \delta_{-1/j}).
\]
Then $ \mu([0,\pm \epsilon)) / \epsilon \to 1$ as $\epsilon \to 0$. In particular, the sine kernel asymptotics \eqref{eqntheorem1limit}  holds at $\xi = 0$ with $\eta = 1$.
\end{lemma}

\begin{proof}
For $\frac 1{n+1} < \epsilon \le \frac 1n$, $\mu([0,\epsilon)) = \sum_{j=n+1}^\infty \frac 1{j(j+1)} = \frac 1{n+1}$ so
\[
\frac { \epsilon^{-1} } { \epsilon^{-1} + 1} \le \frac{ \mu([0,\epsilon)) }{\epsilon} < 1
\]
and therefore $ \mu([0, \epsilon)) / \epsilon \to 1$ as $\epsilon \to 0$. Analogously, $ \mu((-\epsilon,0)) / \epsilon \to 1$ as $\epsilon \to 0$. By Theorem~\ref{theorem1},  \eqref{eqntheorem1limit}  holds at $\xi = 0$ with $\eta = 1$.
\end{proof}

Nonetheless, bulk universality on some set of energies implies that the measure is $1$-dimensional there: 

\begin{theorem}\label{thmBulkUniversality1dimensional}
If bulk universality holds on some set $A$ in the sense that for every $\xi \in A$, the kernels have scaling limit \eqref{RegularlyVaryingBulkUniversalityLimit} with regularly varying scaling, then $\mu$ is $1$-dimensional on $A$ in the sense that $\chi_A \,d\mu$ is $h^\alpha$-continuous for every $\alpha < 1$; $h^\alpha$ denotes the $\alpha$-dimensional Hausdorff measure.
\end{theorem}

\begin{proof}
By Theorem~\ref{theorem2}, for every $\xi \in A$, the limit
\[
\lim_{r\to\infty} \scrg(r) \mu\left(\left(\xi - \tfrac 1r, \xi + \tfrac 1r \right) \right)
\]
is nonzero, and $\scrg$ is regularly varying with index $1$. In particular, for every $\alpha < 1$, $\scrg(r) r^{-\alpha} \to \infty$, so
\[
\lim_{r\to\infty} r^\alpha \mu\left(\left(\xi - \tfrac 1r, \xi + \tfrac 1r \right) \right) = \infty.
\]
This is interpreted as an upper $\alpha$-derivative with respect to Hausdorff measure $h^\alpha$. By Rogers--Taylor \cite{RogersTaylor1,RogersTaylor2} (see also \cite{Last}, \cite[Section 6.3]{LukicBook}), on the set $S_\alpha \supset A$ where
\[
\limsup_{r\to\infty} r^\alpha \mu\left(\left(\xi - \tfrac 1r, \xi + \tfrac 1r \right) \right) = \infty,
\]
$\chi_{S_\alpha} \,d\mu$ is continuous with respect to $h^\alpha$.
\end{proof}

In the remainder of this section, we discuss sparse decaying Jacobi matrices. We call a Jacobi matrix  sparse if  there exists a sequence $N_j$ with $N_{j+1} / N_j \to \infty$ such that $a_n = 1$, $b_n = 0$ for all $n \notin \{N_j \mid j \in \bbN_0 \}$. We call it decaying if $a_n \to 1$, $b_n \to 0$ as $n\to \infty$, since it is then a decaying perturbation of the free Jacobi matrix. The spectral type of a sparse decaying Jacobi matrix on its essential spectrum $[-2,2]$ is completely understood \cite{Pearson,KiselevLastSimon}: it has pure a.c. spectrum on $[-2,2]$ if it is a Hilbert--Schmidt perturbation of the free Jacobi matrix, and pure singular spectrum on $[-2,2]$ otherwise.

The first  examples of bulk universality without a.c.\ spectrum were found within this class: for a fixed decaying sequence $(v_j)_{j=1}^\infty \notin \ell^2$,  Breuer \cite{breuer:2011} proved that there exist functions $\tilde N_k(N_1,\dots,N_k)$ such that with the recursive choice $N_{k+1} =\tilde N_k(N_1,\dots,N_k)$ and with $a_n \equiv 1$, $b_{N_j} = v_j$ for all $j$ and $b_n = 0$ otherwise, the sine kernel asymptotics
\begin{equation}\label{eqnBreuerScaling}
\lim_{n\to \infty} \frac{K(n,\xi+\frac zn, \xi + \frac wn ) }{K(n,\xi,\xi)} = \frac{ \sin ((4-\xi^2)^{-1/2} (z - \overline w)) }{ (4-\xi^2)^{-1/2} (z-\overline w)}
\end{equation}
hold for every $\xi \in (-2,2)$.   Note the explicit factor of $n$ in the scaling limit \eqref{eqnBreuerScaling}, as opposed to a regularly varying function of $K(n,\xi,\xi)$.

Our first remark is that Breuer's examples are within the scope of this paper:

\begin{lemma}\label{lemmaSparseRegularlyVarying}
For every sparse decaying Jacobi matrix and  every $\xi \in (-2,2)$, the function $K(t,\xi,\xi)$ is a regularly varying function of $t$ with index $1$ and
\[
K(t,\xi,\xi) \sim 2 t \frac{ p_n(\xi)^2 - \xi a_n p_n(\xi) p_{n-1}(\xi) + a_n^2 p_{n-1}(\xi)^2 }{ 4 - \xi^2}, \qquad n = \lfloor t \rfloor, \qquad t \to \infty.
\]
\end{lemma}

\begin{proof}
Denote $\theta = \arccos(\xi/2)$ and diagonalize
\[
\begin{pmatrix}
2 \cos \theta & -1 \\
1 & 0
\end{pmatrix}
=
U
\begin{pmatrix}
e^{i\theta} & 0 \\
0 & e^{-i\theta}
\end{pmatrix}
U^{-1}, \qquad 
U = \begin{pmatrix}
e^{i\theta} & e^{-i\theta} \\
1 & 1
\end{pmatrix}.
\]
Introduce vectors $A_n$ by
\[
U^{-1} \begin{pmatrix}
p_n(\xi) \\
a_n p_{n-1}(\xi)
\end{pmatrix}
=
\begin{pmatrix}
e^{i(n-1)\theta} & 0 \\
0 & e^{-i(n-1)\theta}
\end{pmatrix}
A_n.
\]
Then the Jacobi recursion rewrites as 
\[
A_{n} =
\begin{pmatrix}
e^{-i(n-1)\theta} & 0 \\
0 & e^{i(n-1)\theta}
\end{pmatrix}
U^{-1}
\begin{pmatrix}
\frac{\xi - b_{n}}{a_{n}} & - \frac 1{a_{n}} \\
a_{n} & 0
\end{pmatrix}
U
\begin{pmatrix}
e^{i(n-2)\theta} & 0 \\
0 & e^{-i(n-2)\theta}
\end{pmatrix}
A_{n-1}.
\]
In particular, $A_n$ is constant on $N_j \le n < N_{j+1}$. This implies constancy of
\[
\lVert A_n \rVert^2 =2 \frac{ p_n(\xi)^2 - \xi a_n p_n(\xi) p_{n-1}(\xi) + a_n^2 p_{n-1}(\xi)^2 }{ 4 - \xi^2}
\]
on the same intervals. 
Moreover, since $a_n \to 1$ and $b_n \to 0$, $\lVert A_{n} \rVert / \lVert A_{n-1} \rVert \to 1$ as $n\to\infty$.

Consider the function $f$ defined by $f(t) = \lVert A_n \rVert^2$ for $n \le t < n+1$. For every $\lambda > 1$, for large enough
$n$, there is at most one jump in the value of $\lVert A_n \rVert^2$ between $t$ and $\lambda t$, so $f(\lambda t) / f(t) \to 1$
as $x \to + \infty$. By Karamata's theorem \cite{Karamata30,Karamata33} (see also \cite[Thm 1.5.11]{bingham.goldie.teugels:1989}), the function $\scrg(t) = \int_0^t f(s) \,ds$ is regularly varying with index $1$ 
and $\scrg(t) \sim t f(t)$. Meanwhile,
\[
p_n(\xi)^2 = \lVert A_n \rVert^2 - 2 \Re \left(  e^{2in\theta} (A_n)_1 \overline{(A_n)_2} \right) 
\]
and partial sums of the oscillatory part are bounded by 
\[
\left\lvert \sum_{n=s+1}^t e^{2in\theta} \right\rvert \le \frac 2{\lvert 1 - e^{2i\theta} \rvert} = \frac 2{\sqrt{4-\xi^2}}
\]
so for integers $s,t$ with $N_j \le s < t \le N_{j+1}$,
\[
\left\lvert K(t,\xi,\xi) - K(s,\xi, \xi) - (t-s) \lVert A_{N_j} \rVert^2 \right\rvert \le \frac 2{\sqrt{4-\xi^2}} \lVert A_j \rVert^2.
\]
Since the expression in the absolute value is piecewise linear in $s,t \in [N_j, N_{j+1}]$ and the inequality holds at endpoints of the linear parts, it holds for all $s, t \in [N_j, N_{j+1}]$. 
By telescoping,
\begin{equation}\label{eqn:16aug1}
\lvert K(t,\xi,\xi) - \scrg(t)  \rvert \le \frac 2{\sqrt{4-\xi^2}}  \sum_{j: N_j \le t} \lVert A_j \rVert^2.
\end{equation}
The limit
\[
\lim_{j\to\infty} \frac{ \lVert A_{j+1} \rVert^2}{  \lVert A_j \rVert^2  (N_j - N_{j-1}) } = 0
\]
implies by the Stolz--Ces\`aro theorem that $\sum_{j=1}^n \lVert A_j \rVert^2 / \scrg(N_n) \to 0$ as $n\to\infty$, and by monotonicity of $\scrg$ that $\sum_{j: N_j \le t} \lVert A_j \rVert^2 /  \scrg(t) \to 0$ as $t \to \infty$.  Thus, \eqref{eqn:16aug1} implies $K(t,\xi,\xi) \sim \scrg(t)$.
\end{proof}

Combining this with our Theorem~\ref{theorem2} gives very precise asymptotic behavior of the spectral measure on intervals, for Jacobi matrices in Breuer's class:

\begin{corollary}\label{corDecayingSparseMeasure}
For every $\xi \in (-2,2)$, for any decaying sparse Jacobi matrix for which \eqref{eqnBreuerScaling} holds, 
\[
\lim_{n\to\infty} \scrg_\xi(n) \mu\left(\left( \xi - \tfrac 1{n} , \xi \right)\right) = \lim_{n\to\infty} \scrg_\xi(n) \mu\left(\left[\xi, \xi + \tfrac 1{n} \right)\right) = 1
\]
where
\[
\scrg_\xi(n) =  \frac{2\pi n}{\sqrt{4-\xi^2}} \left( p_n(\xi)^2 - \xi a_n p_n(\xi) p_{n-1}(\xi) + a_n^2 p_{n-1}(\xi)^2 \right).
\]
\end{corollary}

\begin{proof}
The function $\scrg$ from the previous proof is continuous and strictly increasing, so Lemma~\ref{lemmaSparseRegularlyVarying} can be restated in the form $t \sim \scrg^{-1}(K(t,\xi,\xi))$. Moreover, since $K(t,\xi,\xi)$ is regularly varying, \eqref{eqnNevaiConditionRelated} holds. Thus, Theorem~\ref{theorem2}(iv) holds with $\scrh = \scrg^{-1}$, and this implies Theorem~\ref{theorem2}(i), that is,
\[
\lim_{t\to\infty} K(t,\xi,\xi) \mu\left(\left ( \xi - \tfrac {\pi \sqrt{4-\xi^2}}{t}, \xi \right)\right) = \lim_{t\to\infty} K(t,\xi,\xi) \mu\left(\left[\xi, \xi + \tfrac {\pi\sqrt{4-\xi^2}}{t} \right)\right) = 1.
\]
Applying this to the sequence $t = n\sqrt{4-\xi^2}$ and combining with
\[
K(n \sqrt{4-\xi^2}, \xi,\xi) \sim \sqrt{4-\xi^2}K(n , \xi,\xi) \sim g_\xi(n), \qquad n\to\infty
\]
(by Lemma~\ref{lemmaSparseRegularlyVarying} and regular variation with index $1$)
concludes the proof.
\end{proof}

By a result of Zlato\v s \cite{Zlatos}, sparse decaying Jacobi matrices always obey
\[
\lim_{\epsilon \to 0} \frac{ \log \mu((\xi-\epsilon,\xi))}{\log \epsilon}  = \lim_{\epsilon \to 0} \frac{ \log \mu([\xi,\xi+\epsilon))}{\log \epsilon}  = 1
\]
and therefore have $1$-dimensional spectral measures on $(-2,2)$; note that Corollary~\ref{corDecayingSparseMeasure} gives a more precise statement about the local behavior of the spectral measure, but within the narrower class of \cite{breuer:2011}. It is natural to conjecture:

\begin{conjecture}
For every sparse decaying Jacobi matrix $J$, its spectral measure $\mu$ has a unique  tangent measure at every $\xi \in (-2,2)$, and this tangent measure is the Lebesgue measure.
\end{conjecture}

\section{The local distribution of zeros} \label{sec:FreudLevin}

In this section, we consider applications to local zero distributions. We will begin with a generalization of  the Freud--Levin theorem, formulated in the general context of Hermite--Biehler functions.

We will repeatedly use the following observations.  For $E\in \HB^*$, the function $E^\sharp / E$ maps $\bbC_+$ into $\bbD$, and maps $\bbR$ into $\partial\bbD$. For fixed $x \in \bbR$, $K_E(z,x) = 0$ if and only if $z \in \bbR$, $z \neq x$, and $(E^\sharp / E)(z) = (E^\sharp / E)(x)$; this follows from \eqref{eq:25}, and in particular, non-real zeros are ruled out by \eqref{eq:37}.
Moreover, by the Cauchy--Riemann relations and local properties of analytic functions, there is a strictly increasing continuous choice of argument $\varphi:\bbR \to \bbR$ such that
\begin{equation}\label{eqn:ChoiceOfArgument}
(E^\sharp/ E)(x) = e^{i\varphi(x)}.
\end{equation}
This representation implies that for $a, b \in \partial\bbD$ with $a \neq b$, the solutions of $E^\sharp / E = a$ and $E^\sharp / E = b$ strictly interlace. Moreover, by the strict interlacing property, the following are equivalent:
\begin{enumerate}[(i)]
\item
 $A$ has infinitely many positive zeros  (these are solutions of $E^\sharp/E = - 1$)
 \item $B$ has infinitely many positive zeros (these are solutions of $E^\sharp/E = 1$)
 \item $K_E(\cdot,0)$ has infinitely many positive zeros (these are solutions of $E^\sharp/E = (E^\sharp/E)(0)$)
 \item  $\varphi(x) \to \infty$ as $x \to\infty$
\end{enumerate}

\begin{theorem} \label{thm:FreudLevinGeneral}
Consider a sequence of Hermite-Biehler functions $E_n = A_n - i B_n \in \HB^*$, a point $\xi\in\bbR$ and scaling sequence $(\tau_n)_{n=1}^\infty$ such that
\[
\lim_{n\to\infty} \frac{ K_{E_n}\left(\xi + \frac{z}{\tau_n}, \xi + \frac{w}{\tau_n} \right) }{ K_{E_n}(\xi,\xi)}  = K_E(z,w)
\]
uniformly on compacts, for some Hermite--Biehler function $E = A - i B \in \HB^*$.
If $B$ has infinitely many positive zeros, then:
\begin{enumerate}[(i)]
\item For every $k \ge 0$, for all large enough $n$, $A_n$ has at least $k$ zeros greater than $\xi$; in other words its $k$-th zero to the right of $\xi$, denoted $\xi_k^{(n)}$, is well-defined for all large enough $n$.
\item Denoting by $\theta$ the smallest positive zero of $K_E(\cdot,0)$, 
\[
\limsup_{n\to\infty} \tau_n ( \xi_1^{(n)} - \xi )  \le \theta.
\]
\item If the limit 
\[
 \lim_{n\to\infty} \tau_n ( \xi_1^{(n)} - \xi )
\]
exists, denote its value by $\kappa_1$ and denote by $\kappa_2 < \kappa_3 < \dots$ all the zeros of $K_{E} (\cdot, \kappa_1)$ in $(\kappa_1, \infty)$. Then for every $k \in \bbN$,
\begin{equation}\label{eqn:FreudLevinZeros1}
\lim_{n\to\infty} \tau_n ( \xi_k^{(n)} - \xi ) = \kappa_{k}.
\end{equation}
 If, in addition, $K_E(\cdot, \kappa_1)$ has at least $m$ zeros in $(-\infty,\kappa_1)$ for some $m\in \bbN$, then $A_{n}$ has at least $m$ zeros in $(-\infty,\xi)$ for all large enough $n$, and \eqref{eqn:FreudLevinZeros1} holds also for $k=0,-1,\dots, -m+1$.
 \end{enumerate}
\end{theorem}

\begin{proof}
Since the shift by $\xi$ and scaling by $\tau_n$ can be composed with $E_n$, there is no loss of generality in assuming $\xi=0$ and $\tau_n = 1$.

(i) Fix $k \in\bbN$. Since $K_{E_n}(\cdot, 0) / K_{E_n}(0,0) \to K_E(\cdot,0)$, by the Hurwitz theorem, for all large enough $n$, $K_{E_n}(\cdot, 0)$ has at least $k$ zeros with $\Re z > 0$. Thus, there are at least $k$ strictly positive solutions $z$ of $(E_n^\sharp/ E_n)(z) = (E_n^\sharp/ E_n)(0)$. Including $z = 0$, this means at least $k+1$ zeros in $[0,\infty)$.  By the strictly interlacing property, there are at least $k$ positive zeros of $A_n$.

(ii) Denote by $\theta_n$ the smallest positive zero of $K_{E_n}(\cdot, 0)$. By the Hurwitz theorem, $\theta_n \to \theta$, and by the strictly interlacing property, since $(E_n^\sharp / E_n)(\theta_n) = (E_n^\sharp / E_n)(0)$,  $\xi_1^{(n)} \in [0, \theta_n)$. Thus, $\limsup_{n\to\infty} \xi_1^{(n)} \le \limsup_{n\to\infty} \theta_n = \theta$.

(iii) By the Hurwitz theorem, solutions of $K_{E_n}(\cdot,\kappa_1)$ converge to solutions of $K_E(\cdot, \kappa_1)$. The claims for eigenvalues below $\xi$ follow analogously.
\end{proof}

\begin{proof}[Proof of  Corollary~\ref{cor:ZeroDistributionPolynomial}]
(i), (ii) follow by applying Theorem~\ref{thm:FreudLevinGeneral}(i),(ii) to the sequence $E_n = p_n + i p_{n-1}$.

(iii) follows by applying Theorem~\ref{thm:FreudLevinGeneral}(iii) to the subsequence $E_{n_k} = p_{n_k} + i p_{n_k-1}$.
\end{proof}

For further applications,  we need a rewriting of the limit kernel in the case $\sigma_- = \sigma_+$. Let us factor Bessel functions as
\[
J_\nu(z) = \left( \frac z2 \right)^\nu F_\nu(z), \qquad F_\nu(z) = \sum_{n=0}^\infty \frac{ (-1)^n }{ n! \Gamma(n+\nu+1) } \left( \frac z2 \right)^{2n}.
\]
In particular, we note that $F_\nu$ is entire, even, and $F_\nu^\sharp = F_\nu$. A rewritting of a kernel in terms of functions $F_\nu$ is essentially a rewriting in terms of Bessel functions, without branch ambiguities.

\begin{lemma}\label{lemma:LimitKernelFisherHartwig}
In the case $\sigma_- = \sigma_+ = 1$, the limit kernel $K_{\sigma_-, \sigma_+,\beta}$ is of the form
\begin{equation}\label{eqn:LimitKernelFisherHartwig}
K_{1,1,\beta}(z,w) =\Gamma(\tfrac \beta 2 + 1) \Gamma(\tfrac \beta 2) \frac{z  F_{\beta/2}(\kappa z) F_{\beta/2-1}(\kappa \ol w) -  F_{\beta/2-1}(\kappa z)  \ol w F_{\beta/2}(\kappa \ol w) }{z -\ol w}
\end{equation}
where
\begin{equation}\label{eqn:kappaEvenMeasure}
\kappa = 2^{1+1/\beta} \Gamma(\tfrac \beta 2 + 1)^{2/\beta}.
\end{equation}
\end{lemma}

\begin{proof}
Specializing the formulas from Definition~\ref{defn:LimitKernels}, in the case $\sigma_- = \sigma_+=1$ we obtain $\alpha = (\beta - 1) /2$, and \eqref{eqn:kappaEvenMeasure} follows from the Legendre duplication formula.

For any $n$, the identity
\[
\frac{(\alpha)_n + (\alpha+1)_n }{2 (2\alpha+1)_n} = \frac{(2\alpha+n) (\alpha+1)_{n-1} }{2 (2\alpha+1)_{n-1} (2\alpha+n)} = \frac{ (\alpha)_n }{ (2\alpha)_n}
\]
implies that
\[
\frac{ M(\alpha,2\alpha+1,z) +  M(\alpha+1,2\alpha+1,z) } 2 = M(\alpha,2\alpha,z)
\]
and therefore 
\[
A(z) = e^{i\kappa z} M\left(\frac{\beta- 1}2 , \beta - 1, - 2 i \kappa z \right).
\]
This is similar to the form of $B$,
\[
B(z) = z e^{i\kappa z} M\left( \frac{\beta+1}2 , \beta + 1, -2 i \kappa z \right).
\]
We use a connection between the Kummer hypergeometric function and modified Bessel function \cite[Remark 4.2]{eichinger.woracek:homo-arXiv},
\[
 e^{iz} M(\nu+\tfrac 12, 2\nu + 1, - 2 iz) =\Gamma(\nu+1) \frac{ I_\nu(-iz) }{ (-iz/2)^\nu} = \Gamma(\nu+1) F_\nu(z).
\]
This identity implies that 
\[
A(z) 
=  \Gamma(\tfrac \beta 2) F_{\beta/2-1}(\kappa z)
\]
\[
B(z) 
 = z \Gamma(\tfrac \beta 2 + 1) F_{\beta/2}(\kappa z)
\]
from which  \eqref{eqn:LimitKernelFisherHartwig} follows.
\end{proof}

This allows us to describe precisely the local distribution of zeros of even measures around $0$ with a Fisher--Hartwig singularity at $0$, in terms of zeros of Bessel functions \eqref{eqn:BesselFunctionZeros}. Whereas the Freud--Levin theorem describes the asymptotic distribution up to one free parameter, in this special case, the asymptotic distribution is described exactly, distinguishing between polynomials of even/odd degree:

\begin{lemma}\label{lemma:ZerosEvenMeasure}
Assume that $\nu$ is an even measure on $\bbR$ corresponding to a determinate moment problem, that the function
\[
u(r) = 1 / \nu([0,\tfrac 1r))
\]
is regularly varying of index $\alpha > 0$,  and that $f$ denotes an asymptotic inverse of $u$. Then, with $\kappa = 2^{1+1/\alpha} \Gamma(\tfrac \alpha 2 + 1)^{2/\alpha}$, the following holds at $\xi = 0$:
\begin{enumerate}[(i)]
\item Polynomials of odd degree $p_{2n+1}$ have zeros
\[
\xi_{-n}^{(2n+1)} < \dots < \xi_n^{(2n+1)}
\]
with the symmetry $\xi_{-k}^{(2n+1)} = - \xi_k^{(2n+1)}$ and limits
\begin{equation}\label{eqn:16sep1}
\lim_{n\to\infty} \kappa f(K(2n+1,0,0)) \xi_k^{(2n+1)} = j_{\alpha/2, k}
\end{equation}
with $\kappa$ given by \eqref{eqn:kappaEvenMeasure}.
\item Polynomials of even degree $p_{2n}$ have zeros 
\[
\xi_{-n+1}^{(2n)} < \dots < \xi_n^{(2n)}
\]
with the symmetry $\xi_{-k+1}^{(2n)} = \xi_k^{(2n)}$ and limits
\begin{equation}\label{eqn:16sep2}
\lim_{n\to\infty} \kappa f(K(2n,0,0)) \xi_k^{(2n)} = j_{\alpha/2 - 1, k}
\end{equation}
\end{enumerate}
\end{lemma}

\begin{proof}
(i) Denote by $K_\infty$ the limit kernel \eqref{eqn:LimitKernelFisherHartwig} and by $E = A - iB$ the corresponding Hermite--Biehler function.  Since $B(0) = 0$, the positive zeros of $K_\infty(\cdot ,0)$ are precisely the positive zeros of $B$. Since the function $B(z)$ is a multiple of $F_{\alpha/2}(\kappa z)$, those zeros are precisely $j_{\alpha/2,1} / \kappa < j_{\alpha/2, 2} / \kappa < \dots$. 

By symmetry, $p_{2n+1}$ is odd, so it has a zero at zero: thus, in our notation, $\xi_0^{(2n+1)} = 0$ for all $n$, so by Theorem~\ref{thm:FreudLevinGeneral}, \eqref{eqn:16sep1} follows.

(ii) By (i) and Theorem~\ref{thm:FreudLevinGeneral},  
\[
\limsup_{n\to\infty}  f(K(2n,0,0)) \xi_1^{(2n)} \le  j_{\alpha/2 , 1} / \kappa.
\]
Moreover, fix some sequence $n_l \to \infty$ such that the limit exists,
\[
\lim_{l\to\infty}  f(K(2n_l,0,0)) \xi_1^{(2n_l)}  = \gamma
\]
for some $\gamma \in [0, j_{\alpha/2, 1}/\kappa]$. Going one zero to the left, by Theorem~\ref{thm:FreudLevinGeneral}, the limit
\[
\lim_{l\to\infty}  f(K(2n_l,0,0)) \xi_0^{(2n_l)} 
\]
is the largest zero of $K_\infty(\cdot,\gamma)$ in $(-\infty,\gamma)$. However, by the symmetry $\xi_0^{(2n)} = - \xi_1^{(2n)}$, this limit is $-\gamma$. In particular, $- \gamma < \gamma$ so $\gamma \neq 0$ and $\gamma$ is characterized as the smallest positive number with the property
\[
K_\infty(-\gamma, \gamma) = 0.
\]
By Lemma~\ref{lemma:LimitKernelFisherHartwig}, and since functions $F_\nu$ are even,
\[
K_\infty(- x , x) =  \Gamma(\tfrac \alpha 2 + 1) \Gamma(\tfrac \alpha 2) F_{\alpha/2-1}(\kappa x) F_{\alpha/2}(\kappa x) 
\]
so this is zero if and only if $x \neq 0$ and $F_{\alpha/2}(\kappa x) = 0$ or $F_{\alpha/2-1}(\kappa x) = 0$. In other words, $\kappa \gamma$ is the smallest positive zero of $F_{\alpha/2} F_{\alpha/2 - 1}$, i.e., the smallest positive zero of $AB$. Since $B(0) =0$, by the strict interlacing property, $\kappa\gamma$ must be the smallest positive zero of $A$, and we have proved
\[
\kappa \gamma = j_{\alpha/2 - 1,1}. 
\]
Finally, since the limit $\gamma$ is independent of subsequence, by compactness,
\[
\lim_{n\to\infty} \kappa f(K(2n,0,0)) \xi_1^{(2n)}  = j_{\alpha/2 - 1, 1} / \kappa,
\]
so \eqref{eqn:16sep2} holds for $k=1$. By  Theorem~\ref{thm:FreudLevinGeneral}, rescalings of other zeros converge to other zeros of $A$, i.e., \eqref{eqn:16sep2} holds for all $k$.
\end{proof}

\begin{proof}[Proof of Theorem~\ref{thm:ZerosHardEdge}]
Denote by $\nu$ the even measure on $\bbR$ whose pushforward by the map $x \to x^2$ is the measure $\mu$. Since $\mu$ corresponds to a determinate Stieltjes moment problem, $\nu$ corresponds to a determinate (Hamburger) moment problem (see \cite[Theorem 1]{Chihara82} or \cite[Prop. 3.19]{Schmudgen}). 

Since $\scrg$ is regularly varying with positive index, $\scrg(r) \to \infty$ as $r\to \infty$, so $\nu(\{0\}) = \mu(\{0\}) = 0$. This further implies that
\[
\lim_{r \to \infty} 2\scrg(r^2) \nu([0,\tfrac 1r)) = \lim_{r \to \infty} \scrg(r^2) \mu([0,\tfrac 1{r^2})) = 1. 
\]
Thus, $\nu$ is in the setting of Lemma~\ref{lemma:ZerosEvenMeasure}, with scaling function $u(r) = 2 \scrg(r^2)$ of index $\alpha = 2\beta$. Thus, its asymptotic inverse can be taken to be $f(t) = \sqrt{ \scrh(t/2)}$. By assumption, $\scrg$ has index $\beta> 0$, so $u$ has index $2\beta$.

Moreover, since $\nu$ is even and its pushforward is $\mu$, the orthogonal polynomials for the measure $\mu$ are linked with those for $\nu$ by
\[
p_n(z^2,\mu) = p_{2n}(z,\nu).
\]
Since $p_{2n_-1}(\cdot,\nu)$ is odd, this implies a relation between the Christoffel functions at $0$, $K(n,0,0,\mu) = K(2n,0,0,\nu)$. Moreover, denoting the zeros of $p_n(\cdot, \mu)$ by $\xi_1^{(n)} < \dots < \xi_n^{(n)}$, the zeros of $p_{2n}(\cdot,\nu)$ are
\[
- \sqrt{  \xi_n^{(n)} } <  \dots < - \sqrt{  \xi_1^{(n)} } < \sqrt{  \xi_1^{(n)} } < \dots < \sqrt{  \xi_n^{(n)} }
\]
so by Lemma~\ref{lemma:ZerosEvenMeasure}(ii) written in our current notation,
\[
\lim_{n\to\infty}  \sqrt{ \scrh( K(n,0,0,\mu)/2) }  \sqrt{  \xi_k^{(n)} } =2^{-1-1/(2\beta)} \Gamma( \beta+1)^{-1/\beta}  j_{\beta - 1, k}
\]
for every $k\in \bbN$. Squaring and using regular variation of $\scrh$ with index $1/\beta$ gives \eqref{eqn:HardEdgeZerosResult}.
\end{proof}

%
\appendix

\newpage

\section{Tangent measures}

\begin{definition}
Let $\mu$ be a  locally finite measure on $\bbR^d$. Let $\Sigma^r:\bbR\to\bbR$ be the map $\xi\mapsto r\xi$ and $\Sigma^r_*\mu$ be the pushforward of $\mu$ under $\Sigma^r$.  A measure $\nu$ on $\bbR^d$ is a tangent measure of $\mu$ at $0$ if $\nu$ is locally finite, $\nu(\bbR^d) > 0$, and there exist positive sequences $c_n, r_n$ with $r_n \to \infty$ and $c_n \Sigma_*^{r_n} \mu \to \nu$ weakly as $n\to\infty$.

The set of tangent measures of $\mu$ at $0$ is denoted $\Tan(\mu,0)$.

It is said that $\mu$ has a unique tangent measure at $0$ if there exists $\nu$ such that $\Tan(\mu,0) = \{ c \nu \mid c \in (0,\infty)\}$.

Analogous definitions hold at $\xi\in\bbR^d$, by shifting $\mu$ by $\xi$.
\end{definition}

We note that existence of a tangent measure is not automatic. On the other hand, uniqueness of a tangent measure is sufficient to pass from sequential limits to a limit over $r \to \infty$ and to conclude a scaling property of the unique tangent measure:

\begin{lemma}\label{lemmaA2}
For a locally finite measure $\mu$ on $\bbR$, the following are equivalent:
\begin{enumerate}[(a)]
\item $\Tan(\mu,\xi) = \{ c\nu \mid c \in (0,\infty) \}$ and $\delta_0 \notin \Tan(\mu,\xi)$
\item $\Tan(\mu,\xi) = \{ c \nu \mid c \in (0,\infty) \}$ where $\nu$ is of the form
\begin{equation}\label{limitmeasuretangentmeasure}
d\nu(t) = \begin{cases}
\sigma_-\beta|t|^{\beta-1}\,dt &\text{if}\  t<0,\\
\sigma_+\beta|t|^{\beta-1}\,dt&\text{if}\  t>0,
\end{cases}
\end{equation}
for some  $\sigma_-,  \sigma_+ \in [0,\infty)$ with $\sigma_- + \sigma_+ > 0$ and $\beta > 0$
\item There exist $\sigma_-,  \sigma_+ \in [0,\infty)$ with $\sigma_- + \sigma_+ > 0$ and $\scrg:(0,\infty) \to (0,\infty)$ which is regularly varying with index $\beta > 0$ such that
\[
\lim_{r\to\infty} \scrg(r) \mu\left(\left(\xi- \tfrac 1r,\xi \right)\right) = \sigma_-, \qquad \lim_{r\to\infty} \scrg(r) \mu\left(\left[ \xi,  \xi + \tfrac 1r \right)\right)  = \sigma_+ .
\]
\end{enumerate}
\end{lemma}

\begin{proof}
For notational simplicity, let $\xi = 0$ in this proof.

(b)$\implies$(a) is trivial. 

(c)$\implies$(b): for any $t > 0$, since $\scrg$ is regularly varying with index $\beta > 0$,
\[
\lim_{r\to\infty} \scrg(r) \mu\left(\left(- \tfrac tr,0 \right)\right) = \lim_{r\to\infty} \frac{ \scrg(r)}{\scrg(r/t)} \scrg(r/t) \mu\left(\left(- \tfrac tr,0 \right)\right) = \sigma_- t^\beta
\]
and analogously
\[
\lim_{r\to\infty} \scrg(r) \mu\left(\left[0, \tfrac t r \right) \right) = \sigma_+ t^\beta
\]
so by the Portmanteau theorem,  $\scrg(r)  \Sigma_*^{r}\mu \to \nu$ as $r \to \infty$, with $\nu$ given by \eqref{limitmeasuretangentmeasure}. Moreover, for any $c_n, r_n \to\infty$, if $c_n \Sigma_*^{r_n}\mu$ converges to a nonzero locally finite measure, the ratio $c_n / \scrg(r_n)$ must converge in $(0,\infty)$, and the limit must be a multiple of $\nu$.

(a)$\implies$(c): By \cite[Lemma 2.5]{mattila:2005} and its proof, there exists $\beta \ge 0$ such that $\Sigma_*^{r} \nu = r^{-\beta} \nu$. Moreover, $\nu((-t, t)) = t^{\beta} \nu( (-1,1))$ for all $t > 0$, so $\beta = 0$ implies that $\nu$ is a point mass at $0$; thus, by our assumption, $\beta > 0$. Now $\Sigma_*^{r} \nu = r^{-\beta} \nu$ implies that $\nu$ is of the form \eqref{limitmeasuretangentmeasure}.

By \cite[Lemma 2.5(3)]{mattila:2005}, the function $\scrg(r) = 1 / \mu((-1/r, 1/r))$ is regularly varying with index $\beta$, and by \cite[Lemma 2.5(2)]{mattila:2005}, $\scrg(r) \Sigma_*^r \mu \to \nu$. By the Portmanteau theorem, this implies $\scrg(r) \mu(( \xi - 1/r, \xi)) \to \nu((-1,0)) = \sigma_-$ and similarly $\scrg(r) \mu([\xi, \xi+1/r)) \to \nu([0,1)) = \sigma_+$.
\end{proof}

\section{Portmanteau Theorem}
The goal of this section is to prove the equivalence of item $(iii)$ and $(iv)$ in \Cref{thm:950b}. The equivalence for sequences of bounded measures is well-known and content of Portmanteau's theorem. For sequences of unbounded measures it is proved in \cite{barczy.pap:2006}. However, we are not aware of a reference that proves the claim as stated in \Cref{thm:950b} and thus we add a proof for completeness.

\begin{lemma}\label{lem:950b}
	Let $\k\in\bbN_0$. Let $((\mu_j,p_j))_{j\in J}$ be a net in $\bbE_{\leq\k}$ and $(\mu,p)\in\bbE_{\leq\k}$.  Then 
		$
		{\displaystyle\lim\limits_{j\in J}\mu_j=\mu}
		$ in the $w^*$-topology of $C_c(\bbR)^*$ implies
		\begin{align*}
			\forall a,b\in\bbR,\,a<b,\,\mu(\{a\})=\mu(\{b\})=0:\ 
			\lim_{j\in J}\mu_j((a,b))=\mu((a,b))
		\end{align*}
	If there is $j_0\in J$ and $c>0$ so that $\sup_{j\geq j_0}\|\mu_j\|_\kappa<c$, then also the converse implication holds. 
\end{lemma} 
\begin{proof}
	Assume that $\lim_{j\in J}\mu_j=\mu$ in $C_c(\bbR)^*$ and take $a<b$ with $\mu(\{a\})=\mu(\{b\})=0$.
	For every $n\in\bbN$ let $\{f_n,g_n,h_n\}$ be a partition of unity subordinate to 
	\begin{align*}
		&\left(a-\tfrac{2}{n},a+\tfrac{2}{n}\right)\cup \left(b-\tfrac{2}{n},b+\tfrac{2}{n}\right),\quad &(\text{ for }  f_n)\\
		&\left(a+\tfrac{1}{n},b-\tfrac{1}{n}\right),\quad &(\text{ for }  g_n)\\
		&\left(-\infty,a-\tfrac{1}{n}\right)\cup \left(b+\tfrac{1}{n},\infty\right),\quad &(\text{ for }  h_n)
	\end{align*}
By the dominated convergence theorem, we have
\begin{align*}
	&\lim_{n\to\infty}\int_\bbR f_nd\mu=\mu\left(\{a\}\right)+\mu\left(\{b\}\right),\\
	&\lim_{n\to\infty}\int_\bbR g_nd\mu=\mu\left((a,b)\right).
\end{align*}
Moreover, by construction for every measure $\nu$ 
\[
\int_\bbR g_nd\nu\leq \int_{(a,b)}1d\nu=\nu((a,b))\leq \int_\bbR\left(f_n+g_n\right)d\nu
\]
yielding
\begin{equation}\label{ex:portmanteau1}
\left|\nu((a,b))-\int_\bbR g_nd\nu\right|\leq \int_\bbR f_nd\nu.
\end{equation}

Fix $\epsilon>0$. Choose $n\in\bbN$ so that 
\[
\int_\bbR f_nd\mu<\epsilon.
\]
Moreover, choose $j_0\in J$ so that for all $j\geq j_0$
\[
\left|\int_\bbR f_nd\mu_j-\int_\bbR f_nd\mu\right|<\epsilon\quad\text{and}\quad \left|\int_\bbR g_nd\mu_j-\int_\bbR g_nd\mu\right|<\epsilon.
\]
Then, using \eqref{ex:portmanteau1} we get for all $j\geq j_0$ 
\begin{align*}
	\left|\mu_j((a,b))-\mu((a,b))\right|&\leq \left|\mu_j((a,b))-\int_\bbR g_nd\mu_j\right|\\
	&+\left|\int_\bbR g_nd\mu_j-\int_\bbR g_nd\mu\right|+\left|\int_\bbR g_nd\mu-\mu((a,b))\right|\\
	&\leq \int_\bbR f_nd\mu_j+\epsilon+\int_\bbR f_nd\mu\leq 4\epsilon,
\end{align*}
concluding the proof of the first implication. 

For the converse implication assume that there is $j_0\in J$ and $c>0$ so that $\sup_{j\geq j_0}\|\mu_j\|_\kappa<c$. Recall the notation $\omega_k(t)=(1+t^2)^{\kappa+1}$. Let $T>0$ and $m=\sup_{t\in[-T,T]}\omega_k(t)$. Then for $j\geq j_0$ we have
\[
\mu_j([-T,T])\leq m\int_{[-T,T]}\frac{d\mu_j}{\omega_k}\leq mc.
\]
For $\nu$ either $\mu_j$ or $\mu$, let $B_\nu\in L^\infty([-T,T])^*$ be defined by $B_\nu(f)=\int_{[-T,T]}fd\nu$. By assumption, there is $C>0$ so that 
\begin{align}\label{eq:portmanteau2}
\sup_{j\geq j_0}\|B_{\mu_j}\|_{L^\infty([-T,T])^*}<C\quad \|B_{\mu}\|_{L^\infty([-T,T])^*}<C.
\end{align}

Let $\chi_A$ denote the indicator function of a set $A$ and set 
\[
L=\left\{\chi_{(a,b)}\mid\ -T\leq a<b\leq b\wedge \mu(\{a\})=\mu(\{b\})=0 \right\}.
\]
Then for all $f\in\spann L$ we have 
\begin{align}\label{eq:portmanteau3}
\lim\limits_{j\in J}\int f d\mu_j=\int f d\mu.
\end{align}
The uniform bounds \eqref{eq:portmanteau2} imply that \eqref{eq:portmanteau3} holds on $X=\operatorname{cl}_{L^\infty}(\spann L)$. It remains to note that $C_c((-T,T))\subset X$.
\end{proof}
%
%

\newpage

\bibliographystyle{amsplain}

\providecommand{\bysame}{\leavevmode\hbox to3em{\hrulefill}\thinspace}
\providecommand{\MR}{\relax\ifhmode\unskip\space\fi MR }
\providecommand{\MRhref}[2]{%
  \href{http://www.ams.org/mathscinet-getitem?mr=#1}{#2}
}
\providecommand{\href}[2]{#2}

\end{document}